\documentclass[a4paper,10pt,reqno]{amsart}
\usepackage{amsmath}
\usepackage{cases}
\usepackage{amsfonts}
\usepackage[colorlinks,linkcolor=blue,citecolor=blue]{hyperref}
\usepackage{latexsym, amssymb, amsmath, amsthm, bbm}
\usepackage[all]{xy}
\usepackage{pgfplots}
\usepackage{enumerate}
\usepackage{mathrsfs}

\DeclareSymbolFont{EulerExtension}{U}{euex}{m}{n}
\DeclareMathSymbol{\euintop}{\mathop} {EulerExtension}{"52}
\DeclareMathSymbol{\euointop}{\mathop} {EulerExtension}{"48}

\allowdisplaybreaks[4]

\def \k{\Bbbk}
\def \id{\operatorname{id}}
\def \A{\mathcal{A}}
\def \B{\mathcal{B}}
\def \C{\mathcal{C}}
\def \D{\mathcal{D}}
\def \E{\mathcal{E}}
\def \F{\mathcal{F}}
\def \G{\mathcal{G}}
\def \X{\mathcal{X}}
\def \Y{\mathcal{Y}}
\def \Z{\mathcal{Z}}

\def \span{\operatorname{span}}

\usepackage{mathtools}

\numberwithin{equation}{section}
\theoremstyle{plain}
 \newtheorem{theorem}{Theorem}[section]

\newtheorem{lemma}[theorem]{Lemma}
\newtheorem{proposition}[theorem]{Proposition}
\newtheorem{corollary}[theorem]{Corollary}
\newtheorem{definition}[theorem]{Definition}
\newtheorem{example}[theorem]{Example}
\newtheorem{remark}[theorem]{Remark}

\begin{document}
\title[Quiver approach to quasi-quantum groups]{A quiver approach to quasi-quantum groups with the Chevalley property}
\thanks{}
\author[J. Yu]{Jing Yu}
\address{School of Mathematical Sciences, University of Science and Technology of China, Hefei 230026, China}
\email{yujing46@ustc.edu.cn}

\maketitle
\begin{abstract}
In this paper, we develop a quiver approach to coquasi-Hopf algebras with the dual Chevalley property. We introduce a modified generalized path coalgebra $\Bbbk(\mathrm{Q},\mathcal{S})$ associated with a given quiver $\mathrm{Q}$ and a collection of simple coalgebras $\mathcal{S}=\{C_i\mid i\in \mathrm{Q}_0\}$ indexed by the vertices of $\mathrm{Q}$, such that its link quiver coincides with $\mathrm{Q}$. We prove that such a coalgebra admits a graded coquasi-Hopf algebra structure with the dual Chevalley property if and only if $\mathrm{Q}$ is a generalized Hopf quiver and $\bigoplus_{i\in \mathrm{Q}_0}C_i$ forms a cosemisimple coquasi-Hopf algebra. Moreover, we provide a classification of these coquasi-Hopf algebra structures. We then study the link-indecomposable components of a coquasi-Hopf algebra with the dual Chevalley property, and give the generalized dual Gabriel's theorem for such coquasi-Hopf algebras. As an application, we apply the quiver method to classify finite integral tensor categories with the Chevalley property of finite representation type. We also give structural characterizations of coradically graded coquasi-Hopf algebras of tame corepresentation type. Furthermore, we investigate finite braided integral tensor categories with the Chevalley property via the quiver approach.
\end{abstract}
\date{}

\tableofcontents
\section{Introduction}
The concept of quasi-Hopf algebras was introduced by Drinfeld \cite{Dri90} in connection with the Knizhnik-Zamolodchikov system of equations from conformal field theory. Subsequently, Majid \cite{Maj92} introduced the dual notion, termed coquasi-Hopf algebras, also known as Majid algebras \cite{SS93}. In accordance with Drinfeld's philosophy of quantum groups \cite{Dri87}, we can understand both of these mutually dual algebraic structures in the framework of quasi-quantum groups; they are obtained from Hopf algebras by weakening the coassociativity and associativity axioms, respectively.
They have found important applications in various areas of mathematics and physics, for example low-dimensional topology, conformal field theory, representation theory, and number theory. See, for instance, \cite{DPR90, DNR25, FGR22, GR25, RT91}.

Recall that, under Tannaka-Krein duality \cite{Maj95}, finite-dimensional quasi-quantum groups are deeply related to finite tensor categories. Within this framework, Etingof and Ostrik \cite{EO04} demonstrated that a finite tensor category in which all objects have integer Frobenius-Perron dimensions (termed a finite integral tensor category) can be realized from a finite-dimensional quasi-quantum group. Thus, classifying these tensor categories is converted into classifying the associated quasi-quantum groups.

The easiest class of non-semisimple tensor categories is formed by those with the Chevalley property, i.e., categories in which the tensor product of any two simple objects remains semisimple. The name derives from Chevalley's theorem \cite{Che54}, which established this property for finite-dimensional representations of groups and Lie algebras over a field of characteristic $0$. Moreover, the finite-dimensional module categories of the quantized enveloping algebras $U_q(\mathfrak{g})$, the quantized Borel subalgebras $U_q^{\geq 0}(\mathfrak{g})$, and the quantized coordinate rings $\mathcal{O}_q(G)$  over the field $\mathbb{C}$ of complex numbers, satisfy the Chevalley property, where $q \in \mathbb{C}^{\times}$ is not a root of unity; see \cite{Jos94}, as well as \cite[Remark 1.6]{HQW26}. Tensor categories with the Chevalley property arising from some big quantum groups at roots of unity have also been studied recently in \cite{HQWZ26}.

The classification of finite integral tensor categories with the Chevalley property provides a key motivation for constructing new types of (co)quasi-Hopf algebras with the (dual) Chevalley property. Here, by the (dual) Chevalley property we mean that for a quasi-Hopf algebra the Jacobson radical is a quasi-Hopf ideal, and for a coquasi-Hopf algebra the coradical is a coquasi-Hopf subalgebra.
In the non-semisimple setting, pointed tensor categories, where all simple objects are invertible \cite{EGNO15}, provide an important class of tensor categories with the Chevalley property, and the current classification results for such categories are mainly focused on finite integral pointed tensor categories. Etingof and Gelaki pioneered this line of research with a series of papers \cite{EG04, EG05, EG06, Gel05}, introducing a method to construct genuine elementary quasi-Hopf algebras from known finite-dimensional pointed Hopf algebras. They also gave an explicit classification of such algebras over cyclic groups of prime order. Here, ``genuine" means the quasi-Hopf algebra is not twist equivalent to a Hopf algebra. Building on this, Angiono \cite{Ang10} extended the construction and obtained a complete classification over cyclic groups whose order is not divisible by $2, 3, 5, 7$. Recently, Huang, Liu, Yang, Ye, and Zhang completed the classification of finite-dimensional coradically graded pointed coquasi-Hopf algebras over abelian groups in their series of works \cite{HLYY20, HLYY26,HYZ18}, thereby generalizing the classification results for pointed Hopf algebras by Andruskiewitsch and Schneider in \cite{AS10}.

Meanwhile, building on Gabriel's foundational contributions \cite{Gab72}, quivers and their representations emerged as a central framework in the representation theory of associative algebras. Within the context of Hopf algebras and quantum groups, related notions such as Hopf quivers \cite{CR02} and covering quivers \cite{GS98} were introduced and further developed in works including \cite{CHZ06, MT19,VZ04}. In \cite{Hua09,Hua12}, Huang constructed pointed coquasi-Hopf algebras on Hopf quivers exhaustively using the projective representation theory of groups combined with a proper deformation theory. From the perspective of the quiver setting for pointed coquasi-Hopf algebras, it turns out that nothing new emerges beyond the Hopf quivers. Subsequently, Huang, together with Liu and Ye, classified pointed coquasi-Hopf algebras of finite corepresentation type \cite{HLY11} and graded elementary tame quasi-Hopf algebras \cite{HLY15} by using the quiver method. As a result, several classification results for finite integral pointed tensor categories according to their representation type were obtained. They \cite{HL12} also studied coquasitriangular pointed coquasi-Hopf algebras and braided pointed tensor categories via the quiver approach. Therefore, quiver method proves to be a useful tool in the classification of finite integral pointed tensor categories.

The study of non-pointed, non-semisimple integral tensor categories with the Chevalley property originates primarily from the pioneering work of Andruskiewitsch, Etingof, and Gelaki \cite{AEG01}, who proved that any finite-dimensional triangular Hopf algebra over $\mathbb{C}$ with the Chevalley property can be obtained by twisting a triangular Hopf algebra whose $R$-matrix has rank at most $2$. Subsequently, Etingof and Gelaki \cite{EG03} showed that every finite symmetric tensor category over an algebraically closed field of characteristic $0$ has the Chevalley property, and that any such category is equivalent to the category of finite-dimensional representations of a triangular Hopf algebra with $R$-matrix of rank $\leq 2$. More recently, they \cite{EG21a, EG21b} classified finite-dimensional triangular quasi-Hopf algebras with the Chevalley property over an algebraically closed field $\k$ of characteristic $p > 0$, thereby also classifying finite symmetric integral tensor categories with the Chevalley property over $\k$. In addition, finite-dimensional Hopf algebras with the dual Chevalley property and a given coradical have been classified in a case-by-case manner using the approach of Andruskiewitsch and Schneider \cite{AS10}. An example is the classification of Hopf algebras whose coradical is a fixed $16$-dimensional semisimple Hopf algebra $H_{b:1}$ \cite{ZGH21}. However, more general classification results for integral tensor categories with the Chevalley property remain relatively scarce.

In our previous work, we investigated the link quivers of Hopf algebras with the dual Chevalley property and classified the corepresentation types of such algebras \cite{YLL24, YL24, YL26}. We showed that the link quivers of these Hopf algebras are completely determined by the Grothendieck ring structure of their coradical and the number of arrows with ending vertex $\k1.$
Naturally, one may ask whether the same holds for the link quivers of coquasi-Hopf algebras with the dual Chevalley property. On the other hand, the inverse problem of the above question is also worth considering. Since one can construct pointed (coquasi-)Hopf algebra structures on Hopf quivers, perhaps one can also construct (coquasi-)Hopf algebra structures with the dual Chevalley property on certain quivers. However, it remains unknown how to do so from a given quiver.
Inspired by Geiss, Leclerc, and Schr\"{o}er's work (\cite{GLS17}) on translating the representations of a class of Iwanaga-Gorenstein algebras defined via quivers with relations associated with symmetrizable Cartan matrices into representations of modulated graphs, we wish to place a simple coalgebra $C_i$ at each vertex $i\in \mathrm{Q}_0$ of a quiver $\mathrm{Q} = (\mathrm{Q}_0, \mathrm{Q}_1)$, and at each arrow $\alpha: i \rightarrow j$ of the quiver, place a simple $C_j$-$C_i$-bicomodule. Then, by means of the structure of a cotensor coalgebra, we can construct a coradically graded coalgebra. Clearly, if some $C_i$ is not one-dimensional, then this coalgebra is not pointed.
From now on, we only need to determine for which quivers we can construct a (coquasi-)Hopf algebra with the dual Chevalley property such that the given quiver is its link quiver, and to understand how such structures can be constructed. This essentially corresponds to the yet undefined concept of generalized Hopf quivers.

In this paper, we establish a quiver approach to coquasi-Hopf algebras with the dual Chevalley property. We first prove that, as in the Hopf algebra case, the link quiver of such a coquasi-Hopf algebra $H$ is completely determined by the Grothendieck ring structure of its coradical $H_0$ and the number of arrows with ending vertex $\k1$; see Proposition \ref{prop:alphaijk}. Moreover, the number of arrows with ending vertex $\k1$ is related to the comodule decomposition of the left-left Yetter-Drinfeld module $(H_1/H_0)^{co H}$; see Remark \ref{rm:Hopf;pointed}. Abstracting this property, we introduce the definition of a generalized Hopf quiver; see Definition \ref{def:generalizedHopfquiver}. We then introduce a modified generalized path coalgebra (see Subsection \ref{subsection4.1}) associated with a given quiver and a collection of simple coalgebras ordered by the vertices of the quiver, whose link quiver is the initial given quiver. We prove that such a coalgebra admits a graded coquasi-Hopf algebra structure with the dual Chevalley property if and only if the given quiver is a generalized Hopf quiver and the direct sum of such a class of simple coalgebras forms a cosemisimple coquasi-Hopf algebra, thereby addressing the question posed above. We also give a classification of such coquasi-Hopf algebra structures; see Theorem \ref{thm:generalizedHopfquiver}. As a byproduct, we construct a class of tensor categories with the Chevalley property; see Proposition \ref{prop:k(Q,S)antipodebijective}. It should be noted that, compared to pointed tensor categories, this class of tensor categories with the Chevalley property covers a broader range of types. It can include some special fusion categories as tensor subcategories, for example, the module category of the fixed-point subalgebra
$V^G$ of a holomorphic vertex operator algebra $V$; see Example \ref{ex:DwGquiver}.
In addition, we study the link-indecomposable components of a coquasi-Hopf algebra with the dual Chevalley property (see Proposition \ref{prop:HcHd}), and give the generalized dual Gabriel's theorem for such coquasi-Hopf algebras (see Theorem \ref{thm:Gabriel}).

Similar to how the quiver method helps in the study of pointed integral tensor categories, we use the quiver method that we established for coquasi-Hopf algebras with the dual Chevalley property to investigate the classification of finite integral tensor categories with the Chevalley property. We first establish the relationship between the corepresentation types of coquasi-Hopf algebras with the dual Chevalley property and their link quivers (see Proposition \ref{prop:corep=quiver}), and then classify finite-dimensional coquasi-Hopf algebras with the dual Chevalley property of finite corepresentation type using the quiver method (see Proposition \ref{prop:finite}), thereby obtaining classification results for the corresponding tensor categories (see Theorem \ref{thm:finite}). We also provide some characterizations of graded coquasi-Hopf algebras with the dual Chevalley property of tame corepresentation type; see Propositions \ref{prop:tame1-dim} and \ref{prop:tameC}. Furthermore, we study coquasitriangular coquasi-Hopf algebras with the dual Chevalley property, and thereby investigate finite braided integral tensor categories with the Chevalley property via the quiver approaches; see Theorem \ref{thm:gabrielcoquasi} and Proposition \ref{prop:braidedtensorcategory}.

The organization of this paper is as follows: In Section \ref{section2}, we recall the definitions of tensor categories, coquasi-Hopf algebras, Majid bimodules, and Yetter-Drinfeld modules, and present some of their basic properties. In Section \ref{section3}, we characterize the link quiver of coquasi-Hopf algebras with the dual Chevalley property, and furthermore investigate the decomposition of these coquasi-Hopf algebras into link-indecomposable components. In Section \ref{section4}, we first construct a class of quiver-related coalgebras: a modified generalized path coalgebra associated with a given quiver and a collection of simple coalgebras indexed by the vertices of this quiver. Then we give the definition of a generalized Hopf quiver, and prove that such a coalgebra admits a graded coquasi-Hopf algebra structure with the dual Chevalley property if and only if the given quiver is a generalized Hopf quiver and the direct sum of the simple subcoalgebras forms a cosemisimple coquasi-Hopf algebra. Moreover, we give the generalized dual Gabriel's theorem for coquasi-Hopf algebras with the dual Chevalley property. Section \ref{section5} is devoted to applying the quiver method to the classification of finite integral tensor categories with the Chevalley property. Specifically, we classify finite integral tensor categories with the Chevalley property of finite representation type, and provide structural characterizations of graded coquasi-Hopf algebras with the dual Chevalley property of tame corepresentation type. Furthermore, we investigate finite braided integral tensor categories with the Chevalley property via the quiver approaches.

\section{Tensor categories and coquasi-Hopf algebras}\label{section2}
Throughout this paper $\Bbbk$ denotes an \textit{algebraically closed field} unless specified otherwise and all spaces are over $\Bbbk$. In this section, we introduce some essential concepts and present foundational results. Further details can be found in \cite{EGNO15} for tensor categories and in \cite{BCPV19} for (co)quasi-Hopf algebras.
\subsection{Unital based rings and tensor categories}
Let $\mathbb{Z}_+$ be the set of non-negative integers. Some relevant concepts are recalled as follows.
\begin{definition}\emph{(}\cite[Chapter 3]{EGNO15}\emph{)}\label{def:based}
Let $A$ be a unital associative ring that is free as a $\mathbb{Z}$-module.
\begin{itemize}
  \item[(1)]A $\mathbb{Z}_+$-basis of $A$ is a basis $B=\{b_{i}\}_{i\in I}$ such that $b_ib_j=\sum_{k\in I}\alpha_{ij}^kb_k$, where $\alpha_{ij}^k\in\mathbb{Z}_+$.
  \item[(2)]A ring $A$ endowed with a fixed $\mathbb{Z}_+$-basis $\{b_i\}_{i\in I}$ is called a unital based ring if the following conditions are satisfied:
  \begin{itemize}
  \item[(i)]$1$ is a basis element.
  \item[(ii)]
There exists an involution $i \mapsto i^*$ of $I$ such that the induced map
$$a=\sum\limits_{i\in I}a_ib_i \mapsto a^*=\sum\limits_{i\in I}a_ib_{i^*},\;\; a_i\in \mathbb{Z}$$ is an anti-involution of $A$, and
$$\tau(b_ib_j)=\left\{
\begin{aligned}
1,~~~  \text{if} ~~~ i=j^*, \\
0,~~~  \text{if} ~~~ i\neq j^*.
\end{aligned}
\right.$$
  \end{itemize}
\item[(3)]A fusion ring is a unital based ring of finite rank.
\item[(4)]A ring $A$ endowed with a fixed $\mathbb{Z}_+$-basis $\{b_i\}_{i\in I}$ is said to be a unital transitive $\mathbb{Z}_+$-ring if $1$ belongs to the basis and, for any $i,j\in I$ there exist $k,l\in I$ such that $b_ib_k$ and $b_l b_i$ contain $b_j$ with a nonzero coefficient.
\end{itemize}
\end{definition}
\begin{remark}\rm\label{rm:1once}
Let $A$ be a unital based ring with a $\mathbb{Z}_+$-basis $\{b_i\}_{i\in I}$.
By Definition \ref{def:based} (2) (ii), for each basis element $b_i\in B$, the identity element $1$ appears exactly once in the expansion of the product $b_ib_i^*$. Moreover, \cite[Exercise 3.3.2]{EGNO15} implies $A$ is transitive.
\end{remark}
\begin{lemma}\label{lem:cijkinvariant}
For all $i,j,k\in I$ in a unital based ring, we have $\alpha_{ij}^k=\alpha_{kj^*}^i.$
\end{lemma}
\begin{proof}
Since $*$ is an anti-involution, it follows that $$b_{j^*}b_{i^*}=(b_ib_j)^*=\sum\limits_{k\in I}\alpha_{ij}^kb_{k^*}.$$
On the other hand, from the product structure we also have $$b_{j^*}b_{i^*}=\sum\limits_{k\in I}\alpha_{j^* i^*}^{k^*}b_{k^*}.$$ Comparing coefficients yields $\alpha_{ij}^k=\alpha_{j^*i^*}^{k^*}$ for all $i,j,k\in I.$
Applying \cite[Proposition 3.1.6]{EGNO15} then gives $\alpha_{ij}^k=\alpha_{j^*i^*}^{k^*}=\alpha_{kj^*}^i.$
\end{proof}

Let $A$ be a transitive unital $\mathbb{Z}_+$-ring of finite rank. For each basis element $b_i\in B$, let $\operatorname{FPdim}(b_i)$ denote the maximal non-negative eigenvalue of the matrix of left multiplication by $b_i$. As this matrix has non-negative entries, the Frobenius-Perron theorem (\cite[Theorem 3.2.1]{EGNO15}) guarantees the existence of such an eigenvalue. The function $\operatorname{FPdim}$ is called the \textit{Frobenius-Perron dimension}. We remark that $\operatorname{FPdim}$ is the unique character of $A$ which takes non-negative values on $B$. The reader is referred to \cite[Chapter 3]{EGNO15} for details.

According to \cite[Proposition 3.3.6 (2)]{EGNO15}, there exists a nonzero element $R \in A \otimes_{\mathbb{Z}} \mathbb{C}$ which is unique up to scalar multiples, such that $X  R = \operatorname{FPdim}(X) R$ for all $X \in A$. This element also satisfies $R  Y = \operatorname{FPdim}(Y) R$ for all $Y \in A$ and is called a \textit{regular element} of $A$. If $A$ is a fusion ring, then by \cite[Proposition 3.3.11]{EGNO15}, the element $R = \sum_{i \in I} \operatorname{FPdim}(b_i) b_i$ has this property.
We state the following lemma.
\begin{lemma}\label{lem:FPeq}
Let $A$ be a fusion ring with a $\mathbb{Z}_+$-basis $B=\{b_i\}_{i\in I}$. For $i,j\in I$, write the product as $b_ib_j=\sum_{k\in I}\alpha_{ij}^kb_k$ with $\alpha_{ij}^k\in\mathbb{Z}_+$. Then for any $k,t\in I$, we have
$$\operatorname{FPdim}(b_k) \operatorname{FPdim}(b_t)=\sum\limits_{i\in I} \operatorname{FPdim}(b_i)\alpha_{i,k}^t.$$
\end{lemma}
\begin{proof}
It follows from \cite[Propositions 3.3.6 (2) and 3.3.11]{EGNO15} that $R = \sum_{i \in I} \operatorname{FPdim}(b_i) b_i$ is a regular element
and $\operatorname{FPdim}(b_k)R=R b_k.$
This means that
$$ \operatorname{FPdim}(b_k) (\sum\limits_{i \in I} \operatorname{FPdim}(b_i) b_i)=(\sum\limits_{i \in I} \operatorname{FPdim}(b_i) b_i)b_k
 =\sum\limits_{i\in I} \sum\limits_{t\in I}\operatorname{FPdim}(b_i)\alpha_{ik}^t b_t,$$
which follows that $\operatorname{FPdim}(b_k) \operatorname{FPdim}(b_t)=\sum_{i\in I} \operatorname{FPdim}(b_i)\alpha_{i,k}^t.$
\end{proof}

Recall that a locally finite $\Bbbk$-linear abelian rigid monoidal category $\mathscr{C}$ is called a \textit{tensor category} (\cite[Definition 4.1.1]{EGNO15}) if the bifunctor $\otimes:\mathscr{C}\times \mathscr{C}\rightarrow \mathscr{C}$ is bilinear on morphisms and the unit object $\mathbbm{1}$ is simple. Here by a \textit{locally finite category} (\cite[Definition 1.8.1]{EGNO15}) we mean one whose morphism spaces are finite-dimensional and in which every object has finite length. Moreover, a tensor category with enough projectives and finitely many isomorphism classes of simple objects is called \textit{finite} \cite[Definition 1.8.6]{EGNO15}.

For a tensor category $\mathscr{C}$, we denote its Grothendieck ring by $\operatorname{Gr}(\mathscr{C})$. From \cite[Proposition 4.5.4]{EGNO15}, $\operatorname{Gr}(\mathscr{C})$ is a transitive unital $\mathbb{Z}_+$-ring, whose $\mathbb{Z}_+$-basis $\mathcal{V}$ is given by the set of isomorphism classes of simple objects in $\mathscr{C}.$
If in addition, $\mathscr{C}$ is semisimple, then according to \cite[Proposition 4.9.1]{EGNO15}, $\operatorname{Gr}(\mathscr{C})$ is a unital based ring. Let $X$ be an object in a finite tensor category $\mathscr{C}$, then its \textit{Frobenius-Perron dimension} $\operatorname{FPdim}(X)$ is defined to be the Frobenius-Perron dimension of its class $[X]$ in $\operatorname{Gr}(\mathscr{C})$.

Next, we introduce the specific type of tensor category that constitutes the main subject of this paper.
\begin{definition}\emph{(}\cite[Definition 4.12.3]{EGNO15}\emph{)}
A tensor category $\mathscr{C}$ is said to have the Chevalley property, if the tensor product of every two simple objects of $\mathscr{C}$ is semisimple.
\end{definition}

Let $\mathscr{C}$ be a tensor category with the Chevalley property and $\mathscr{C}_0$ be the full subcategory of semisimple objects of $\mathscr{C}$. In this setting, $\mathscr{C}_0$ is a semisimple tesnor category and $\operatorname{Gr}(\mathscr{C}_{0})$ is isomorphic to $\operatorname{Gr}(\mathscr{C})$ as unital based rings. For a general tensor category $\mathscr{C}$, however, the Grothendieck ring $\operatorname{Gr}(\mathscr{C})$ is not necessarily a unital based ring. For instance, consider the $2$-dimensional irreducible representation $X$ of the group $S_3$ over a field of characteristic $2$: the tensor product $X \otimes X^*$ contains two copies of the trivial representation in its composition series (see \cite[Example 3.1.9(V)]{EGNO15}). This is contrary to Remark \ref{rm:1once}.

\subsection{Coquasi-Hopf algebras}
Coquasi-Hopf algebras, also known as Majid algebras or dual quasi-Hopf algebras, are precisely the dual notion of Drinfeld's quasi-Hopf algebras \cite{Dri90}. We explicitly recall only the definitions about coquasi-Hopf algebras (see \cite[Section 2.4]{Maj95}). The corresponding definitions for quasi-Hopf algebras can be obtained in a dual manner.
\begin{definition}
A coquasi-bialgebra is a coalgebra $(H,\Delta,\varepsilon)$ over $\k$ equipped with a
compatible quasi-algebra structure. More explicitly, there exist two coalgebra homomorphisms $$m: H \otimes H \rightarrow H, \ a
\otimes b \mapsto ab \;\;\text{and}\;\; \mu: \k \rightarrow H,\ \lambda \mapsto \lambda
1_H,$$ a convolution-invertible map $\Phi: H^{\otimes 3} \rightarrow \k$
called the reassociator such that for all $a,b,c,d \in
H$ the following equalities hold:
\begin{eqnarray}
&a_1(b_1c_1)\Phi(a_2,b_2,c_2)=\Phi(a_1,b_1,c_1)(a_2b_2)c_2,\label{eq:associativity}\\
&1_H a=a=a1_H, \label{eq:1h=h1=h}\\
&\Phi(a_1,b_1,c_1d_1)\Phi(a_2b_2,c_2,d_2) =\Phi(b_1,c_1,d_1)\Phi(a_1,b_2c_2,d_2)\Phi(a_2,b_3,c_3), \label{eq:penta}\\
&\Phi(a,1_H,b)=\varepsilon(a)\varepsilon(b).\label{eq:Phi(1)=e}
\end{eqnarray}
A coquasi-bialgebra $H$ is called a coquasi-Hopf algebra if it admits a coquasi-antipode, i.e., a coalgebra antimorphism $S : H \to H$ together with two maps $\alpha, \beta : H \rightarrow \Bbbk$ satisfying for any $a \in H$:
\begin{eqnarray}
&S(a_1)\alpha(a_2)a_3=\alpha(a)1_H,\;\; a_1\beta(a_2)S(a_3)=\beta(a)1_H, \label{eq:alpha,beta} \\
&\Phi(a_1,S(a_3),a_5)\beta(a_2)\alpha(a_4)
=\Phi^{-1}(S(a_1),a_3,S(a_5)) \alpha(a_2)\beta(a_4)=\varepsilon(a). \label{eq:Phialphabeta=e}
 \end{eqnarray}
Here and below we adopt the shorthand $\Phi(a,b,c):=\Phi(a\otimes b\otimes c)$ and use the Sweedler's notation $\Delta(a)=a_1 \otimes
a_2$ for the coproduct; the result of the $n$-fold iterated coproduct applied to
$a$ is denoted $a_1 \otimes a_2 \otimes \cdots \otimes
a_{n+1}$. Moreover, a coquasi-antipode $(S,\alpha,\beta)$ is said to be bijective if $S$ is bijective.
\end{definition}

Note that (\ref{eq:penta}) and (\ref{eq:Phi(1)=e}) imply the identities
\begin{eqnarray}
\Phi(1_H,a,b)=\Phi(a,b,1_H)=\varepsilon(a)\varepsilon(b),\;\;\forall a,b\in H.
\end{eqnarray}
The relations (\ref{eq:alpha,beta}) and (\ref{eq:Phialphabeta=e}) also imply $S(1_H)=1_H$ and $\alpha(1_H)\beta(1_H)=1$.
It is worth pointing out that the coquasi-antipode is unique up to a convolution invertible element $U\in H^*$. In other words, if $(S^\prime,\alpha^\prime,\beta^\prime)$ is another triple satisfying (\ref{eq:alpha,beta}) and (\ref{eq:Phialphabeta=e}), then $$S^\prime=U*S*U^{-1},\;\; \alpha^\prime=U*\alpha,\;\; \beta^\prime=\beta*U^{-1}.$$
Moreover, coquasi-Hopf algebras are a direct generalization of Hopf algebras. Indeed, when the reassociator $\Phi$ is trivial (i.e., $\Phi(a,b,c) = \varepsilon(a)\varepsilon(b)\varepsilon(c)$ for all $a,b,c \in H$), and $\alpha=\varepsilon=\beta$, a coquasi-Hopf algebra reduces to an ordinary Hopf algebra.

Let $H$ be a coquasi-Hopf algebra over $\k.$ We now recall the monoidal structure of the category $\mathscr{M}{}^H$ of finite-dimensional right $H$-comodules. The tensor product of two $H$-comodules $U$ and $V$ is taken as the vector space tensor product $U\otimes V$, equipped with the coaction $\rho^R_{U\otimes V}(u\otimes v)= u_0\otimes v_0\otimes u_{1}v_{1}$, where $\rho^R_U(u)=u_0\otimes u_{1}$ and $\rho^R_V(v)=v_0\otimes v_{1}$ in Sweedler's notation. The associativity constraints are given by:
\begin{eqnarray*}
a_{U,V,W}: (U\otimes V)\otimes W&\rightarrow& U\otimes (V\otimes W)\\
(u\otimes v)\otimes w&\mapsto& u_0\otimes (v_0\otimes w_0)\Phi(u_1, v_1, w_1)
\end{eqnarray*}
for all $u\in U, v\in V, w\in W$ and $U, V, W\in\mathscr{M}{}^H$. The unit object is the trivial comodule $\k.$
The left dual of an $H$-comodule $V\in\mathscr{M}{}^H$ is the dual vector space $V^*$ equipped with the coaction $\rho^R_{V^*}(f)= f_0\otimes f_1$, defined such that  $$\langle f_0,v \rangle f_1=\langle f, v_0 \rangle S(v_1)$$ for all $v\in V, f\in V^*$. If $S$ is invertible, then the right dual of $V$ can be defined analogously by replacing $S$ with $S^{-1}$ in the formula above. Consequently, the corepresentation category of a coquasi-Hopf algebra with bijective coquasi-antipode forms a tensor category.

Combining \cite[Theorem 1.1]{BIR09} and \cite[Theorem 4.10]{BC03} yields the following lemma.
\begin{lemma}\label{lem:cosstensor}
Let $H$ be a cosemisimple coquasi-Hopf algebra with a coquasi-antipode $(S,\alpha,\beta)$. Then $S$ is bijective. Furthermore, the category $\mathscr{M}{}^H$ of finite-dimensional right $H$-comodules is a semisimple tensor category.
\end{lemma}

As a corollary, we obtain the following result, which extends \cite[Theorem 3.3]{Lar71} to the coquasi-Hopf algebra case.
\begin{corollary}\label{coro:S2C=C}
Let $H$ be a cosemisimple coquasi-Hopf algebra over $\k$. Then for any simple subcoalgebra $C$ of $H$, we have $S^2(C)=C.$
\end{corollary}
\begin{proof}
Because $S^2$ is a coalgebra map, $S^2(C)$ is itself a subcoalgebra of $H$. Since $C$ is simple, it follows that either $S^2(C)\cap C=0$ or $C\subseteq S^2(C)$. Using Lemma \ref{lem:cosstensor}, the category $\mathscr{M}{}^H$ of finite-dimensional right $H$-comodules is a semisimple tensor category. It follows from \cite[Proposition 4.8.1]{EGNO15} that $S^2(C)=C.$
\end{proof}

\begin{example}\rm\label{ex:group}
Let $G$ be a group. The group algebra $\k G$ is naturally a cosemisimple Hopf algebra with
$\Delta(g)=g\otimes g,\varepsilon(g)=1,$ and $ S(g)=g^{-1}$
for all $g\in G$. Let $\omega$ be a normalized $3$-cocycle on $G$, i.e.,
\begin{eqnarray*}
\omega(ef,g,h)\omega(e,f,gh)=\omega(e,f,g)\omega(e,fg,h)\omega(f,g,h),\;\;\omega(f,1,g)=1
\end{eqnarray*}
for all $e,f,g,h \in G$.
By linear extension, $\omega : (\k G)^{\otimes 3} \rightarrow \k$ becomes a convolution-invertible map. Define two linear functions  $\alpha,\beta : \k G \rightarrow  \k$ by $ \alpha(g):=\varepsilon(g) $ and $ \beta(g):=\frac{1}{\omega(g,g^{-1},g)} $
 for all $g\in G$. Then $kG$ together with these $\omega, \ \alpha$ and $\beta$ forms a cosemisimple coquasi-Hopf algebra, denoted by $(\k G,\omega)$ (see \cite[Example 2.2]{HLYY20}). The Gr-category $\operatorname{Vec}^{\omega}_G$ (see \cite[Example 2.3.8]{EGNO15}) is precisely the category of comodules over $(\k G,\omega)$.
\end{example}
\begin{example}\rm\label{ex:DwG}
For a finite group G and a $3$-cocycle $\omega$ on $G$, Dijkgraaf, Pasquier and Roche \cite{DPR90} constructed a quasi-Hopf algebra $D^\omega(G)$, known as the \textit{twisted Drinfeld double} of $G$, in connection with the topological field theories. They also conjectured that the representation category of the fixed-point subalgebra $V^G$ of a holomorphic vertex operator algebra $V$ is equivalent to the module category of some twisted Drinfeld double $D^\omega (G)$. This conjecture was proved in \cite{DNR25}. It should be noted that $D^\omega(G)$ is semisimple if and only if the characteristic of $\k$ does not divide the order of $G$ (see \cite[Section 8.6]{BCPV19}). For convenience, we assume that the characteristic of $\k$ is $0$. By taking the dual $(D^\omega(G))^*$, one obtains a cosemisimple coquasi-Hopf algebra. More explicitly, let $\k^G$ denote the linear dual of $\k G$; it is spanned by the orthogonal idempotents $p_x$ defined by $p_x(y)=\delta_{x,y}$ for all $x,y\in G$. Then $(D^\omega(G))^*$ is $\k^G\otimes \k G$ as a vector space, with unit $\sum_{g\in G}p_g \# 1$, and with multiplication, comultiplication, counit, reassociator, and coquasi-antipode given by
\begin{eqnarray*}
&(p_g \otimes x)(p_h \otimes y) = \delta_{g,h} \sigma(g,x,y) p_g \otimes xy,\\
&\Delta(p_g \otimes x) = \sum\limits_{ ab = g} \tau(a,b,x)p_a \otimes  bxb^{-1} \otimes p_b \otimes x,\;\;\varepsilon(p_g \otimes x) = \delta_{g,1},\\
&\Phi(p_g \otimes x, p_h \otimes y, p_f \otimes z) = \delta_{g,1} \delta_{h,1} \delta_{f,1} \omega(x,y,z),\\
&S(p_g\otimes  x)=\sigma (g^{-1},  gxg^{-1},  gx^{-1}g^{-1})^{-1}\tau(g^{-1}, g, x)^{-1}p_{g^{-1}}\otimes gx^{-1}g^{-1},\\
&\alpha(p_g\otimes x)=\delta_{g,1},\;\;\beta(p_g\otimes x)=\delta_{g,1}\frac{1}{\omega(x,x^{-1},x)},
\end{eqnarray*}
for all $g,h,f,x,y,z\in G$, where
\begin{eqnarray*}
\tau(g,h,f) &:=& \frac{\omega(g,h,f)  \omega(ghfh^{-1}g^{-1},g,h)}{\omega(g,hfh^{-1},h)},\\
\sigma(g,h,f) &:=& \frac{\omega(g,h,f) \, \omega(ghg^{-1},gfg^{-1},g)}{\omega(ghg^{-1},g,f)}.
\end{eqnarray*}
Consider, for example, the quaternion group $Q_8=\langle x, y\mid x^4=1, y^2=x^2, yx=x^{-1}y\rangle$. Define
$\omega(x^{a_1} y^{b_1}, x^{a_2} y^{b_2}, x^{a_3} y^{b_3}) = (-1)^{a_1 b_2  b_3}$
for all $a_1,a_2,a_3 \in \{0,1,2,3\}$ and $b_1,b_2,b_3 \in \{0,1\}$. Then $(D^\omega(Q_8))^*$ forms a non-pointed cosemisimple coquasi-Hopf algebra.
\end{example}

According to \cite[Corollary 4.4]{BC03}, a coquasi-antipode of a finite-dimensional coquasi-Hopf algebra is bijective.
Therefore, the representation category of any finite-dimensional quasi-Hopf algebra, as well as the corepresentation category of any finite-dimensional coquasi-Hopf algebra, naturally forms a finite tensor category. Conversely, under mild assumptions, any finite tensor category arises in precisely this manner.
\begin{lemma}\emph{(}\cite[Proposition 2.6]{EO04}\emph{)}\label{lem:integertensor}
A finite tensor category $\mathscr{C}$ is integral (i.e., the Frobenius-Perron dimensions of all its objects are integers) if and only if $\mathscr{C}$ is tensor equivalent to the representation category of a finite-dimensional quasi-Hopf algebra.
\end{lemma}
Inspired by \cite{AEG01}, we say that a coquasi-Hopf algebra $H$ (not necessarily finite-dimensional) has the \textit{dual Chevalley property} if its coradical $H_0$ is a coquasi-Hopf subalgebra. On the other hand, a quasi-Hopf algebra $H$ (not necessarily finite-dimensional) is said to have the \textit{Chevalley property} if its Jacobson radical $\operatorname{Rad}(H)$ is a quasi-Hopf ideal of $H$.

Obviously, for a finite-dimensional coquasi-Hopf algebra
$H$, it has the dual Chevalley property if and only if $\mathscr{M}^H$ is a tensor category with the Chevalley property, if and only if
$H^*$ is a quasi-Hopf algebra with the Chevalley property, if and only if ${}_{H^*}\mathscr{M}$ has the Chevalley property.

Since finite-dimensional coquasi-Hopf algebras are dual to quasi-Hopf algebras, it follows from Lemma \ref{lem:integertensor} that a finite tensor category $\mathscr{C}$ is tensor equivalent to the corepresentation category of a finite-dimensional coquasi-Hopf algebra if and only if the Frobenius-Perron dimensions of objects in $\mathscr{C}$ are integers. Consequently, classifying finite tensor categories with the Chevalley property whose objects have integer Frobenius-Perron dimensions is equivalent to classifying finite-dimensional (co)quasi-Hopf algebras with the (dual) Chevalley property.

\begin{example}\rm
Recall that a coquasi-Hopf algebra is called \textit{pointed} if its underlying coalgebra is pointed. In the pointed case, every simple comodule is one-dimensional. It follows immediately that every pointed coquasi-Hopf algebra satisfies the dual Chevalley property.
\end{example}
\begin{example}\rm\label{example:coquasiChevalley}
Bosonization is well established in the context of Hopf algebras. This construction was later extended to coquasi-Hopf algebras \cite{AP13}.
Hence, coquasi-Hopf algebras with the dual Chevalley property can be constructed via the bosonization of cosemisimple coquasi-Hopf algebras and Nichols algebras. Moreover, given a coquasi-Hopf algebra $H$ with the dual Chevalley property, let $\{H_n\}_{n\geq0}$ denote its coradical filtration and define the associated coradically graded coalgebra as
$\operatorname{gr}H=\bigoplus_{n\geq0}H_n/H_{n-1}$, where $H_{-1}=0.$ Then naturally
$\operatorname{gr}H$ inherits from $H$ a graded coquasi-Hopf algebra structure (\cite[Proposition 6.2]{AP13}) that also satisfies the dual Chevalley property. The
corresponding graded associator $\operatorname{gr}\Phi$ fulfills
$\operatorname{gr}\Phi(\bar{a},\bar{b},\bar{c})=0$ for all homogeneous elements
$\bar{a},\bar{b},\bar{c} \in \operatorname{gr} H$ unless they all lie in $H_0.$
Similar conditions hold for $\operatorname{gr}\alpha$ and $\operatorname{gr}\beta.$ According to \cite[Corollary 6.4]{AP13}, there exists a Hopf algebra $R$ in the category ${}^{H_0}_{H_0}\mathcal{YD}$ of Yetter-Drinfeld module over $H_0$  (see Definition \ref{def:YDmod}) such that $\operatorname{gr} H\cong R\# H_0.$
\end{example}
\begin{remark}\rm
Let $H$ be a non-cosemisimple coquasi-Hopf algebra with the dual Chevalley property. Then all simple comodules are non-projective. Otherwise, if $V$ were a simple projective comodule, then since $\mathbbm{1}$ is contained in $V^* \otimes V$ (by Remark \ref{rm:1once}), \cite[Proposition 4.2.12]{EGNO15} would imply that $\mathbbm{1}$ is also projective, contradicting \cite[Corollary 4.2.13]{EGNO15}.
\end{remark}

\subsection{Majid bimodules and the fundamental theorem}\label{subsection2.3}
In this subsection, we let $H$ be a coquasi-bialgebra over $\k$ with reassociator $\Phi$. We first present the definitions of left (resp. right) coquasi-Hopf $H$-bicomodules and of $H$-Majid bimodules, and then recall the fundamental theorem for coquasi-Hopf algebras.

It is well-known that the category ${}^{H}\mathscr{M}^{H}$ of finite-dimensional $H$-bicomodules forms a monoidal category. Its associativity constraints are given for all $U, V, W\in {}^{H}\mathscr{M}^{H}$ and $u\in U, v\in V, w\in W$, by
$$a_{U, V, W}((u\otimes v)\otimes w):=\Phi^{-1}(u_{-1}, v_{-1}, w_{-1})u_0\otimes (v_0\otimes w_0) \Phi (u_1, v_1, w_1).$$
Moreover, $((H, \Delta, \Delta), m, \mu)$ is an algebra in ${}^{H}\mathscr{M}^{H}$. Consequently, it makes sense to define a \textit{left coquasi-Hopf $H$-bicomodules} \cite[Remark 2.3]{BC03} as being a left $H$-module in ${}^{H}\mathscr{M}^{H}$. We may therefore denote by
${}^{H}_H\mathscr{M}^{H}$ the category of left $H$-module in ${}^{H}\mathscr{M}^{H}$. Note, however, that $H$ itself is not an algebra in $\mathscr{M}^{H}$ or in ${}^{H}\mathscr{M}$.
In a similar manner, one can define the categories ${}^{H}\mathscr{M}^{H}_H$ and ${}_H^{H}\mathscr{M}^{H}_H$. The category ${}^{H}\mathscr{M}^{H}_H$ is known as the category of \textit{right coquasi-Hopf $H$-bicomodules} (\cite[Remark 2.3]{BC03}), while ${}_H^{H}\mathscr{M}^{H}_H$ is referred to as the category of \textit{$H$-Majid bimodules} (\cite[Definition 3.2]{Hua09}).

More specifically, we give the following definition.
\begin{definition}\emph{(}\cite[Definition 3.2]{Hua09}\emph{)}
Let $H$ be a coquasi-bialgebra over $\k$ with reassociator $\Phi.$ A
linear space $M$ is called an $H$-Majid bimodule, if $M$ is an
$H$-bicomodule with structure maps $(\rho_M^{L},\rho_M^{R}),$ and there
exist two $H$-bicomodule morphisms
$$p_L: H \otimes M \rightarrow M, \ h \otimes m \mapsto h\cdot m, \quad
p_R: M \otimes H \rightarrow M, \ m \otimes h \mapsto m\cdot h$$ such that
for all $g, h \in H, m \in M,$ the following equalities hold:
\begin{eqnarray}
&1_H\cdot m=m=m\cdot 1_H, \\
&g_1\cdot(h_1\cdot m_0)\Phi(g_2,h_2,m_1)=\Phi(g_1,h_1,m_{-1})(g_2h_2)\cdot m_0, \\
&m_0\cdot (g_1h_1)\Phi(m_1,g_2,h_2)=\Phi(m_{-1},g_1,h_1)(m_0\cdot g_2)\cdot h_2, \\
&g_1\cdot (m_0\cdot h_1)\Phi(g_2,m_1,h_2)=\Phi(g_1,m_{-1},h_1)(g_2\cdot m_0)\cdot h_2.
\end{eqnarray}
\end{definition}
Note that when $\Phi$ is trivial, the notions of left (resp. right) coquasi-Hopf $H$-bicomodules and $H$-Majid bimodules respectively revert to the classical left (resp. right) $H$-Hopf modules and $H$-Hopf bimodules.

Analogous to the fundamental theorem for Hopf modules in Hopf algebra theory (\cite[Theorem 1.9.4]{Mon93}), certain coquasi-bialgebras with specific properties also admit a corresponding fundamental theorem.
Ardizzoni and Pavarin \cite{AP12} proved that, for a coquasi-bialgebra $H$, the fundamental theorem is equivalent to the existence of a suitable map $s: H \rightarrow H$, called a preantipode (\cite[Definition 3.6]{AP12}). Here, a \textit{preantipode} is a $\k$-linear map $s:H\rightarrow H$ such that for all $a\in H$:
\begin{eqnarray*}
&s(a_2)_1 \otimes a_1s(a_2)_2 = s(a)\otimes 1_H,\\
&s(a_1)_1a_2 \otimes s(a_1)_2 = 1_H \otimes s(a),\\
&\Phi(a_1 \otimes s(a_2)\otimes a_3) = \varepsilon(a).
\end{eqnarray*}
Moreover, by setting $s:=\beta * S* \alpha$, it follows from \cite[Theorem 3.10]{AP12} that every coquasi-Hopf algebra $(H, m, \mu, \Delta, \varepsilon, \Phi, S, \alpha, \beta)$ admits a preantipode; hence the fundamental theorem holds for all coquasi-Hopf algebras.

According to \cite[Theorems 2.7, 3.9 and 3.10]{AP12}, we present the following proposition, which generalizes \cite[Theorem 1.9.4]{Mon93}.
\begin{proposition}\label{prop:Hopfmod}
Let $H$ be a coquasi-Hopf algebra over $\k$ and let $(M, \rho_M^L, \rho_M^R,p_L)\in {}_H^{H}\mathscr{M}^{H}$. Define the map
 $\phi:H\otimes {}^{co H}M \rightarrow M$ by setting $\phi(h\otimes m):=h\cdot m$, where ${}^{co H}M= \{m\in M\mid \rho_M^L(m)=1_H\otimes m\}$. Then $\phi$ is bijective.
\end{proposition}

If, in addition, $ H $ is a coquasi-Hopf algebra with a bijective antipode $S$, then we obtain a ``mixed'' version of the fundamental theorem.
\begin{proposition}\label{prop:mixedHopfmod}
Let $H$ be a coquasi-Hopf algebra over $\k$ with a bijective coquasi-antipode $(S,\alpha,\beta)$, and let $(M, \rho_M^L, \rho_M^R,p_L)\in {}_H^{H}\mathscr{M}^{H}$. Denote the space of right coinvariants by $M^{co H}= \{m\in M\mid \rho_M^R(m)=m\otimes 1_H\}$. Then the map
\begin{eqnarray*}
\psi: H\otimes M^{co H} \rightarrow M,\;\; h\otimes m\mapsto h\cdot m
\end{eqnarray*}
is bijective.
\end{proposition}
\begin{proof}
Note that $(H^{op}, m^{op}, \mu, \Delta, \varepsilon, (\Phi^{-1})^{321}, S^{-1}, \alpha S^{-1}, \beta S^{-1} )$ remains a coquasi-Hopf algebra, where $(\Phi^{-1})^{321}(f,g,h)=\Phi^{-1}(h,g,f)$, and it forms an algebra in the category ${}^H\mathscr{M}^H$.
It follows that any left $H$-module in ${}^H\mathscr{M}^H$ corresponds canonically to a right $H^{op}$-module in ${}^H\mathscr{M}^H$. This correspondence is realized concretely by equipping $M$ with the right $H^{op}$-action defined as $m\circ h:=h\cdot m.$ Proceeding analogously to the proof of Proposition \ref{prop:Hopfmod}, one establishes that the map $$\phi: M^{co H} \otimes H^{op}\rightarrow M,\;\;m\otimes h\mapsto m\circ h$$ is bijective, from which the bijectivity of $\psi$ follows.
\end{proof}
Lemma \ref{lem:cosstensor} and Proposition \ref{prop:mixedHopfmod} lead directly to the following corollary.
\begin{corollary}\label{coro:H0fundamental}
Let $H$ be a cosemisimple coquasi-Hopf algebra over $\k$ and $(M, \rho_M^L, \rho_M^R,p_L)\in {}_H^{H}\mathscr{M}^{H}$. Then the map
\begin{eqnarray*}
\psi: H\otimes M^{co H} \rightarrow M,\;\; h\otimes m\mapsto h\cdot m
\end{eqnarray*}
is bijective.
\end{corollary}

\begin{example}\rm\label{example:H1/H0}
Let $H$ be a coquasi-Hopf algebra over $\k$ with the dual Chevalley property, and $\{H_n\}_{n\geq0}$ be its coradical filtration. As shown in Example \ref{example:coquasiChevalley}, $\operatorname{gr}(H)$ carries the structure of a graded coquasi-Hopf algebra. Let $\pi: H_1 \longrightarrow H_1/H_0$ denote the quotient map. For any $\overline{h}\in H_1/H_0$, define
\begin{eqnarray}\label{comodulestructure}
\rho_M^L(\bar{h})=(\id\otimes\pi)\Delta(h),\;\;\rho_M^R(\bar{h})=(\pi\otimes\id)\Delta(h).
\end{eqnarray}
It is straightforward to verify that $(H_1/H_0, \rho_M^L, \rho_M^R)$ is an $H_0$-bicomodule. Furthermore, we endow $H_1/H_0$ with an $H_0$-bimodule structure in ${}^{H_0}\mathscr{M}^{H_0}$ via
$$ h\cdot \overline{x}:=\overline{hx},\; \;  \overline{x}\cdot h:=\overline{xh}, \;\; \text{for all }h\in H_0,x\in H_1.$$
With this bicomodule structure and the bimodule structure defined above, $H_1/H_0$ is an $H_0$-Majid bimodule.
\end{example}

\subsection{Yetter-Drinfeld modules}\label{subsection:2.4}
Let $H$ be a coquasi-bialgebra over $\k$.
Ardizzoni and Pavarin \cite{AP13} established a connection between ${}_H^{H}\mathscr{M}^{H}_H$ and the category of Yetter-Drinfeld modules over $H$. We briefly recall their results.

\begin{definition}\emph{(}\cite[Definition 3.1]{Bal09}\emph{)} \label{def:YDmod}
A left-left Yetter-Drinfeld module over $H$ is a triple $\left( V,\rho^L_V ,\vartriangleright \right),$ where
\begin{itemize}
\item $(V,\rho^L_V )$ is an object in ${^{H}\mathscr{M}}$;
\item $\vartriangleright :H\otimes V\rightarrow V$ is a $\Bbbk$-linear map satisfying the following conditions for all $h,l \in H$ and $v \in V$:
\begin{eqnarray}
&\left( hl\right) \vartriangleright v=
\frac{\Phi \left(
h_{2}, \left( l_{2}\vartriangleright v_{0}\right) _{-1},
l_{3}\right)}{\Phi \left( h_{1}, l_{1}, v_{-1}\right)\Phi \left( (h_{3}\vartriangleright (l_{2}\vartriangleright
v_{0})_{0}\right) _{-1}, h_{4}, l_{4})}
\left(
h_{3}\vartriangleright \left( l_{2}\vartriangleright v_{0}\right)
_{0}\right) _{0}
,& \label{eq:YD1}\\
&1_{H}\vartriangleright v=v, &\label{eq:YD2}\\
&\left( h_{1}\vartriangleright v\right) _{-1}h_{2}\otimes \left(
h_{1}\vartriangleright v\right) _{0}=h_{1}v_{-1}\otimes \left(
h_{2}\vartriangleright v_{0}\right) . \label{eq:YDmod3}&
\end{eqnarray}
\end{itemize}
\end{definition}
The category ${}_H^H\mathcal{YD}$ of Yetter-Drinfeld modules over $H$, together with the tensor product $\otimes$ and the unit object $\Bbbk$, is a monoidal category.
Similarly, we may define the notions of right-right, right-left and
left-right Yetter-Drinfeld modules over $H$ by
$$\mathcal{YD}{}_H^H={}_{H^{op,cop}}^{H^{op,cop}}\mathcal{YD}, {}^H\mathcal{YD}{}_H={}_{H^{op}}^{H^{op}}\mathcal{YD}, {}_H\mathcal{YD}{}^H={}_{H^{cop}}^{H^{cop}}\mathcal{YD}.$$

\begin{remark}\rm\label{rm:YD2def}
Let $H$ be a coquasi-Hopf algebra over $\k$ with a bijective coquasi-antipode. Inspired by \cite[Section 2]{Bal09}, we introduce the linear maps
\begin{eqnarray*}
p(h,g) = \Phi^{-1}(g, h_1, S(h_3) \beta(h_2)), \;\;
q(h,g) = \Phi(g, h_3,S^{-1}(h_1) \alpha (S^{-1}(h_2))),
\end{eqnarray*}
for all $g,h \in H$. Observe that $(H^{op,cop},m^{op},\mu,\Delta^{cop}, \Phi^{321},S,\beta,\alpha)$ is also a coquasi-Hopf algebra, where $\Phi^{321}(f,g,h)=\Phi(h,g,f)$.
It follows from \cite[Remark 3.2 (2)]{Bal09} that
axiom (\ref{eq:YDmod3}) is equivalent to
\begin{eqnarray}\label{eq:YDmod2}
 (h\vartriangleright v)_{-1} \otimes (h\vartriangleright v)_0=p(h_4,(h_3 \vartriangleright v_0)_{-1})   q(S(h_6),h_1v_{-2})    (h_2v_{-1})S(h_5)\otimes (h_3 \vartriangleright v_0)_0.
\end{eqnarray}
\end{remark}

\begin{example}\rm\label{ex:coquaitriangular}
Recall that a \textit{coquasitriangular} coquasi-Hopf algebra (\cite[Definition 11.14]{BCPV19}) is a pair $(H, R)$, where $H$ is a coquasi-Hopf algebra and $R:H\otimes H\rightarrow \k$ is a convolution invertible map such that
\begin{eqnarray}
&R(a_1, b_1)a_2b_2=b_1a_1 R(a_2, b_2),&\label{eq:CQT1}\\
&R(ab, c)=\Phi(c_1,a_1,b_1)R(a_2,c_2)\Phi^{-1}(a_3,c_3,b_2)R(b_3,c_4)\Phi(a_4,b_4,c_5),&\label{eq:CQT2}\\
&R(a,bc)=\Phi^{-1}(b_1,c_1,a_1)R(a_2,c_2)\Phi(b_2,a_3,c_3)R(a_4,b_3)\Phi^{-1}(a_5,b_4,c_4).&\label{eq:CQT3}
\end{eqnarray}
Clearly, $\mathscr{M}^H$ forms a braided monoidal category; see \cite[Definition 8.1.1]{EGNO15} for the definition.
Moreover, let $M$ be any left comodule over a coquasitriangular coquasi-Hopf algebra $(H,R)$. We claim that $M$ admits a left-left Yetter-Drinfeld module structure via
$$ h\vartriangleright m:=R(m_{-1}, h)m_0.$$
Indeed, by (\ref{eq:CQT3}), we have
\begin{eqnarray*}
(hl)\vartriangleright m=R(m_{-1},hl)m_0
=\frac{\Phi(h_2,R(m_{-4},l_2)m_{-3},l_3)}{\Phi(h_1,l_1,m_{-5})\Phi(R(m_{-2},h_3)m_{-1},h_4,l_4)}m_0.
\end{eqnarray*}
This establishes (\ref{eq:YD1}). Conditions (\ref{eq:YD2}) and (\ref{eq:YDmod3}) in the definition of a Yetter-Drinfeld module follow, respectively, from the fact that $R$ is convolution invertible and from equation (\ref{eq:CQT1}) satisfied by $R$.
\end{example}

If, in addition, $H$ admits a preantipode, then, as established in \cite[Proposition 3.8 and Lemma 4.4]{AP13}, the category $({}_H^{H}\mathscr{M}^{H}_H, \otimes_H, H)$ of $H$-Majid bimodules is monoidal equivalent to the category $({}^H_H\mathcal{YD}, \otimes,\k)$ of Yetter-Drinfeld modules over $H$. As every coquasi-Hopf algebra admits a preantipode (\cite[Theorem 3.10]{AP12}), the following proposition holds.
\begin{proposition}\label{prop:Majidmod=YD}
Let $H$ be a coquasi-Hopf algebra over $\k$. Then the functors
\begin{eqnarray*}
G:\; {}_H^{H}\mathscr{M}^{H}_H \rightarrow  {}^H_H\mathcal{YD},\;\;
M\mapsto M^{coH}
\end{eqnarray*}
and
\begin{eqnarray*}
F:{}^H_H\mathcal{YD}\rightarrow{}_H^{H}\mathscr{M}^{H}_H,\;\;N\mapsto N\otimes H
\end{eqnarray*}
are mutually inverse monoidal equivalences. More explicitly, $M^{coH}\in {}^H_H\mathcal{YD}$ and $N\otimes H\in {}_H^{H}\mathscr{M}^{H}_H$ are equipped with the following structures: for all $m\in M^{coH}, n\in N, h,l\in H,$
\begin{eqnarray*}
&\rho^L _{M^{coH}}\left( m\right) :=\rho^L _{M}(m),& \\
&h\vartriangleright m:=\beta(h_3)\alpha(h_6)
\Phi(h_1m_{-1},S(h_5),h_7)
(h_2m_0)S(h_4),&\\
&l\cdot
(n\otimes h):=\Phi (l_{1}, n_{-1}, h_{1})\Phi
^{-1}((l_{2}\vartriangleright n_{0})_{-1}, l_{3},
h_{2})(l_{2}\vartriangleright n_{0})_{0}\otimes l_{4}h_{3}, & \\
&(n\otimes
h)\cdot l:=\Phi ^{-1}(n_{-1}, h_{1}, l_{1})n_{0}\otimes
h_{2}l_{2}, &  \\
&\rho _{N\otimes H}^{L}\left( n\otimes h\right) :=n_{-1}h_{1}\otimes
(n_{0}\otimes h_{2}), & \\
&\rho _{N\otimes H}^{R}\left( n\otimes h\right) :=(n\otimes h_{1})\otimes
h_{2}.&
\end{eqnarray*}
\end{proposition}
\begin{proof}
The structures of $M^{coH}\in {}^H_H\mathcal{YD}$ and $N\otimes H\in {}_H^{H}\mathscr{M}^{H}_H$ follow directly from \cite[Lemmas 3.6 and 3.7]{AP13}. Since every coquasi-Hopf algebra admits a preantipode (\cite[Theorem 3.10]{AP12}),  it follows from \cite[Proposition 3.8 and Lemma 4.4]{AP13} that $F$ and $G$ are mutually inverse monoidal equivalences.
\end{proof}

\section{Link quivers and link-indecomposable components}\label{section3}
The aim of this section is to investigate the properties of the link quiver of a coquasi-Hopf algebra with the dual Chevalley property, and to characterize the structure of their link-indecomposable components.
\subsection{Preliminaries on link quivers}
In this subsection, we fix a coalgebra $(H,\Delta,\varepsilon)$ over $\Bbbk$. Let $\{H_n\}_{n\geq0}$ denote its coradical filtration, and let $\mathcal{S}$ be the set of all simple subcoalgebras of $H.$

The link quiver of a coalgebra was originally introduced by Montgomery using the wedge of simple subcoalgebras of $H$ (see \cite[Definition 1.1]{Mon95}). A modified definition was later given in \cite[Definition 4.1]{CHZ06}. We begin by recalling the notion of the link quiver.
\begin{definition}\emph{(}\cite[Definition 4.1]{CHZ06}\emph{)}
Let $H$ be a coalgebra over $\Bbbk$. Denote the set of all simple subcoalgebras of $H$ by $\mathcal{S}$. The link quiver $\mathrm{Q}(H)$ of $H$ is defined as follows: its vertex set is $\mathcal{S}$; for any simple subcoalgebra $C, D\in \mathcal{S}$ with $\dim_{\Bbbk}(C)=r^2, \dim_{\Bbbk}(D)=s^2$, there are exactly $\frac{1}{rs}\dim_{\Bbbk}((C\wedge D)/(C+D))$ arrows from $D$ to $C$.
\end{definition}

\begin{remark}\rm\label{rm:link=ext}
Let $H$ be a coalgebra, and let $\{V_i \mid i \in I\}$ be a complete set of representatives of the isomorphism classes of simple right $H$-comodules. Recall that the \textit{Ext quiver} of $H$ is defined as an oriented graph with vertex set indexed by $I$, where for any $i, j \in I$, the number of arrows from $i$ to $j$ is given by $\dim_k \operatorname{Ext}^1_H(V_i, V_j)$. As shown in \cite[Theorem 2.1 and Corollary 4.1]{CHZ06}, the link quiver of $H$ coincides with its Ext quiver.
\end{remark}

Next, we turn to an alternative description of the link quiver of a coalgebra. Before proceeding further, we recall the definitions of multiplicative and primitive matrices, which serve as the non-pointed analogues of group-like and skew-primitive elements, respectively. For details, see \cite{Li22}.

Recall that a square matrix $\G=(g_{ij})_{r\times r}$ over $H$ is said to be \textit{multiplicative}, if $$\Delta(g_{ij})=\sum\limits_{t=1}^r g_{it}\otimes g_{tj}$$ and $\varepsilon(g_{ij})=\delta_{i, j}$ for any $1\leq i,j \leq r$, where $\delta_{i, j}$ denotes the Kronecker notation. A multiplicative matrix $\C$ is said to be \textit{basic}, if its entries are linearly independent.
Note that a basic multiplicative matrix $\C$ spans a simple subcoalgebra of $H$, and every simple subcoalgebra of $H$ is spanned by some basic multiplicative matrix.
According to \cite[Lemma 2.4]{Li22}, the basic multiplicative matrix of the simple coalgebra $C$ would be unique up to the similarity relation.
\begin{remark}\emph{(}\cite[Remark 3.4]{YLL24}\emph{)}\rm
Let $\mathcal{M}$ be a set of representatives for the similarity classes of basic multiplicative matrices over $H$. Then $\mathcal{S}$ is in canonical bijection with $\mathcal{M}$ via the map sending a simple subcoalgebra to its basic multiplicative matrix. Moreover, $\mathcal{S} = \{ \operatorname{span}(\C) \mid \C \in \mathcal{M} \}$, where $\operatorname{span}(\C)$ is the subspace of $H_0$ spanned by the entries of $\C$.
\end{remark}

Let $\mathcal{C}=(c_{ij})_{r\times r}$ and $\mathcal{D}=(d_{ij})_{s\times s}$ be basic multiplicative matrices over $H$.
 A matrix $\mathcal{X}=(x_{ij})_{r\times s}$ over $H$ is said to be $(\mathcal{C}, \mathcal{D})$-\textit{primitive}, if $$\Delta(x_{ij})=\sum\limits_{k=1}^r c_{ik}\otimes x_{kj}+\sum\limits_{t=1}^s x_{it}\otimes d_{tj}$$ holds for all $1\leq i \leq r, 1\leq j\leq s$. A primitive matrix $\X$ is \textit{non-trivial}, if there exists some entry of $\X$ which does not belong to the coradical $H_0$.

Let $\Y=\left(y_{ij}\right)_{m\times n}$ be any matrix with entries in $H_1.$ We denote by $\overline{\Y}$ the matrix $\left(\overline{y_{ij}}\right)_{m\times n}$, where $\overline{y_{ij}}=y_{ij}+H_0\in H_1/H_0$. The subspace of $H_1/H_0$ spanned by the entries of $\overline{\Y}$ is written as $\operatorname{span}(\overline{\Y})$. Take $C, D\in\mathcal{S}$ with basic multiplicative matrices $\C$ and $\D$ respectively. By \cite[Lemma 2.4]{YLL24}, for any non-trivial $(\C, \D)$-primitive matrix $\X = (x_{ij})_{r \times s}$,  $\span(\overline{\X})$ is a simple $C$-$D$-bicomodule, and $\dim_\k(\operatorname{span}(\overline{\X})) = rs$.
According to \cite[Corollary 2.11 and Lemma 2.17]{YLL24}, there exists a family of non-trivial $(\C,\D)$-primitive matrices $\{\X_{\C,\D}^{(\gamma_{\C, \D})}\}_{\gamma_{\C, \D}\in\Gamma_{\C, \D}}$ such that
\begin{eqnarray}\label{eq:complete}
(C\wedge D)/(C+D)\cong \bigoplus\limits_{\gamma_{\C, \D}\in\Gamma_{\C, \D}}\operatorname{span}(\overline{\X_{\C,\D}^{(\gamma_{\C, \D})}}).
\end{eqnarray}
\begin{definition}\emph{(}\cite[Definition 2.12]{YLL24}\emph{)}
Let $\mathcal{C}$ and $\mathcal{D}$ be basic multiplicative matrices over a coalgebra $H$.
A family of non-trivial $(\mathcal{C}, \mathcal{D})$-primitive matrices $\{\X_{\C,\D}^{(\gamma_{\C, \D})}\}_{\gamma_{\C, \D}\in\Gamma_{\C, \D}}$ that satisfies property (\ref{eq:complete}) is said to be complete.
\end{definition}

Using \cite[Corollary 2.18]{YLL24}, the number of arrows from vertex $D$ to vertex $C$ in the link quiver of $H$ can be identified with the cardinal number of a complete family of non-trivial $(\mathcal{C}, \mathcal{D})$-primitive matrices.
Specifically, denote ${}^{\C}{\mathcal{P}}^{\D}$ by the set of a complete family of non-trivial $(\C, \D)$-primitive matrices, and define
$${}^{\C}\mathcal{P}=\bigcup_{\D\in\mathcal{M}}{}^{\C}{\mathcal{P}}^{\D},\;\;
{\mathcal{P}}^{\D}=\bigcup_{\C\in\mathcal{M}}{}^{\C}{\mathcal{P}}^{\D},\;\;
\mathcal{P}=\bigcup_{\C\in\mathcal{M}} {}^{\C}\mathcal{P}.
$$
Then $^{\C}{\mathcal{P}}^{\D}$ may be regarded as the set of arrows from vertex $D$ to vertex $C$, ${\mathcal{P}^{\D}}$ as the set of arrows starting at $D$, and ${^{\C}\mathcal{P}}$ as the set of arrows ending at $C$.

We can summarize the above analysis with the following lemma.
\begin{lemma}\label{lem:linkquiver=S,P}
The link quiver of a coalgebra $H$ can be presented as the directed graph $\mathrm{Q}(H)=(\mathcal{S}, \mathcal{P})$, with vertex set
$\mathcal{S}$ and arrow set $\mathcal{P}.$
\end{lemma}

\subsection{Link quivers of coquasi-Hopf algebras with the dual Chevalley property}
In this subsection, let $(H, m, \mu, \Delta, \varepsilon, \Phi, S, \alpha, \beta)$ be a coquasi-Hopf algebra over $\k$ with the dual Chevalley property. Denote by $\mathcal{S}$ the set of all simple subcoalgebras of $H.$ We aim to present four distinct construction methods for complete families of non-trivial primitive matrices over
$H$. Through the process of these four different constructions, we investigate the cardinal number of such families, which then enables us to characterize the link quiver of $H.$

Inspired by the Kronecker product for matrices, we begin by defining the following operators on matrices over $H$. For any matrix $\A=(a_{ij})_{r\times s}$ and $\B=(b_{ij})_{u\times v}$ over $H$, define $\A\odot \B$ and $\A\odot^\prime \B$ as follows
 $$\A\odot \B=
\left(\begin{array}{ccc}
      a_{11}\B& \cdots &  a_{1s}\B  \\
      \vdots  & \ddots & \vdots  \\
      a_{r1}\B&  \cdots & a_{rs}\B
    \end{array}\right),\;\;
\A\odot^\prime \B=\left(\begin{array}{cccc}
      \A b_{11} &   \cdots &  \A b_{1v} \\
      \vdots &  \ddots & \vdots  \\
      \A b_{u1} &  \cdots & \A b_{uv}
    \end{array}\right).$$

Let $C, D\in\mathcal{S}$ with basic multiplicative matrices $\C$ and $\D$, respectively. Since $H$ has the dual Chevalley property, $\C\odot^\prime \D$ is again a multiplicative matrix over $H_0$. From \cite[Proposition 2.6 (2)]{Li22} that there exists an invertible matrix $L_{\C, \D}$ over $\k$ such that $$L_{\C, \D} \left(\C\odot^\prime \D \right)L_{\C, \D}^{-1}=
\left(\begin{array}{cccc}
       \E_1 & 0 & \cdots & 0  \\
      0 & \E_2 & \cdots & 0  \\
      \vdots & \vdots & \ddots & \vdots  \\
      0 & 0 & \cdots & \E_{u_{\C, \D}}
  \end{array}\right),$$
where $\E_{1}, \E_{2},\cdots, \E_{ u_{(\C, \D)}}\in\mathcal{M}$ are the basic multiplicative matrices corresponding to simple subcoalgebras $E_{1}, E_{2},\cdots, E_{ u_{(\C, \D)}}\in\mathcal{S}$, respectively. Note that a cosemisimple coalgebra $CD$ admits a decomposition into a direct sum of simple subcoalgebras, and $u_{(\C, \D)}$ is precisely the number of such simple subcoalgebras. Hence, in fact, $u_{(\C, \D)}$ does not depend on the choice of basic multiplicative matrices $\C$ and $\D$, nor on the choice of the invertible matrix $L_{\C, \D}$.

Now we can define a multiplication on $\mathbb{Z}\mathcal{S}$ as follows: for $C, D\in \mathcal{S}$,
\begin{eqnarray}\label{eq:multi}
C\cdot D=\sum\limits_{i=1}^{u_{(\C, \D)}} E_i.
\end{eqnarray}
Given a finite-dimensional right $H$-comodule $(M, \rho_M^R)$, recall that the \textit{coefficient coalgebra} $\operatorname{cf}(M)$ of $M$ is the smallest subcoalgebra of $H$ satisfying $\rho^R_M(M)\subseteq M\otimes \operatorname{cf}(M)$. Let $V_C, V_D, V_E$ denote the simple right $H$-comodules whose coefficient coalgebras are $C, D, E$, respectively. Since $(V_C \otimes V_D) \otimes V_E \cong V_C \otimes (V_D \otimes V_E)$, the multiplication defined above is associative. Moreover,
by Lemma \ref{lem:cosstensor}, Corollary \ref{coro:S2C=C}, and the fact that $(V\otimes W)^* \cong W^* \otimes V^*$, we obtain that the map $C\mapsto S(C)$ defines an anti-involution. With the multiplication and anti-involution defined above, we have $\mathbb{Z}\mathcal{S}$ is a unital based ring with $\mathbb{Z}_+$-basis $\mathcal{S}$.

The following lemma is obtained by sending $M$ to $\operatorname{cf}(M)$.
\begin{lemma}\label{lem:ZS=gr}
Let $H$ be a coquasi-Hopf algebra over $\k$ with the dual Chevalley property and $\mathcal{S}$ be the set of all simple subcoalgebras of $H$. Then $\operatorname{Gr}(\mathscr{M}^{H_0})$ is isomorphic to $\mathbb{Z}\mathcal{S}$ as unital based rings.
\end{lemma}

\begin{remark}\rm\label{rm:FP=dim}
The equality (\ref{eq:multi}) in $\mathbb{Z}\mathcal{S}$ implies $$\sqrt{\dim_{\Bbbk}(C)}\sqrt{\dim_{\k}(D)}=\sum\limits_{i=1}^{u_{(\C, \D)}} \sqrt{\dim_{\Bbbk}(E_i)}.$$
If, in addition, $H_0$ is finite-dimensional, then $\mathbb{Z}\mathcal{S}$ is a fusion ring. For any $C\in\mathcal{S}$, we claim that $\operatorname{FPdim}(C)=\sqrt{\dim_{\k}(C)}$. Indeed, the map $C\mapsto \sqrt{\dim_{\k}(C)}$ is a character of $\mathbb{Z}\mathcal{S}$ that takes non-negative values. According to \cite[Proposition 3.3.6 (3)]{EGNO15}, we have $\operatorname{FPdim}(C)=\sqrt{\dim_{\k}(C)}$.
\end{remark}

In what follows, define ${}^1\mathcal{S}:=\{C\in\mathcal{S}\mid \k1+C\neq \k1\wedge C\}.$ For any $D\in{}^1\mathcal{S}$ with basic multiplicative matrix $\D\in\mathcal{M}$, we apply \cite[Corollary 2.11]{YLL24} to fix a complete family $\{\X_{1,\D}^{(\gamma_{1,\D})}\}_{{\gamma_{1,\D}\in\Gamma}_{1,\D}}$ of non-trivial $(1, \D)$-primitive matrices. We then set
\begin{eqnarray}\label{eq:^1P}
{^1\mathcal{P}}:=\bigcup\limits_{D\in {}^1\mathcal{S}} \{\mathcal{X}_{1,\D}^{(\gamma_{1, \mathcal{D}})}\mid \gamma_{1, \mathcal{D}}\in\Gamma_{1, \mathcal{D}}\}.
\end{eqnarray}

In fact, we can show that ${}^1\mathcal{P}$ must be nonempty for a non-cosemisimple coquasi-Hopf algebra $H$ with the dual Chevalley property.
\begin{lemma}\label{lem:1P>0}
Let $H$ be a non-cosemisimple coquasi-Hopf algebra over $\k$ with the dual Chevalley property. Then we have $\mid {}^1\mathcal{P} \mid \geq 1.$
\end{lemma}
\begin{proof}
Since $H$ is non-cosemisimple, it follows that $H_1/H_0\neq0.$ From Example \ref{example:H1/H0}, $H_1/H_0$ admits the structure of an $H_0$-Majid bimodule. Applying Proposition \ref{prop:Hopfmod}, we obtain an isomorphism
\begin{eqnarray*}
\phi: H_0\otimes {}^{co H_0}(H_1/H_0) \cong H_1/H_0,
\end{eqnarray*}
given by $\phi(h\otimes \overline{x}) = \overline{hx}$.
This means that ${}^{co H_0}(H_1/H_0)\neq0.$ According to \cite[Theorem 4.1]{CHZ06}, we can show that $${}^{co H_0}(H_1/H_0)\cong \bigoplus\limits_{C\in\mathcal{S}}(\k1\wedge C)/(\k1+C).$$
The construction of ${}^1\mathcal{P}$ implies $\mid {}^1\mathcal{P} \mid \geq 1.$
\end{proof}

Next, we consider the case where $H$ is non-cosemisimple and construct a complete family of non-trivial primitive matrices over $H$ starting from ${^1\mathcal{P}}$.
Notice that for any non-trivial $(1, \D)$-primitive matrix $\Y\in{^1\mathcal{P}}$ and $\C\in\mathcal{M}$, we have
\begin{eqnarray}\label{equationBY}
\left(\begin{array}{cc}
I&0\\
0&L_{\C, \D}
 \end{array}\right)
\left(\C\odot^\prime
\left(\begin{array}{cc}
1&\Y\\
0&\D
 \end{array}\right)\right)
 \left(\begin{array}{cc}
I&0\\
0&L_{\C, \D}^{-1}
 \end{array}\right)
=\left(\begin{array}{ccccccc}
    \mathcal{C}  & {\mathcal{Y}_{\C}^{ 1}} & {\mathcal{Y}_{\C}^{ 2}} & \cdots & {\mathcal{Y}_{\C}^{ u_{(\mathcal{C}, \mathcal{D})}}}  \\
    0 &   \mathcal{E}_{1} &0&\cdots  &  0 \\
    0& 0&\mathcal{E}_{2}&\cdots &0\\
    \vdots  &\vdots  &\vdots& \ddots & \vdots  \\
    0   & 0 &0  &\cdots& \mathcal{E}_{u_{(\mathcal{C}, \mathcal{D})}}
  \end{array}\right),
\end{eqnarray}
where $L_{\C, \D}$ is an invertible matrix over $\k$ and $\E_1, \E_2, \cdots, \E_{u_{(\C, \D)}}\in \mathcal{M}$. In particular, when $C = \k 1$, we set $L_{1,\D} = I$, where $I$ denotes the identity matrix over $\k$. Then we have $u_{(1,\D)} = 1$ and $\Y_{1}^{1} = \Y$.

We state the following lemma.
\begin{lemma}\label{Lemma:cap=0}
For any $1\leq i\leq u_{\C, \D}$, the matrix $\Y_{\C}^{i}$ appearing in \eqref{equationBY} is a non-trivial $(\C, \E_i)$-primitive matrix. Moreover, the sum $\sum_{i=1}^{ u_{(\C, \D)}}\operatorname{span}(\overline{\Y_{\C}^{ i}})$ is direct.
\end{lemma}
\begin{proof}
Using \cite[Remark 2.5]{Li22}, a direct proof shows that for each $1\leq i\leq u_{\C,\D}$, $\Y_\C^i$ is a $(\C, \E_i)$-primitive matrix. A similar argument as in the proof of Lemma \ref{lem:1P>0} shows that \begin{eqnarray*}
\phi: H_0\otimes {}^{co H_0}(H_1/H_0) \rightarrow H_1/H_0, \;\;h\otimes \overline{x} \mapsto \overline{hx}
\end{eqnarray*}
is an isomorphism. Note that we have $\span(\C)\subseteq H_0$ and $\span(\overline{\Y})\subseteq {}^{co H_0}(H_1/H_0)$. It follows from \cite[Lemma 2.4]{YLL24} that the entries of $\C\odot^\prime \Y$ are linearly independent in $H_1/H_0,$ which means that the set of all column vectors of the matrix
\begin{eqnarray*}
\left(\begin{array}{ccccccc}
{\mathcal{Y}_{\C}^{ 1}} & {\mathcal{Y}_{\C}^{ 2}} & \cdots & {\mathcal{Y}_{\C}^{ u_{(\mathcal{C}, \mathcal{D})}}}
\end{array}\right)
\end{eqnarray*}
is linearly independent over $H_1/H_0.$ Since each column contains an element that does not belong to $H_0$, it follows from \cite[Corollary 2.6]{YLL24} that each $\Y_{\C}^i$ is non-trivial. By an argument similar to that in the proof of \cite[Lemma 3.5]{YLL24}, the sum $\sum_{i=1}^{ u_{(\C, \D)}}\operatorname{span}(\overline{\Y_{\C}^{ i}})$ is direct.
\end{proof}

Using the notation in (\ref{equationBY}), we denote
\begin{eqnarray}\label{^BPY}
^{\C}\mathcal{P}_{\Y}:=\{\Y_{\C}^{ i}\mid 1\leq i\leq u_{(\C, \D)}\},
\end{eqnarray}
\begin{eqnarray}\label{definition:^BP}
^{\C}\mathcal{P}:=\bigcup\limits_{\Y\in{{}^1\mathcal{P}}}{}^{\C}\mathcal{P}_{\Y},\;\;\; \mathcal{P}_{\Y}:=\bigcup\limits_{\C\in \mathcal{M}}{}^{\C}\mathcal{P}_{\Y}.
\end{eqnarray}
Moreover, denote
\begin{eqnarray}\label{definition:P}
\mathcal{P}:=\bigcup\limits_{\C\in \mathcal{M}}{^{\C}\mathcal{P}}=\bigcup\limits_{\Y\in{{}^1\mathcal{P}}}\mathcal{P}_{\Y} .
\end{eqnarray}

\begin{remark}\rm\label{remark:ucd}
We remark that $\bigcup_{\Y\in{{}^1\mathcal{P}}}{}^{1}\mathcal{P}_{\Y}$ coincides with ${}^1\mathcal{P}$ as defined in (\ref{eq:^1P}). Moreover, the elements of the set $^{\C}\mathcal{P}_{\Y}$ depend on the choice of the invertible matrix $L_{\C, \D}$ in (\ref{equationBY}). However, Lemma \ref{Lemma:cap=0} guarantees that its cardinal number $\mid {}^{\C}\mathcal{P}_{\Y}\mid = u_{\C, \D}$ remains invariant under such choices.
\end{remark}

With the notations above, we now demonstrate that the constructed family $\mathcal{P}$ is precisely the desired set of complete primitive matrices.
\begin{proposition}\label{prop:Pcomplete}
\begin{itemize}
\item[(1)] The union $\mathcal{P}=\bigcup_{\Y\in{}^1\mathcal{P}}\mathcal{P}_{\Y}$ is disjoint;
\item[(2)]Denote
${}^{\C}\mathcal{P}^{\D}:=\{\X\in \mathcal{P}\mid  \X \text{ is a non-trivial }(\C, \D)\text{-primitive matrix}\}.$ Then
 ${}^{\C}\mathcal{P}^{\D}$ forms a complete family of non-trivial $(\C, \D)$-primitive matrices. Moreover, we have $H_1/H_0=\bigoplus_{\X\in\mathcal{P}}\operatorname{span}(\overline{\X})$.
\end{itemize}
\end{proposition}
\begin{proof}
\begin{itemize}
\item[(1)]
An analogous argument to the one used in the proof of Lemma \ref{lem:1P>0} shows that the linear map
\begin{eqnarray*}
\phi: H_0\otimes {}^{co H_0}(H_1/H_0) \rightarrow H_1/H_0, \;\;h\otimes \overline{x} \mapsto \overline{hx}
\end{eqnarray*}
is an isomorphism.
In view of \cite[Lemma 2.8, Corollary 2.11 and Lemma 2.17]{YLL24}, one can show that
\begin{eqnarray*}
{}^{co H_0}(H_1/H_0) = \bigoplus_{\Y\in{}^1\mathcal{P}}\operatorname{span}(\overline{\Y}).
\end{eqnarray*}
By an argument similar to that in \cite[Lemma 3.8]{YLL24}, we have the union $\mathcal{P}=\bigcup_{\Y\in{}^1\mathcal{P}}\mathcal{P}_{\Y}$ is disjoint.
\item[(2)]
By Proposition \ref{prop:Hopfmod}, we have
\begin{eqnarray*}
H_1/H_0 &=& \{h\cdot\overline{x} \mid h\in H_0, \overline{x}\in {}^{co H_0}(H_1/H_0)\}\\
&\subseteq&  \sum_{\C\in \mathcal{M}}\sum_{\Y\in{}^1\mathcal{P}}\operatorname{span}(\overline{\C\odot^\prime \Y})\\
&=& \sum_{\X\in\mathcal{P}}\operatorname{span}(\overline{\X}).
\end{eqnarray*}
Consequently, $$H_1/H_0=\sum_{\X\in\mathcal{P}}\operatorname{span}(\overline{\X}).$$
Furthermore, by Lemma \ref{Lemma:cap=0} and $(1)$, we obtain
\begin{eqnarray*}
\sum_{\X\in\mathcal{P}}\operatorname{span}(\overline{\X})&=& \bigoplus\limits_{\Y\in{}^1\mathcal{P}}\bigoplus\limits_{\C\in\mathcal{M}}\bigoplus\limits_{\X\in{}^{\C}\mathcal{P}_{\Y}}\operatorname{span}(\overline{\X}).
\end{eqnarray*}
The collection ${}^{\C}\mathcal{P}^{\D}$ can be shown, by applying \cite[Theorem 4.1 (iii)]{CHZ06} and (\ref{eq:complete}), to form a complete family of non-trivial $(\C,\D)$-primitive matrices.
\end{itemize}
\end{proof}

Similarly, let $\mathcal{S}^1 = \{C \in \mathcal{S} \mid C + \k1 \neq C \land \k1\}$. For any $D \in \mathcal{S}^1$ with associated basic multiplicative matrix $\D \in \mathcal{M}$, by \cite[Corollary 2.11]{YLL24}, we can choose a complete family $\{\X_{\D,1}^{(\gamma_{\D,1})}\}_{\gamma_{\D,1} \in \Gamma_{\D,1}}$ of non-trivial $(\C,1)$-primitive matrices.
Define
\begin{equation}\label{def:P^1prime}
\mathcal{P}^{\prime 1} = \bigcup_{D \in \mathcal{S}^1} \bigl\{ \X_{\D,1}^{(\gamma_{\D,1})} \mid \gamma_{\D,1} \in \Gamma_{\D,1} \bigr\}.
\end{equation}

Using a similar approach, we can construct a complete family of non-trivial primitive matrices, beginning with $\mathcal{P}^{\prime 1}.$
For any non-trivial $(\F, 1)$-primitive matrix $\Z\in{\mathcal{P}^{\prime 1}}$, we have
\begin{eqnarray*}\label{equationBYprime}
\left(\begin{array}{cc}
L_{\C, \F}&0\\
0&I
 \end{array}\right)
\left(\C\odot^\prime
\left(\begin{array}{cc}
\F&\Z\\
0&1
 \end{array}\right)\right)
 \left(\begin{array}{cc}
L_{\C, \F}^{-1}&0\\
0&I
 \end{array}\right)
=\left(\begin{array}{cccccc}
      \G_{1} & 0 & \cdots&0& {\Z_{\C}^{1}} \\
      0&\G_{2}& \cdots&0& {\Z_{\C}^{2}}\\
       \vdots&\vdots & \ddots & \vdots& \vdots  \\
     0 & 0 &\cdots& \G_{u_{(\C, \F)}}&{\Z_{\C}^{u_{(\C, \F)}}}\\
     0&0&\cdots&0&\C
  \end{array}\right),
\end{eqnarray*}
where $I$ is the identity matrix over $\k$ and $\G_1, \G_2, \cdots, \G_{u_{(\C, \F)}}\in \mathcal{M}$.
Denote
\begin{eqnarray}\label{^BPYprime}
\mathcal{P}^{\prime\C}_{\Z}&=&\{\Z_{\C}^{i}\mid 1\leq i\leq u_{(\C, \F)}\},
\end{eqnarray}
\begin{eqnarray}\label{definition:^BPprime}
\mathcal{P}^{\prime \C}&=&\bigcup\limits_{\Z\in{\mathcal{P}^{\prime 1}}}{}\mathcal{P}^{\prime \C}_{\Z},\;\;\; \mathcal{P}^\prime_{\Z}=\bigcup\limits_{\C\in \mathcal{M}}\mathcal{P}^{\prime \C}_{\Z},
\end{eqnarray}
\begin{eqnarray}\label{definition:Pprime}
\mathcal{P}^\prime&=&\bigcup\limits_{\C\in \mathcal{M}}{\mathcal{P}^{\prime \C}}=\bigcup\limits_{\Z\in{\mathcal{P}^1}}\mathcal{P}^\prime_{\Z} .
\end{eqnarray}
According to Corollary \ref{coro:H0fundamental}, there is an isomorphism
$$ H_0 \otimes {(H_1/H_0)^{\operatorname{co}H_0}} \cong H_1/H_0, $$
given by $h \otimes \overline{x} \mapsto h \cdot \overline{x}$ for all $h \in H_0$ and $\overline{x} \in {(H_1/H_0)^{\operatorname{co}H_0}}$. By an argument analogous to that used in Remark \ref{remark:ucd} and Proposition \ref{prop:Pcomplete}, we obtain the following proposition.
\begin{proposition}\label{prop:Pprimecomplete}
\begin{itemize}
\item[(1)]The cardinal number $\mid \mathcal{P}^{\prime\C}_{\Z}\mid = u_{\C, \F}$, where $\mathcal{P}^{\prime\C}_{\Z}$ appears in (\ref{^BPYprime});
\item[(2)]The union $\mathcal{P}^\prime=\bigcup_{\Z\in{\mathcal{P}^1}}\mathcal{P}^\prime_{\Z}$ is disjoint;
\item[(3)]Denote
${}^{\C}\mathcal{P}^{\prime \D}:=\{\X^\prime\in \mathcal{P}^\prime\mid\X^\prime \text{ is a non-trivial }(\C, \D)\text{-primitive matrix}\}.$
Then it is a complete family of non-trivial $(\C, \D)$-primitive matrices, and we have
$H_1/H_0=\bigoplus_{\X^{\prime}\in\mathcal{P}^{\prime}}\span(\overline{\X^{\prime}})$.
\end{itemize}
\end{proposition}
We can show that the cardinal number of ${}^1\mathcal{P}$ and $\mathcal{P}^{\prime 1}$ are the same.
\begin{lemma}\label{lem:P1prime>0}
Let $H$ be a non-cosemisimple coquasi-Hopf algebra over $\k$ with the dual Chevalley property.
\begin{itemize}
\item[(1)]Then for any $D\in {}^1\mathcal{S}$, we have $\mid{^1\mathcal{P}^{\D}}\mid=\mid{}^{K S(\D)^TK^{-1}}\mathcal{P}^{\prime 1}\mid$, and $D\in{}^1\mathcal{S}$ if and only if $S(D)\in\mathcal{S}^1$, where $K S(\D)^TK^{-1}\in\mathcal{M}$ is a basic multiplicative matrix of $S(D)$, with $K\in\mathrm{GL}(\k)$, and $S(\D)=(S(d_{ij}))_{r_{\D}\times r_{\D}}$.
\item[(2)]Moreover, $\mid{^1\mathcal{P}}\mid=\mid\mathcal{P}^{\prime 1}\mid \geq 1$.
\end{itemize}
\end{lemma}
\begin{proof}
\begin{itemize}
\item[(1)]For any non-trivial $(1,\D)$-primitive matrix $\Y\in {}^1\mathcal{P}$, we consider the expression
$$
K S(\D)^TK^{-1} \odot^\prime\left(\begin{array}{cc}
1&\Y\\
0&\D
 \end{array}\right).$$
It follows from Remarks \ref{rm:1once} and \ref{remark:ucd} that the set $${}^{K S(\D)^T K^{-1}}\mathcal{P}_{\Y}\cap (\bigcup_{\C\in\mathcal{M}} {}^{\C}\mathcal{P}^{1})$$
contains a unique element, which we denote by $\Y^\prime$. It should be pointed out that $\Y^\prime$ is a non-trivial $(K S(\D)^TK^{-1}, 1)$-primitive matrix and that $\dim_{\k}(\span(\overline{\Y^\prime}))=\dim_{\k}(\span(\overline{\Y}))$. Similarly, we can also start from a non-trivial $(K S(\D)^TK^{-1}, 1)$-primitive matrix to obtain a non-trivial $(1,\D)$-primitive matrix. Thus, we can prove $D\in{}^1\mathcal{S}$ if and only if $S(D)\in\mathcal{S}^1$. Moreover, the definition of unital based ring guarantees that $C\cdot D$ contains $1$ if and only if $C=S(D)$. Using Proposition \ref{prop:Pcomplete} and (\ref{eq:complete}),
we obtain the following
\begin{eqnarray*}
  &&\dim_{\k}(\k 1\wedge D)/(\k 1+D))  \\
  &=&\sum\limits_{\Y\in {}^1\mathcal{P}^{\D}}\dim_{\k}(\span(\overline{\Y}))\\
  &=&\sum\limits_{\Y\in {}^1\mathcal{P}^{\D}}\dim_{\k}(\span(\overline{\Y^\prime}))\\
  &=&\dim_{\k}(S(D)\wedge \k 1)/(S(D)+\k 1)).
\end{eqnarray*}
It follows that $\mid{^1\mathcal{P}^{\D}}\mid=\mid{}^{K S(\D)^TK^{-1}}\mathcal{P}^{\prime 1}\mid$, and $D\in{}^1\mathcal{S}$ if and only if $S(D)\in\mathcal{S}^1$.
\item[(2)]
The definitions of
${}^1\mathcal{P}$ and $\mathcal{P}^{\prime 1}$ imply that $\mid{^1\mathcal{P}}\mid=\mid\mathcal{P}^{\prime 1}\mid.$ According to Lemma \ref{lem:1P>0}, we have $\mid\mathcal{P}^{\prime 1}\mid\geq 1$.
\end{itemize}
\end{proof}
\begin{remark}\rm\label{rm:4constructions}
Symmetrically, we can also obtain two other distinct methods for constructing complete families of non-trivial primitive matrices by considering
$$
\left(\begin{array}{cc}
1&\Y\\
0&\D
 \end{array}\right)\odot \C\;\; \text{and}\;\;\left(\begin{array}{cc}
\F&\Z\\
0&1
 \end{array}\right)\odot \C,
 $$
 where $\C\in\mathcal{M}, \Y\in{}^1\mathcal{P}$ and $\Z\in\mathcal{P}^{\prime1}.$
It should be noted that although a complete set of non-trivial primitive matrices can be chosen in different ways, its cardinal number is a fixed invariant.
While Lemma \ref{lem:linkquiver=S,P} establishes that the arrows in the link quiver of $H$ can be represented by a complete family of non-trivial primitive matrices, the four distinct methods we have developed for constructing such families allow us to characterize the link quiver of $H.$
These four constructions demonstrate that the link quiver of $H$ is entirely determined by the Grothendieck ring structure of $\mathscr{M}^{H_0}$ and the cardinal number of ${}^1\mathcal{P}$ (or $\mathcal{P}^{\prime 1}$).
\end{remark}
In what follows, for convenience, let $\mathcal{S}=\{C_i\mid i\in I\}$ denote the set of all simple subcoalgebras of $H$. For any $C_i, C_j\in\mathcal{S}$, we write $$C_i\cdot C_j=\sum\limits_{t\in I}\alpha_{ij}^tC_t,$$ where $\alpha_{ij}^t\in\mathbb{Z}_+$.
Furthermore, we set $\mathcal{M}=\{\C_i\mid i\in I\}$ such that each $\C_i\in\mathcal{M}$ is a basic multiplicative matrix of $C_i\in\mathcal{S}$.

With the notations above, we have
\begin{proposition}\label{prop:alphaijk}
Let $H$ be a non-cosemisimple coquasi-Hopf algebra over $\k$ with the dual Chevalley property, and let $\mathrm{Q}(H)=(\mathcal{S}, \mathcal{P})$ be its link quiver. Suppose ${}^1\mathcal{S}=\{C_k\mid k\in J\}$, where $J \subseteq I.$ Then for any $i,j\in I,$ the number of arrows from $C_j$ to $C_i$ is given by $$\mid {}^{\C_i}\mathcal{P}^{\C_j}\mid=\sum_{k\in J}\mid{}^1 \mathcal{P}^{\C_k}\mid \alpha_{ik}^j=\sum_{k\in J}\mid{}^1 \mathcal{P}^{\C_k}\mid\alpha_{ki}^j.$$
\end{proposition}
\begin{proof}
For any non-trivial $(1, \C_k)$-primitive matrix $\Y_{k,m} \in {}^1\mathcal{P}$, where $k \in J$, $1 \leq m \leq \mid{}^1 \mathcal{P}^{\C_k}\mid$, we consider the expression
$$
\C_i \odot^\prime\left(\begin{array}{cc}
1&\Y_{k,m}\\
0&\C_k
 \end{array}\right).$$
According to Remark \ref{remark:ucd} and Proposition \ref{prop:Pcomplete}, we have $$\mid {}^{\C_i}\mathcal{P}^{\C_j}\mid=\sum_{k\in J}\mid{}^1 \mathcal{P}^{\C_k}\mid \alpha_{ik}^j.$$
Symmetrically, by Remark \ref{rm:4constructions}, we can obtain the remaining equality
$$\mid {}^{\C_i}\mathcal{P}^{\C_j}\mid=\sum_{k\in J}\mid{}^1 \mathcal{P}^{\C_k}\mid\alpha_{ki}^j$$
by considering the expression
$$
\left(\begin{array}{cc}
1&\Y_{k,m}\\
0&\C_k
 \end{array}\right)\odot \C_i.
 $$
\end{proof}
\begin{remark}\rm
In the proof of Proposition \ref{prop:alphaijk}, we determine the cardinal number of ${}^{\C_i}\mathcal{P}{}^{\C_j}$ through two distinct constructions. Using Lemma \ref{lem:P1prime>0} (1), the equality $\alpha_{ij}^k=\alpha_{kj^*}^i$ from Lemma \ref{lem:cijkinvariant} and \cite[Proposition 3.1.6]{EGNO15} then demonstrate that the cardinal number resulting from the other two constructions are the same as that from the constructions used in the proof of Proposition \ref{prop:alphaijk}.
\end{remark}

\begin{corollary}\label{coro:cpdspecial}
Let $H$ be a non-cosemisimple coquasi-Hopf algebra over $\k$ with the dual Chevalley property, and let $\mathrm{Q}(H)=(\mathcal{S}, \mathcal{P})$ be its link quiver.
\begin{itemize}
\item[(1)]For any $i,j\in I$, $\mid {}^{\C_i}\mathcal{P}^{\C_j} \mid=\mid {}^{K_jS(\C_j)^TK_j^{-1}}\mathcal{P}^{K_iS(\C_i)^TK_i^{-1}}\mid$,
 where
$K_i, K_j\in GL(\k)$ and $K_iS(\C_i)^TK_i^{-1}, K_jS(\C_j)^TK_j^{-1}\in\mathcal{M}$.
\item[(2)]If all the simple subcoalgebras in ${}^1\mathcal{S}$ are $1$-dimensional, then for any $i\in I$, we have
$\mid{^{\C_i}\mathcal{P}}\mid=\mid{\mathcal{P}^{\C_i}}\mid=\mid{^1\mathcal{P}}\mid;$
\item[(3)]$\mid{}^{\C_i}\mathcal{P}\mid=\mid\mathcal{P}^{\C_i}\mid=1 $ holds for all $i\in I$ if and only if $\mid{}^1\mathcal{P}\mid=1$ and the unique subcoalgebra $C_k\in{}^1\mathcal{S}$ is $1$-dimensional;
\item[(4)]If ${}^1\mathcal{S}=\{C_k\}$, then both $C_k$ and $S(C_k)$ are in the center of $\mathbb{Z}\mathcal{S}$.
\end{itemize}
\end{corollary}
\begin{proof}
\begin{itemize}
\item[(1)]According to \cite[Proposition 3.1.6]{EGNO15}, we have $\alpha_{ki}^j=\alpha_{j^* k}^{i^*}$. Proposition \ref{prop:alphaijk} implies that $$\mid {}^{\C_i}\mathcal{P}^{\C_j} \mid=\mid {}^{K_jS(\C_j)^TK_j^{-1}}\mathcal{P}^{K_iS(\C_i)^TK_i^{-1}}\mid.$$
\item[(2)]For any group-like element $g\in G(H)$, let $C_k=\k g$ be the corresponding $1$-dimensional simple subcoalgebra. Then for any $C_i\in\mathcal{S}$, one can show that $C_i g$ is also a simple subcoalgebra of $H$. It follows from Proposition \ref{prop:alphaijk} that $$\mid {^{\C_i}\mathcal{P}}\mid=\sum_{j\in I} \sum\limits_{k\in J} \mid{}^1 \mathcal{P}^{\C_k}\mid\alpha_{ik}^j=\sum\limits_{k\in J} \mid{}^1 \mathcal{P}^{\C_k}\mid=\mid{^1\mathcal{P}}\mid.$$
Observe that $S(C_k)=\k g^{-1}$ is also a $1$-dimensional simple subcoalgebra, implying that $C_i g^{-1}$ is likewise a simple subcoalgebra of $H$.
By Lemma \ref{lem:cijkinvariant}, we have $\alpha_{jk}^i=\alpha_{ik^*}^j.$ Consequently, $$\mid {\mathcal{P}^{\C_i}}\mid=\sum_{j\in I} \sum\limits_{k\in J} \mid{}^1 \mathcal{P}^{\C_k}\mid\alpha_{jk}^i=\sum_{j\in I} \sum\limits_{k\in J} \mid{}^1 \mathcal{P}^{\C_k}\mid\alpha_{ik^*}^j=\sum\limits_{k\in J} \mid{}^1 \mathcal{P}^{\C_k}\mid=\mid {^{\C_i}\mathcal{P}}\mid.$$
\item[(3)]The ``if" part follows immediately from $(2)$. Conversely, Remark \ref{rm:1once} shows that, for $k\in J$, there is only one $1$ in the summand of $S(C_k)\cdot C_k$ in $\mathbb{Z}\mathcal{S}$. It then follows from Proposition \ref{prop:alphaijk} that $\dim_{\k}(C_k)=1.$
\item[(4)]According to Proposition \ref{prop:alphaijk}, we have $$\mid {}^{\C_i}\mathcal{P}^{\C_j}\mid=\mid{}^1 \mathcal{P}^{\C_k}\mid \alpha_{ik}^j=\mid{}^1 \mathcal{P}^{\C_k}\mid\alpha_{ki}^j.$$It follows that $\alpha_{ik}^t=\alpha_{ki}^t,$
which implies $$C_i\cdot C_k=C_k\cdot C_i$$ for all $i\in I.$ Besides, we have $$S(C_k)\cdot S(C_i)=S(C_i)\cdot S(C_k)$$ for all $i\in I$. Hence, both $C_k$ and $S(C_k)$ lie in the center of $\mathbb{Z}\mathcal{S}$.
\end{itemize}
\end{proof}

\begin{remark}\label{rm:Hopf;pointed}\rm
\begin{itemize}
\item[(1)]In our previous work \cite{YLL24, YL24, YL26}, we characterized the link quivers of Hopf algebras with the dual Chevalley property. Although the proof in the coquasi-Hopf algebra case is different from a technical perspective, it is worth noting that, from the perspective of link quivers, the link quivers of coquasi-Hopf algebras with the dual Chevalley property coincide with those of Hopf algebras with the Chevalley property; they are completely determined by the Grothendieck ring of the coradical and the number of arrows pointing to the simple subcoalgebra $\k1$. From Example \ref{example:H1/H0} and Proposition \ref{prop:Majidmod=YD}, it follows that $(H_1/H_0)^{coH}$ and ${}^{coH}(H_1/H_0)$ admit a left-left and a right-right Yetter-Drinfeld module structure over $H_0$, respectively. This means that the arrows with ending vertex $\k1$ is related to the comodule decomposition of the Yetter-Drinfeld module ${}^{coH}(H_1/H_0)$. It should be emphasized again that, by Lemma \ref{lem:P1prime>0}, we have $\mid{}^1\mathcal{P}\mid=\mid \mathcal{P}^1\mid$.
\item[(2)]Let $G$ be a group and let $\mathfrak{C}$ be the set of its conjugacy classes. A \textit{ramification datum} $\chi$ of $G$ is a formal sum $\sum_{U \in \mathfrak{C}} \chi_U U$ of conjugacy classes with coefficients in $\mathbb{Z}_+$. The corresponding \textit{Hopf quiver} $\mathrm{Q} = \mathrm{Q}(G,\chi)$ is defined as follows: the vertex set is $\mathrm{Q}_0 = G$, and for each $x \in G$ and each conjugacy class $U \in \mathfrak{C}$, there are $\chi_U$ arrows from $x$ to $ux$ for any $u \in U$. According to \cite[Section 3]{Hua09}, the link quiver of a pointed coquasi-Hopf algebra $H$ is a Hopf quiver, which is consistent with (1) and Proposition \ref{prop:alphaijk}. Indeed, in this case the coradical of $H$ is a group algebra, and the comodule decomposition of a Yetter-Drinfeld module over a group algebra is determined by conjugacy classes (see, for example \cite[Theorem 3.3]{Hua12}). More explicitly, if a Yetter-Drinfeld module $V$ over $\k G$ contains, in its comodule decomposition, a simple right comodule $V_h$ with coefficient coalgebra $\k h$ for some $h \in G$, then it follows from (\ref{eq:YDmod2}) in Remark \ref{rm:YD2def} that for all $g \in G$, $V$ contains the simple right comodule $V_{ghg^{-1}}$, whose coefficient coalgebra is $\k ghg^{-1}$. This is primarily because  $g$ is an invertible element whose action on any nonzero element $ v \in V$ is nonzero. However, this property does not necessarily hold in the case of non-pointed cosemisimple coquasi-Hopf algebras.
\end{itemize}
\end{remark}

\subsection{Decomposition into link-indecomposable components}
This subsection is devoted to studying the link-indecomposable components of a coquasi-Hopf algebra with the dual Chevalley property.

Let $\mathrm{Q}=(\mathrm{Q}_0, \mathrm{Q}_1)$ be a quiver, where $\mathrm{Q}_0$ is the set of vertices and $\mathrm{Q}_1$ the set of arrows. For an arrow $\alpha\in \mathrm{Q}_1$, let $\operatorname{s}(\alpha)$ denote its start vertex and $\operatorname{t}(\alpha)$ its end vertex. Recall that a \textit{path} $\beta$ in $\mathrm{Q}$ is a finite sequence of concatenated arrows $\beta=\alpha_n\alpha_{n-1}\cdots \alpha_1$ satisfying $\operatorname{t}(\alpha_i)=\operatorname{s}(\alpha_{i+1})$ for $i=1, 2, \cdots, n-1$.
The \textit{length} of a path is the number of arrows in the sequence; vertices are regarded as paths of length zero.
A \textit{walk} is a nonempty sequence of arrows $\alpha_1, \alpha_2, \cdots, \alpha_m$ such that there exists a set $\{\lambda_i\}_{1\leq i\leq m}$ with $\lambda_i \in \{-1, 1\}$ for which $\alpha_1^{\lambda_1}\alpha_2^{\lambda_2}\cdots \alpha_m^{\lambda_m}$ is a path in $\mathrm{Q}$. A quiver $\mathrm{Q}=(\mathrm{Q}_0,\mathrm{Q}_1)$ is said to be \textit{connected} if for any two vertices $a$ and $b$, there exists a walk from $a$ to $b$.

We present the following definition.
\begin{definition}\emph{(}\cite[Definition 1.1]{Mon95}\emph{)}
A subcoalgebra $H^\prime$ of a coalgebra $H$ is called link-indecomposable if the link quiver $\mathrm{Q}(H^\prime)$ of $H^\prime$ is connected.
A link-indecomposable component of $H$ is a maximal link-indecomposable subcoalgebra.
\end{definition}

Next, we investigate the decomposition of link-indecomposable components of a coquasi-Hopf algebra with the dual Chevalley property, as well as the product between these link-indecomposable components. Before proceeding further, we present the folllowing lemma.
\begin{lemma}\label{lem:CD,CFarrow}
Let $H$ be a non-cosemisimple coquasi-Hopf algebra over $\k$ with the dual Chevalley property, let $\mathcal{S}=\{C_i\mid i\in I\}$ denote the set of all simple subcoalgebras of $H$. For any $i,j,k\in I$, if there is some arrow in the link quiver of $H$ from $C_j$ to $C_i$, then the following hold:
\begin{itemize}
\item[(1)]For every $C_t \in \mathcal{S}$ contained in $C_k\cdot C_j$ there exists some $C_{t^\prime} \in \mathcal{S}$ contained in $C_k\cdot C_i$ with an arrow from $C_t$ to $C_{t^\prime}$, and for every $C_{t^\prime} \in \mathcal{S}$ contained in $C_k\cdot C_i$ there exists some $C_t \in \mathcal{S}$ contained in $C_k\cdot C_j$ with an arrow from $C_t$ to $C_{t^\prime}$;
\item[(2)]For every $C_l \in \mathcal{S}$ contained in $C_j\cdot C_k$ there exists some $C_{l^\prime} \in \mathcal{S}$ contained in $C_i\cdot C_k$ with an arrow from $C_l$ to $C_{l^\prime}$, and for every $C_{l^\prime} \in \mathcal{S}$ contained in $C_i\cdot C_k$ there exists some $C_{l^\prime} \in \mathcal{S}$ contained in $C_j\cdot C_k$ with an arrow from $C_l$ to $C_{l^\prime}$.
\end{itemize}
\end{lemma}

\begin{proof}
\begin{itemize}
\item[(1)]Let ${}^1\mathcal{S}=\{C_u\mid u\in J\}$, where $J \subseteq I.$ According to Proposition \ref{prop:alphaijk}, the number of arrows from $C_j$ to $C_i$ is given by $$\mid {}^{\C_i}\mathcal{P}^{\C_j}\mid=\sum_{u\in J}\mid{}^1 \mathcal{P}^{\C_u}\mid \alpha_{iu}^j.$$
This means that there exists some $u \in J$ such that $\alpha_{iu}^j > 0$, which implies that $C_i\cdot  C_u$ contains $C_j$ with a nonzero coefficient. Consequently, for any $C_t \in \mathcal{S}$ contained in $C_k\cdot  C_j$, it must also be contained in $C_k\cdot  C_i\cdot  C_u$. Then there exists some $C_{t^\prime} \in \mathcal{S}$ contained in $C_k\cdot  C_i$ such that $C_{t^\prime}\cdot  C_u$ contains $C_t$ with a nonzero coefficient. It follows that
$$\mid{}^{\C_{t^\prime}}\mathcal{P}^{\C_t}\mid=\sum_{u\in J}\mid {}^1\mathcal{P}^{\C_u}\mid \alpha_{t^\prime u}^t >0.$$
On the other hand, Corollary \ref{coro:cpdspecial} (1) implies that there is an arrow from $S(C_i)$ to $S(C_j)$. It follows from Proposition \ref{prop:alphaijk} that there exists some $u\in J$ such that $\alpha_{uj^*}^{i^*}>0.$ Note that for any $C_{t^\prime}\in\mathcal{S}$ that is contained in $C_k\cdot C_i$, the product $S(C_i)\cdot S(C_k)$ contains $S(C_{t^\prime})$ with a nonzero coefficient. We know that $S(C_{t^\prime})$ is contained in $C_u\cdot S(C_j)\cdot S(C_k)$. Hence, there exists some $C_l\in\mathcal{S}$ contained in $S(C_j)\cdot S(C_k)$ such that $C_u C_l$ contains $S(C_{t^\prime})$ with a nonzero coefficient. Let $C_t=S(C_l)$. Then, by Proposition \ref{prop:alphaijk} and Corollary \ref{coro:cpdspecial} (1), there exists an arrow from $C_t$ to $C_{t^\prime}$.
\item[(2)]Similar to the proof of (1), we can show that (2) holds.
\end{itemize}
\end{proof}

It should be pointed out that although the above result is consistent with the Hopf algebra case (see items (I) and (II) in the proof of \cite[Lemma 3.12]{Li22}; also \cite[Lemma 3.2]{YLL24}), the proof method is completely different. In the context of Hopf algebras, items (I) and (II) in the proof of \cite[Lemma 3.12]{Li22} play a key role in both \cite{Li22} and \cite{YLL24}: it was used in \cite{Li22} to characterize the products of link-indecomposable components, and in \cite{YLL24}
to describe the link quiver of Hopf algebras with the dual Chevalley property. Moreover, the proof in the Hopf algebra case cannot be directly applied to the coquasi-Hopf algebra case. In the proof of the Hopf algebra case, it heavily depends on the fact that
$\sum_{k=1}^r S(c_{ik})(c_{kj}x)=\varepsilon(c_{ij})x,$
where $\C=(c_{ij})_{r\times r}$ is a basic multiplicative matrix and $x\in H$.
In the coquasi-Hopf algebra setting, due to the presence of an reassociator, $ \sum_{t,l=1}^{r} \alpha(c_{lt})S(c_{il})(c_{tj}x)$ is more complicated and generally not equal to $\alpha(c_{ij})x$.

By \cite[Theorem 2.1]{Mon95}, any coalgebra is the direct sum of its link-indecomposable components. Building on this, we now let $H$ be a coquasi-Hopf algebra with the dual Chevalley property and proceed to study its link-indecomposable components.
For each $C \in \mathcal{S}$, let $H_{(C)}$ denote the link-indecomposable component containing $C$. In particular, we denote the link-indecomposable component containing $\k 1$ by $H_{(1)}$.
We have the following proposition, which generalizes not only \cite[Proposition 3.16]{Li22} from Hopf algebras to coquasi-Hopf algebras but also \cite[Theorem 4.1]{Hua09} from the pointed case to the dual Chevalley case.

\begin{proposition}\label{prop:HcHd}
Let $H$ be a non-cosemisimple coquasi-Hopf algebra over $\k$ with the dual Chevalley property and $\mathcal{S}$ be the set of all simple subcoalgebras of $H$. Then:
\begin{itemize}
\item[(1)]For any $C\in \mathcal{S}$, $S(H_{(C)})\subseteq H_{(S(C))}$;
\item[(2)]For any $C, D\in\mathcal{S}$, $H_{(C)}H_{(D)}\subseteq \sum\limits_{E\in\mathcal{S}, E\subseteq CD} H_{(E)}$;
\item[(3)]For any $C\in\mathcal{S}$, $H_{(C)}=CH_{(1)}=H_{(1)}C;$
\item[(4)]$H_{(1)}$ is a coquasi-Hopf subalgebra.
\end{itemize}
\end{proposition}

\begin{proof}
\begin{itemize}
\item [(1)]Because $S$ is a coalgebra antimorphism, we immediately find that $S(H_{(C)})$ is a subcoalgebra of $H$. It remains to prove that $S(H_{(C)})$ is link-indecomposable. By Corollary \ref{coro:cpdspecial} (1), we find that the number of arrows from $C_j$ to $C_i$ equals that from $S(C_j)$ to $S(C_i)$. Consequently, $S(H_{(C)})$ is link-indecomposable and is naturally contained in $H_{S(C)}$.
\item[(2)]It suffices to prove that for each simple subcoalgebra $E^\prime$ of $H_{(C)}H_{(D)}$, there exists a walk from $E^\prime$ to some $E \in \mathcal{S}$ contained in $CD$. By \cite[Corollary 4.1.8]{Rad12}, we have $$(H_{(C)}\otimes H_{(D)})_0=(H_{(C)})_0\otimes (H_{(D)})_0.$$ Then, together with the fact that the multiplication $m$ is a coalgebra epimorphism, that $H$ has the dual Chevalley property, and \cite[Corollary 5.3.5]{Mon93}, we can show that $$(H_{(C)}H_{(D)})_0= (H_{(C)})_0(H_{(D)})_0.$$
This implies that each simple subcoalgebra $E^\prime$ of $H_{(C)}H_{(D)}$ must be contained in some $C^\prime D^\prime$,
where $C^\prime\in (H_{(C)})_0, D^\prime\in (H_{(D)})_0$. According to Lemma \ref{lem:CD,CFarrow} we find that for any $E^\prime\in \mathcal{S}$ contained in $C^\prime D^\prime$, there exists some $E\in \mathcal{S}$ contained in $CD$ such that there is a walk from $E$ to $E^\prime.$ This completes the proof.
\item[(3)]Obviously, $CH_{(1)}$ is a coalgebra, and by (2), $CH_{(1)}\subseteq H_{(C)}$. Next we prove that $CH_{(1)}$ is link-indecomposable. For any simple subcoalgebra $E\in (CH_{(1)})_0$, by an argument similar to the proof of (2), we know that $E$ is contained in $CD$ for some simple subcoalgebra $D\in (H_{(1)})_0.$ Note that we can find a walk from $\k1$ to $D$ in the link quiver of $H_{(1)}$ passing through $F_0, F_1, \cdots, F_n$ such that $F_0=\k1$ and $F_n=D.$ Then by Lemma \ref{lem:CD,CFarrow} (1), we can prove by induction on $n$ that we can find a walk from $C$ to $E.$ Therefore, the link quiver of $CH_{(1)}$ is connected, and by definition we have $H_{(C)}=CH_{(1)}$. Similarly, we can prove that $H_{(C)}=H_{(1)}C.$
\item[(4)]This is a direct consequence of (1) and (2).
\end{itemize}
\end{proof}

\begin{remark}\rm
Note that in the Hopf algebra setting, for any pointed Hopf algebra $H$, the link-indecomposable component $H_{(1)}$ containing $\k1$ is a normal Hopf subalgebra by \cite[Theorem 3.2]{Mon95}. If $H$ has the dual Chevalley property, $H_{(1)}$ remains a Hopf subalgebra (\cite[Proposition 3.16]{Li22}), but not necessarily normal when $H$ is non-pointed (see \cite[Example 6.1]{YLL24}). Thus (coquasi-)Hopf algebras with the dual Chevalley property form a non-trivial generalization of pointed (coquasi-)Hopf algebras.
\end{remark}

Next, we define the following equivalence relation on $\mathcal{S}:$
we say that $C$ and $D$ are equivalent if $C(H_{(1)})_0=D(H_{(1)})_0.$ Clearly, by Proposition \ref{prop:HcHd} (3), $C$ and $D$ are equivalent if and only if there exists a walk between $C$ and $D$. Let $\mathcal{S}_0\subseteq \mathcal{S}$ be a full set of chosen pairwise non-equivalent representatives with respect to this equivalence relation. Then we have the following corollary.
\begin{corollary}\label{coro:H=CH(1)}
Let $H$ be a non-cosemisimple coquasi-Hopf algebra over $\k$ with the dual Chevalley property. Then the link-indecomposable decomposition of $H$ is given by:
$$
H=\bigoplus_{C\in\mathcal{S}_0} CH_{(1)}.
$$
\end{corollary}

By Remark \ref{rm:Hopf;pointed} (2), a Hopf quiver is a special case of the link quiver of a coquasi-Hopf algebra with the dual Chevalley property.
Recall that in \cite[Section 3]{CR02} (see also \cite[Section 4]{Hua09}), a Hopf quiver $\mathrm{Q}(G, \chi)$ is connected if and only if the union $\cup_{\chi_U\neq0}U$ generates $G$. Inspired by this, we next give a characterization of $(H_{(1)})_0$.
For convenience, for each $C\in\mathcal{S} $ and $\lambda\in\{-1, 1\}$, we define
$$
C^{\lambda}:=
\begin{cases}
C, & \text{if } \lambda=1;\\
S(C), &\text{if } \lambda=-1.
\end{cases}
$$

We have the following proposition.
\begin{proposition}\label{prop:(H1)0}
Let $H$ be a non-cosemisimple coquasi-Hopf algebra over $\k$ with the dual Chevalley property, and let ${}^1\mathcal{S}=\{C\in\mathcal{S}\mid \k1+C\neq \k1\wedge C\}$. Then the link quiver $\mathrm{Q}(H)$ of $H$ is connected if and only if for every $D \in \mathcal{S}$, there exist $C_1, \cdots, C_n \in {}^1\mathcal{S}$ such that $C_1^{\lambda_1} C_2^{\lambda_2} \cdots C_n^{\lambda_n}$ contains $D$ with a nonzero coefficient, where $\{\lambda_i \mid 1 \le i \le n\} \subseteq \{-1, 1\}$. In particular, $(H_{(1)})_0$ is generated by $\{c\in C\mid C\in {}^1\mathcal{S}\}\cup\{S(c)\mid c\in C \text{ for some } C\in {}^1\mathcal{S}\}.$
\end{proposition}
\begin{proof}
For every $D\in\mathcal{S}$ contained in $C_1^{\lambda_1} C_2^{\lambda_2} \cdots C_n^{\lambda_n}$, where $C_i\in{}^1\mathcal{S}$ and $\lambda_i\in\{-1,1\}$ for $1\leq i\leq n,$ we prove by induction on $n$ that there exists a walk from $\k1$ to $D.$ By the definition of ${}^1\mathcal{S}$ and Lemma \ref{lem:P1prime>0} (1), we know that the statement holds trivially when $n=1.$ Now assume $n>1$. Then $D$ must be contained in $E C_n^{\lambda_n}$ for some $E\in\mathcal{S}$ that is itself contained in $C_1^{\lambda_1} C_2^{\lambda_2} \cdots C_{n-1}^{\lambda_{n-1}}$.
If $\lambda_n=1,$ then there exists a non-trivial $(1, \C_n)$-primitive matrix $\Y$, and we consider the expression
$$
E\odot^\prime \left(\begin{array}{cc}
1&\Y\\
0&\C_n
 \end{array}\right);
 $$
 if $\lambda_n=-1,$ then there exists a non-trivial $(S(\C_n), 1)$-primitive matrix $\Y^\prime$, and we consider the expression
$$
E\odot^\prime \left(\begin{array}{cc}
S(\C_n)&\Y^\prime\\
0&1
 \end{array}\right).
 $$
From Lemma \ref{lem:CD,CFarrow}, there exists a walk from $D$ to $E.$ By the induction hypothesis, there exists a walk between
$\k1$ and $E$; therefore, we can find a walk from
$\k1$ to $D$. This shows that $\mathrm{Q}(H)$ is connected.

Conversely, if $\mathrm{Q}(H)$ is connected, then for any $D\in\mathcal{S}$ there exists a walk from $\k1$ to $D$ with vertices $F_0, F_1,\cdots ,F_m\in\mathcal{S}$ such that $F_0=\k 1$ and $F_m=D.$ We now proceed by induction on $m$ to prove that for each $F_i$ with $i>1$, there exist $C_1,C_2,\cdots C_i\in{}^1\mathcal{S}$ such that $C_1^{\lambda_1}C_2^{\lambda_2}\cdots C_{i}^{\lambda_{i}}$ contains $F_i$ with a nonzero coefficient, where $\lambda_j\in \{-1,1\}$ for $1\leq j\leq i$. When $m=1$, there exists an arrow either from $F_1$ to $\k1$ or from $\k 1$ to $F_1$. By Lemma \ref{lem:P1prime>0} (1), the proof follows immediately by taking $C_1=F_1$ or $C_1=S(F_1)$ accordingly. If $m>1$, then there exists an arrow $\X\in\mathcal{P}$ either from $F_{m}$ to $F_{m-1}$ or from $F_{m-1}$ to $F_{m}$. We consider the two cases separately.
\begin{itemize}
\item[Case I:]
If $\X\in {}^{\F_{m-1}}\mathcal{P}^{\F_{m}}$, then by the construction of $\mathcal{P}$ (see Proposition \ref{prop:Pcomplete}) there exists a non-trivial $(1, \G)$-primitive matrix $\Z\in{}^1\mathcal{P}$ such that $\X\in{}^{\F_{m-1}}\mathcal{P}_{\Z}$, where $G\in{}^1\mathcal{S}$ and $\G\in\mathcal{M}$. Setting $C_m=G$, we obtain that $F_{m-1}G$ contains $F_{m}$ with a nonzero coefficient.
By the induction assumption, there exist $C_1,C_2,\cdots C_{m-1}\in{}^1\mathcal{S}$ such that $C_1^{\lambda_1}C_2^{\lambda_2}\cdots C_{m-1}^{\lambda_{m-1}}$ contains $F_{m-1}$ with a nonzero coefficient. Hence, $F_m$ appears with a nonzero coefficient in the product $C_1^{\lambda_1}C_2^{\lambda_2}\cdots C_{m-1}^{\lambda_{m-1}}C_m$.
\item[Case II:]If $\X\in {}^{\F_{m-1}}\mathcal{P}^{\F_{m}}$, then by Corollary \ref{coro:cpdspecial} (1), there exists an arrow $\X^\prime\in\mathcal{P}$ from $S(F_{m-1})$ to $S(F_m)$. Similar to the proof in Case I, we can show that there exists some $G^\prime \in {}^1\mathcal{S}$ such that $S(F_{m-1}) G^\prime$ contains $S(F_m)$ with a nonzero coefficient. Since $S$ is an anti-involution on $\mathbb{Z}\mathcal{S}$, it follows that $S(G^\prime)F_{m-1} $ contains $F_m$ with a nonzero coefficient. Setting $C_0=G^\prime$, we finally obtain that $C_0^{-1}C_1^{\lambda_1}C_2^{\lambda_2}\cdots C_{m-1}^{\lambda_{m}}$ contains $F_{m}$ with a nonzero coefficient.
\end{itemize}
By Proposition \ref{prop:HcHd} (3), $H_{(1)}$ is a coquasi-Hopf algebra with connected link quiver. Then for any simple subcoalgebra $D$ of $H_{(1)}$, there exist $C_1, \cdots, C_n \in {}^1\mathcal{S}$ such that $C_1^{\lambda_1} C_2^{\lambda_2} \cdots C_n^{\lambda_n}$ contains $D$ with a nonzero coefficient. This means that $(H_{(1)})_0$ is generated by $$\{c\in C\mid C\in {}^1\mathcal{S}\}\cup\{S(c)\mid c\in C \text{ for some } C\in {}^1\mathcal{S}\}.$$
\end{proof}
As a direct corollary of Proposition \ref{prop:(H1)0}, we have:
\begin{corollary}\label{coro:H(1)pointed}
Let $H$ be a non-cosemisimple coquasi-Hopf algebra over $\k$ with the dual Chevalley property, and let $\mathrm{Q}(H)$ be its link quiver. If all the simple subcoalgebras in ${}^1\mathcal{S}$ are one-dimensional, then the link-indecomposable component $H_{(1)}$ containing $\k 1$ is a pointed coquasi-Hopf algebra.
\end{corollary}

\section{Generalized Hopf quivers}\label{section4}
In this section, we present a method for constructing coquasi-Hopf algebras with the dual Chevalley property using quivers, generalizing the Hopf quiver theory. For a quiver $Q$ and a family of simple coalgebras indexed by the vertices of $\mathrm{Q}$, we define a modified generalized path coalgebra (with link quiver $\mathrm{Q}$), then characterize when it admits a graded coquasi-Hopf algebra structure with the dual Chevalley property, and classify such coquasi-Hopf algebra structures over it. We also prove the generalized dual Gabriel's theorem for coquasi-Hopf algebras with
the dual Chevalley property.

\subsection{A modified generalized path coalgebra}\label{subsection4.1}
Here we introduce the modified generalized path coalgebra, an important quiver-related coalgebra for our purposes.

Let $\mathrm{Q}=(\mathrm{Q}_0, \mathrm{Q}_1)$ be a quiver, where $\mathrm{Q}_0$ is the set of vertices and $\mathrm{Q}_1$ the set of arrows. Let $\mathcal{S} = \{C_{i} \mid i \in \mathrm{Q}_0\}$ be a collection of simple coalgebras indexed by the vertices of $\mathrm{Q}$, where each $C_{i}$ is equipped with a comultiplication $\Delta_i$ and a counit $\varepsilon_i$. We define $\mathcal{M}$ as a set of representatives for the similarity classes of basic multiplicative matrices across all coalgebras $C_{i}\in\mathcal{S}$. More specifically, let
$\mathcal{M}=\{\C_{i}\mid i \in \mathrm{Q}_0\},$ where each
$\C_{i}=(c_{jk}^{i})_{r_{i}\times r_i}$ is a basic multiplicative matrix of $C_{i}\in\mathcal{S}.$

Now we introduce the definition of $\mathcal{M}$-paths.
The elements of $\{c_{jk}^{i}\mid 1\leq j,k\leq r_i, i\in \mathrm{Q}_0\}$ are referred to as \textit{$\mathcal{M}$-paths of length 0}.
For $n \geq 1$, a \textit{$\mathcal{M}$-path of length $n$} is a formal sequence of the form $c_{i1}^{t(\alpha_n)}\alpha_n\alpha_{n-1}\cdots \alpha_1 c_{1j}^{s(\alpha_{1})}$, where $\alpha_n\alpha_{n-1}\cdots \alpha_1$ is a path in $\mathrm{Q}$ of length $n$, $1\leq i\leq r_{t(\alpha_n)}$, and $1\leq j\leq r_{s(\alpha_1)}$.

Let $\k(\mathrm{Q}, \mathcal{M})$ denote the $\k$-linear space spanned by all $\mathcal{M}$-paths.
We remark that $\k(\mathrm{Q}, \mathcal{M})$ carries a coalgebra structure, with the comultiplication and counit defined as follows:
\begin{itemize}
\item for any $\mathcal{M}$-path $c_{jk}^{i}$ of length $0$,
\begin{eqnarray*}
\Delta(c_{jk}^{i})=\sum\limits_{t=1}^{r_i}c_{jt}^{i}\otimes c_{tk}^{i},\;\;\varepsilon(c_{jk}^i)=\delta_{j,k};
\end{eqnarray*}
\item and for each $\mathcal{M}$-path $c_{i1}^{t(\alpha_n)}\alpha_n\alpha_{n-1}\cdots \alpha_1 c_{1j}^{s(\alpha_{1})}$ of length $n\geq 1$,
 \begin{eqnarray*}
&&\Delta(c_{i1}^{t(\alpha_n)}\alpha_n\alpha_{n-1}\cdots \alpha_1 c_{1j}^{s(\alpha_{1})})\\
&=&\sum\limits_{k=2}^{n-1}\sum\limits_{l_k=1}^{r_{s(\alpha_k)}}c_{i1}^{t(\alpha_n)}\alpha_n\alpha_{n-1}\cdots \alpha_{k}c_{1l_k}^{s(\alpha_{k})}\otimes c_{l_k 1}^{t(\alpha_{k-1})}\alpha_{k-1}\alpha_{k-2}\cdots\alpha_1c_{1j}^{s(\alpha_1)}  \\
&&+\sum\limits_{l=1}^{r_{t(\alpha_{n})}}c_{il}^{t(\alpha_n)}\otimes c_{l1}^{t(\alpha_n)}\alpha_n\alpha_{n-1}\cdots \alpha_1 c_{1j}^{s(\alpha_1)}+\sum\limits_{l=1}^{r_{s(\alpha_1)}}c_{i1}^{t(\alpha_n)}\alpha_n\alpha_{n-1}\cdots \alpha_1 c_{1l}^{s(\alpha_1)}\otimes c_{lj}^{s(\alpha_1)},\\
&&\varepsilon(c_{i1}^{t(\alpha_n)}\alpha_n\alpha_{n-1}\cdots \alpha_1 c_{1j}^{s(\alpha_{1})})=0.
\end{eqnarray*}
\end{itemize}
This type of coalgebra is called a \textit{modified generalized path coalgebra} associated with the quiver $\mathrm{Q}$ and the set of representatives for the similarity classes of basic multiplicative matrices $\mathcal{M}.$

\begin{remark}\rm
If all the simple coalgebras in $\mathcal{S}$ are $1$-dimensional, then the modified generalized path coalgebra defined above reduces to the path coalgebra as introduced in \cite[Definition 2.1]{CR02}.
Moreover, a modified generalized path coalgebra can be seen as a subcoalgebra of the generalized path coalgebra introduced by Li and Liu \cite{LL08} (see also \cite[Subsection 5.12]{LS07}), provided each $\mathcal{M}$-path $c_{i1}^{t(\alpha_n)}\alpha_n\alpha_{n-1}\cdots \alpha_1 c_{1j}^{s(\alpha_{1})}$ of length $n\geq 1$ is identified with $c^{t(\alpha_n)}_{i1} \alpha_{n}c^{s(\alpha_n)}_{11}\alpha_{n-1}\cdots c^{t(\alpha_i)}_{11}\alpha_ic^{s(\alpha_i)}_{11}  \cdots c^{t(\alpha_1)}_{11}\alpha_1 c^{s(\alpha_1)}_{1j}.$ Since the full generalized path coalgebra is too large for our purposes, we restrict to this subcoalgebra, which we refer to as the modified generalized path coalgebra.
\end{remark}

Note that $\k(\mathrm{Q}, \mathcal{M})$ is naturally a graded coalgebra with respect to the length grading. We have $\k(\mathrm{Q}, \mathcal{M})=\bigoplus_{n\geq 0}\k(\mathrm{Q}_n, \mathcal{M}),$ where $\k(\mathrm{Q}_n, \mathcal{M})$ is the $\k$-linear space with basis the set of all $\mathcal{M}$-paths of length $n$. The coradical of $\k(\mathrm{Q}, \mathcal{M})$ is $\k(\mathrm{Q}_0, \mathcal{M})=\bigoplus_{i\in\mathrm{Q}_0} C_i,$ and $\k(\mathrm{Q}, \mathcal{S})_n=\bigoplus_{i=0}^n \k(\mathrm{Q}_i, \mathcal{S})$. For any $\alpha\in\mathrm{Q}_1$, $1\leq i\leq r_{t(\alpha)},$ and $ 1\leq j\leq r_{s(\alpha)}$, we have
$$
\Delta(c_{i1}^{t(\alpha)}\alpha c_{1j}^{s(\alpha)})=\sum\limits_{k=1}^{r_{t(\alpha)}}c_{ik}^{t(\alpha)}\otimes c_{k1}^{t(\alpha)}\alpha c_{1j}^{s(\alpha)}+\sum\limits_{l=1}^{r_{s(\alpha)}} c_{i1}^{t(\alpha)}\alpha c_{1l}^{s(\alpha)}\otimes c_{lj}^{s(\alpha)}.
$$
It is straightforward to verify that the matrix $\X_{\mathcal{M},\alpha}:=\left(c_{i1}^{t(\alpha)}\alpha c_{1j}^{s(\alpha)} \right)_{r_{t(\alpha)}\times r_{s(\alpha)}}$ is a non-trivial $(\C^{t(\alpha)}, \C^{s(\alpha)})$-primitive matrix and that
\begin{eqnarray*}
\k(\mathrm{Q}, \mathcal{M})_1/\k(\mathrm{Q}, \mathcal{M})_0
=\operatorname{span}_{\k}\{c_{i1}^{t(\alpha)}\alpha c_{1j}^{s(\alpha)}\mid 1\leq i\leq r_{t(\alpha)}, 1\leq j\leq r_{s(\alpha)}, \alpha\in\mathrm{Q}_1\}.
\end{eqnarray*}
According to Lemma \ref{lem:linkquiver=S,P}, we obtain the following result.
\begin{lemma}\label{lem:k(Q,S)link}
With the notations above, the link quiver of the coalgebra $\k(\mathrm{Q}, \mathcal{M})$ is precisely the original quiver $\mathrm{Q}.$
\end{lemma}

Given a coalgebra $H$ and an $H$-bicomodule $M$ with structure maps $\rho^L_M$ and $\rho^R_M$, one defines the \textit{cotensor coalgebra} (see \cite[p. 1526]{Nic78} and \cite[Subsection 1.2]{VZ04})
\begin{eqnarray}\label{eq:cotensor}
\operatorname{CoT}_H(M):=H\oplus M\oplus M^{\Box 2}\oplus \cdots \oplus M^{\Box n}\oplus \cdots,
\end{eqnarray}
where $M^{\Box 0}:=H$, $M^{\Box 1}:=M$, and for each $n\geq 2$, $M^{\Box n}$ is the kernel of the $\k$-linear map:
$$
M^{\otimes n}\xrightarrow{f}\bigoplus\limits_{i=1}^{n-1}M^{\otimes i}\otimes H \otimes M^{\otimes n-i},
$$
with $f=(f_1, f_2, \cdots , f_{n-1})$ and $f_i=\id^{\otimes i-1}\otimes \rho^R_M\otimes \id^{\otimes n-i}-\id^{\otimes i}\otimes \rho_M^L\otimes \id^{\otimes n-i-1}$ for $i=1, 2, \cdots, n-1.$
The coalgebra structure on $\operatorname{CoT}_H(M)$ is defined by:
\begin{eqnarray*}
\varepsilon \mid_H=\varepsilon_H,\;\;\varepsilon \mid_{M^{\Box i}}=0\;\; (i\geq 1),\\
\Delta\mid_H=\Delta_H,\;\;\Delta\mid_M=\rho^L_M+\rho^R_M,
\end{eqnarray*}
and for $\sum m^1\otimes m^2\otimes \cdots \otimes m^n\in M^{\Box n}$, $n\geq 2,$
\begin{eqnarray*}
&&\Delta(\sum m^1\otimes m^2\otimes \cdots \otimes m^n)\\
&:=&\sum \rho_M^L(m^1)\otimes  m^2\otimes \cdots \otimes m^n+ \sum m^1\otimes (m^2\otimes \cdots \otimes m^n)+\cdots \\
&&+\sum(m^1\otimes m^2\otimes \cdots \otimes m^{n-1})\otimes m^n+\sum m^1\otimes m^2\otimes \cdots \otimes m^{n-1}\otimes \rho_M^R(m^n)\\
&\in&\bigoplus\limits_{i=0}^{n}M^{\Box i} \otimes M^{\Box (n-i)}.
\end{eqnarray*}

We have the following universal property.
\begin{lemma}\emph{(}\cite[Lemma 1.2]{VZ04}\emph{)} \label{lemma:universal}
Given a coalgebra $H$ and an $H$-bicomodule $M$, let $\psi:X\longrightarrow CoT_H(M)$ be a coalgebra map. Define $\psi_n=p_n\psi:X\longrightarrow M^{\Box n}$ as the projection onto the
$n$-th component. Then $\psi_0:X\longrightarrow H$ is a coalgebra map, and $\psi_1:X\longrightarrow M$ is an $H$-bicomodule map, where the $H$-bicomodule structure on $X$ is induced via $\psi_0$. For $n\geq2$, $\psi_n$ coincides with the $H$-bicomodule map obtained by the composition
$$X\xrightarrow{\Delta^{n-1}} X\otimes X \otimes \cdot\cdot\cdot\otimes X\xrightarrow{\psi_1^{\otimes n}} M^{\otimes n}.$$
Thus, $\psi$ is uniquely determined by $\psi_0$ and $\psi_1$.
Conversely, let $\psi_0:X\longrightarrow H$ be a coalgebra map, and $\psi_1:X\longrightarrow M$ an $H$-bicomodule map. Define
$\psi_n:X\longrightarrow M^{\otimes n}$ by the composition above. Then each $\psi_n$ is an $H$-bicomodule map with $\operatorname{Im}(\psi_n)\subseteq M^{\Box n}$. If for each $x\in X$ there are only finite $i$ such that $\psi_i(x)\neq 0$, then
$\psi=\sum_{i\geq0}\psi_i:X\longrightarrow CoT_H(M)$ is a coalgebra map.
\end{lemma}

Observe that $\k(\mathrm{Q}_1, \mathcal{M})$ naturally inherits a  $\k(\mathrm{Q}_0, \mathcal{M})$-bicomodule structure from the comultiplication. We now present an alternative definition of the modified generalized path coalgebra as a cotensor coalgebra.
\begin{lemma}\label{lem:cotensorcoalg}
There exists a coalgebra isomorphism $\k(\mathrm{Q}, \mathcal{M})\cong \operatorname{CoT}_{\k(\mathrm{Q}_0, \mathcal{M})}(\k(\mathrm{Q}_1, \mathcal{M})).$
\end{lemma}
\begin{proof}
Define a map
$$F:\k(\mathrm{Q}, \mathcal{M}) \rightarrow \operatorname{CoT}_{\k(\mathrm{Q}_0, \mathcal{M})}(\k(\mathrm{Q}_1, \mathcal{M})$$
by setting $$F\mid_{ \k(\mathrm{Q}_0, \mathcal{M})\oplus \k(\mathrm{Q}_1, \mathcal{M})}=\id,$$ and for any $\mathcal{M}$-path $c_{i1}^{t(\alpha_n)}\alpha_n\alpha_{n-1}\cdots \alpha_1 c_{1j}^{s(\alpha_{1})}$ of length $n\geq 2$,
\begin{eqnarray*}
&&F(c_{i1}^{t(\alpha_n)}\alpha_n\alpha_{n-1}\cdots \alpha_1 c_{1j}^{s(\alpha_{1})})\\
&=&\sum \limits_{k_2, \cdots, k_n}c_{i1}^{t(\alpha_n)}\alpha_n  c_{1k_n}^{s(\alpha_{n})} \otimes c_{k_n1}^{t(\alpha_{n-1})}\alpha_{n-1}  c_{1k_{n-1}}^{s(\alpha_{n-1})}\otimes \cdots \otimes c_{k_31}^{t(\alpha_2)}\alpha_{2}  c_{1k_{2}}^{s(\alpha_{2})}\otimes c_{k_{2} 1}^{t(\alpha_{1})}\alpha_1 c_{1j}^{s(\alpha_1)}.
\end{eqnarray*}
It is straightforward to show that $F$ is a coalgebra isomorphism.
\end{proof}

\begin{remark}\rm
Assume $\mathcal{M}_1=\{\C^1_i\mid i\in \mathrm{Q}_0\}$ and $\mathcal{M}_2=\{\C^2_i\mid i\in \mathrm{Q}_0\}$ are two sets of representatives of basic multiplicative matrices over the similarity classes of the coalgebras $C_i\in\mathcal{S}$. For an arrow $\alpha\in \mathrm{Q}_1,$ by \cite[Lemma 2.4]{Li22}, there exist invertible matrices $K, L$ over $\k$ such that $K\C^1_{t(\alpha)} K^{-1}=\C^2_{t(\alpha)}$ and $L \C^1_{s(\alpha)} L^{-1}=\C^2_{s(\alpha)}$. Then by \cite[Remark 2.5]{Li22} and \cite[Remark 3.2]{LZ19}, $K\X_{\mathcal{M}_1, \alpha}L^{-1}$ is a non-trivial $(\C_{t(\alpha)}^2, \C_{s(\alpha)}^2)$-primitive matrix.
Denote $\operatorname{span}(K\X_{\mathcal{M}_1, \alpha}L^{-1})$ and $\operatorname{span}(\X_{\mathcal{M}_2, \alpha})$ by the linear spaces spanned by the entries of the respective matrices. These spaces are isomorphic as $C_{t(\alpha)}$-$C_{s(\alpha)}$-bicomodules, whence $$\k(\mathrm{Q}_1, \mathcal{M}_1)\cong \k(\mathrm{Q}_1, \mathcal{M}_2)$$ as $\k(\mathrm{Q}_0, \mathcal{M}_1)$-bicomodules. Consequently, Lemmas \ref{lemma:universal} and \ref{lem:cotensorcoalg} imply that $$\k(\mathrm{Q}, \mathcal{M}_1)\cong \k(\mathrm{Q}, \mathcal{M}_2)$$ as coalgebras. Thus, up to coalgebra isomorphism, the modified generalized path coalgebra does not depend on the choice of representatives of basic multiplicative matrices. We may therefore denote it unambiguously by $\k(\mathrm{Q}, \mathcal{S})$.
\end{remark}

We can now state the generalized dual Gabriel's theorem for an arbitrary coalgebra $H$, which is essentially the dual of \cite[Theorem 1.9]{ARS95} but without the assumption of finite-dimensionality when $H$ is pointed.
\begin{proposition}\label{prop:CoalgGabriel}
Let $H$ be a coalgebra over $\k$, and let $\mathrm{Q}(H)=(\mathcal{S}, \mathcal{P})$ be its link quiver. Then there is a graded coalgebra embedding $\psi:\operatorname{gr}H\hookrightarrow \k(\mathrm{Q}(H), \mathcal{S})$ such that $gr(H) \supseteq \k(\mathrm{Q}(H)_0, \mathcal{S})\oplus \k(\mathrm{Q}(H)_1, \mathcal{S})$.
\end{proposition}
\begin{proof}
According to \cite[Lemma 2.8 and Corollary 2.11]{YLL24}, we know that $\k(\mathrm{Q}(H)_1, \mathcal{S})\cong H_1/H_0$ as $H_0$-bicomodules. To obtain a coalgebra map $$\psi:\operatorname{gr}H\hookrightarrow \operatorname{CoT}_{\k(\mathrm{Q}(H)_0, \mathcal{S})}(\k(\mathrm{Q}(H)_1, \mathcal{S})),$$ by Lemma \ref{lemma:universal}, we need to
construct a coalgebra map $\psi_0:\operatorname{gr}H\longrightarrow \k(\mathrm{Q}(H)_0, \mathcal{S})$, and a $\k(\mathrm{Q}(H)_0, \mathcal{S})$-bicomodule map $\psi_1:\operatorname{gr}H\longrightarrow \k(\mathrm{Q}(H)_1, \mathcal{S})$, such that for each $x\in \operatorname{gr}H$ there are only finite $i$ with $\psi_i(x)\neq0$, where $\psi_i$ is defined as in lemma \ref{lemma:universal}. In fact, we take $\psi_0$ and $\psi_1$ to be the canonical projections: $\psi_0:\operatorname{gr}H\longrightarrow H_0$, $\psi_1:\operatorname{gr}H\longrightarrow H_1/H_0$. Clearly $\psi_n(H_m/H_{m-1})=0$ for $n\neq m$. Then $$\psi=\sum\psi_i:\operatorname{gr}H\longrightarrow \operatorname{CoT}_{\k(\mathrm{Q}(H)_0, \mathcal{S})}(\k(\mathrm{Q}(H)_1, \mathcal{S}))$$ is a graded coalgebra map. By \cite[Theorem 5.3.1]{Mon93}, it is injective.
\end{proof}
It should be noted that the generalized dual Gabriel's theorem in \cite{LL08} does not require the condition $gr(H) \supseteq \k(\mathrm{Q}(H)_0, \mathcal{S})\oplus \k(\mathrm{Q}(H)_1, \mathcal{S})$. What we present here is a modified version.
In the next subsection, we will discuss the conditions under which $\k(\mathrm{Q}, \mathcal{S})$ admits a graded coquasi-Hopf algebra structure (with the dual Chevalley property) and present a coquasi-Hopf algebraic version of the generalized dual Gabriel's theorem.

\subsection{Coquasi-Hopf algebra structures on modified generalized
path coalgebras}
In this subsection, we let $H$ be a cosemisimple coquasi-Hopf algebra over $\k$, and denote by $\mathcal{S} = \{C_i \mid i \in I\}$ the set of all its simple subcoalgebras. We set $\mathcal{M}=\{\C_i\mid i\in I\}$, where each $\C_i \in \mathcal{M}$ is a basic multiplicative matrix of $C_i \in \mathcal{S}$. Suppose that the fusion rules of $\mathbb{Z}\mathcal{S}$ are given by
$$C_i\cdot C_j=\sum\limits_{t\in I}\alpha_{ij}^tC_t\;\;\text{for all }C_i, C_j\in\mathcal{S},$$ where $\alpha_{ij}^t\in\mathbb{Z}_+$.
Let $\{V_i \mid i \in I\}$ be a complete set of representatives of the isomorphism classes of simple left $H$-comodules such that the coefficient coalgebra of $V_i$ is $C_i$ for all $i \in I$; i.e., $\rho^L_{V_i}(V_i) \subseteq C_i \otimes V_i$.

Recall that the \textit{dimension vector} of a finite-dimensional left comodule $M$ over $H$ is defined as $\underline{\dim}(M) \in \mathbb{Z}_{+}^{(I)}$, with $\underline{\dim}(M)_i$ equal to the multiplicity of the simple left comodule $V_i$ in a Jordan-H$\ddot{\textrm{o}}$lder series of $M$. It should be noted that every finite-dimensional comodule over $H$ is semisimple; in particular, any finite-dimensional left-left Yetter-Drinfeld module over $H$ is also semisimple as a comodule. Therefore, in this case, the dimension vector is precisely the multiplicity of each simple comodule in its decomposition as a comodule.
We define
$$
\Lambda_H=\{\underline{\dim}(M) \mid M\text{ is a nonzero finite-dimensional Yetter-Drinfeld module in } {}^H_H\mathcal{YD}\}.
$$
We assemble the above notations into a datum $(H, \{C_i\}_{i\in I}, (n_1, n_2, \cdots), \{\alpha_{ij}^t\}_{i,j,t\in I})$ with $(n_1, n_2, \cdots) \in \Lambda_H$, and call this datum a \textit{cosemisimple datum} of $H$.

With the notations introduced above, we give the following definition.
\begin{definition}\label{def:generalizedHopfquiver}
Let $H$ be a cosemisimple coquasi-Hopf algebra over $\k$ with a cosemisimple datum $(H, \{C_i\}_{i\in I}, (n_1, n_2, \cdots), \{\alpha_{ij}^t\}_{i,j,t\in I})$. A generalized Hopf quiver associated with this datum has vertex set $\{C_i \mid i \in I\}$, and for all $i, j \in I$, there are $\sum_{k\in I}n_k \alpha_{jk}^i$ arrows from $C_j$ to $C_i$.
\end{definition}
Obviously, every Hopf quiver (see Remark \ref{rm:Hopf;pointed} (2)) is a generalized Hopf quiver, while the latter also includes examples that are not Hopf quivers.
\begin{example}\rm
Let $H(e_{\pm 1}, f_{\pm 1})$ be the Hopf algebra generated by $\{e_{i}, f_{i}\mid i\in \mathbb{Z}\}$ with relations
\begin{eqnarray*}
1=e_0+f_0, \;\;e_ie_j=e_{i+j},\;\;f_if_j=f_{i+j},\;\;e_if_j=f_je_i=0\;\;\;\;(i, j\in\mathbb{Z}).
\end{eqnarray*}
The comultiplication, counit and the antipode are given for all $i\in\mathbb{Z}$ by
\begin{eqnarray*}
&\Delta(e_i)=e_i\otimes e_i+f_i\otimes f_{-i},\;\;\varepsilon(e_i)=1,\;\;S(e_i)=e_{-i},\\
&\Delta(f_i)=e_i\otimes f_i+f_i\otimes e_{-i},\;\;\varepsilon(f_i)=0,\;\;S(f_i)=f_{i}.
\end{eqnarray*}
Let $g = e_0 - f_0$; it is clear that $g$ is a group-like element of order $2$. For any $i \ge 1$, set $C_i = \operatorname{span}\{e_i, f_i, e_{-i}, f_{-i}\}$. One can show that each $C_i$ is a simple subcoalgebra with basic multiplicative matrix $\mathcal{C}_i$, where $$
\C_i=\left(\begin{array}{cc}
e_i&f_i\\
f_{-i}&e_{-i}
 \end{array}\right).
$$ We then obtain $$H(e_{\pm 1}, f_{\pm 1})=\k1\oplus \k g\oplus \bigoplus_{i\geq 1}C_i,$$ from which it follows that the set $\mathcal{S}$ of all simple subcoalgebras is $\{\k1, \k g\} \cup \{C_i \mid i \ge 1\}$. Let $M=\span\{u, v\}$ be a left-left Yetter-Drinfeld module over $H(e_{\pm 1}, f_{\pm 1})$ with actions and coactions defined as follows:
\begin{eqnarray*}
&e_i \vartriangleright u=u,\;\; e_i \vartriangleright v=v,\;\;f_i \vartriangleright u=0,\;\; f_i \vartriangleright v=0,\;\;\;(i\in\mathbb{Z})&\\
&\rho_M^L(u)=e_1\otimes u+f_1\otimes v,\;\;\rho_M^L(v)=f_{-1}\otimes u+e_{-1}\otimes v.&
\end{eqnarray*}
Clearly, $M$ is a simple left $C_1$-comodule. Besides, we have $$C_1\cdot C_1=\k 1+\k g+C_2\;\;\text{and }\;C_1\cdot C_i=C_{i+1}+C_{i-1}\;\;\text{for }i\geq2 $$ in $\mathbb{Z}\mathcal{S}.$ Then, by definition, the following quiver is a generalized Hopf quiver
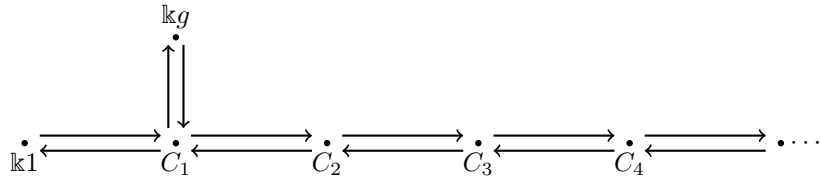
\begin{eqnarray*}
\begin{tikzpicture}
\filldraw [black] (4,0) circle (1pt) node[anchor=north]{$C_4$};
\filldraw [black] (-4,0) circle (1pt)node[anchor=north]{$\k1$};
\filldraw [black] (2,0) circle (1pt) node[anchor=north]{$C_3$};
\filldraw [black] (0,0) circle (1pt) node[anchor=north]{$C_2$};
\filldraw [black] (-2,0) circle (1pt)node[anchor=north]{$C_1$};
\filldraw [black] (6,0) circle (1pt) node[anchor=west]{$\cdots$};
\filldraw [black] (-2,1.4) circle (1pt) node[anchor=south]{$\k g$};
\draw[thick, ->] (-3.8,0.1).. controls (-3.8,0.1) and (-3.8,0.1) .. node[anchor=south]{}(-2.2,0.1) ;
\draw[thick, ->] (-2.2,-0.1).. controls (-2.2,-0.1) and (-2.2,-0.1) .. node[anchor=south]{}(-3.8,-0.1);
\draw[thick, ->] (-1.8,0.1).. controls (-1.8,0.1) and (-1.8,0.1) .. node[anchor=south]{}(-0.2,0.1) ;
\draw[thick, ->] (-0.2,-0.1).. controls (-0.2,-0.1) and (-0.2,-0.1) .. node[anchor=south]{}(-1.8,-0.1);
\draw[thick, ->] (0.2,0.1).. controls (0.2,0.1) and (0.2,0.1) .. node[anchor=south]{}(1.8,0.1) ;
\draw[thick, ->] (1.8,-0.1).. controls (1.8,-0.1) and (1.8,-0.1) .. node[anchor=south]{}(0.2,-0.1);
\draw[thick, ->] (2.2,0.1).. controls (2.2,0.1) and (2.2,0.1) .. node[anchor=south]{}(3.8,0.1) ;
\draw[thick, ->] (3.8,-0.1).. controls (3.8,-0.1) and (3.8,-0.1) .. node[anchor=south]{}(2.2,-0.1);
\draw[thick, ->] (4.2,0.1).. controls (4.2,0.1) and (4.2,0.1) .. node[anchor=south]{}(5.8,0.1) ;
\draw[thick, ->] (5.8,-0.1).. controls (5.8,-0.1) and (5.8,-0.1) .. node[anchor=south]{}(4.2,-0.1);
\draw[thick, ->] (-1.9,1.3).. controls (-1.9,1.3) and (-1.9,1.3) .. node[anchor=south]{}(-1.9,0.2) ;
\draw[thick, ->] (-2.1,0.2).. controls (-2.1,0.2) and (-2.1,0.2) .. node[anchor=south]{}(-2.1,1.3);
\end{tikzpicture}.
\end{eqnarray*}
Observe that $H(e_{\pm 1}, f_{\pm 1})$ is exactly the coradical of the Hopf algebra $H(e_{\pm 1}, f_{\pm 1}, u, v)$ defined in \cite[Definition 5.1]{YL24}, and the quiver above is the link quiver of $H(e_{\pm 1}, f_{\pm 1}, u, v)$.
\end{example}

\begin{proposition}\label{prop:triangularHopfquiver}
Let $H$ be a coquasitriangular cosemisimple coquasi-Hopf algebra over $\k$, and let $C_i\cdot C_j=\sum_{t\in I}\alpha_{ij}^tC_t$ be the fusion rules of $\mathbb{Z}\mathcal{S}$. Suppose $\mathrm{Q}$ is a quiver with vertex set $\{C_i\mid i\in I\}$, and let $\sum_{k\in I}n_k \alpha_{jk}^i$ denote the number of arrows from $C_j$ to $C_i$, where $(n_1, n_2,\cdots )\in\mathbb{Z}_+^{(I)}$ satisfies $\sum_{i\in I}n_i<\infty$. Then $\mathrm{Q}$ is a generalized Hopf quiver.
\end{proposition}
\begin{proof}
By definition, we only need to prove that $(n_1, n_2,\cdots )\in\Lambda_H$. In fact, from Example \ref{ex:coquaitriangular} we know that any comodule with dimension vector $(n_1, n_2,\cdots )$ admits a Yetter-Drinfeld module structure. Therefore $(n_1, n_2,\cdots )\in\Lambda_H$ and the proof is complete.
\end{proof}

In fact, the link quiver of any coquasi-Hopf algebra with the dual Chevalley property is a generalized Hopf quiver.
\begin{lemma}\label{lem:Hopflink=Hopfquiver}
Let $H$ be a coquasi-Hopf algebra over $\k$ with the dual Chevalley property. Then the link quiver of $H$ is a generalized Hopf quiver associated with a cosemisimple datum of $H_0$.
\end{lemma}

\begin{proof}
Let $\mathrm{Q}(H)=(\mathcal{S}, \mathcal{P})$ denote the link quiver of $H$, and set ${}^1\mathcal{S}=\{C_k\mid k\in J\}$, $\mathcal{S}^1=\{C_t\mid t\in T\}$ with $J, T \subseteq I$. Then, by Proposition \ref{prop:alphaijk}, we have that $$\mid {}^{\C_i}\mathcal{P}^{\C_j}\mid=\sum_{k\in J}\mid{}^1 \mathcal{P}^{\C_k}\mid \alpha_{ik}^j.$$
It follows from Lemma \ref{lem:P1prime>0} that $S(C_k)\in \mathcal{S}^{1}$ if and only if $k\in J,$ and $$\mid{^1\mathcal{P}^{\C_k}}\mid=\mid{}^{K S(\C_k)^TK^{-1}}\mathcal{P}^{ 1}\mid,$$ where $K S(\C_k)^TK^{-1}\in\mathcal{M}$ is a basic multiplicative matrix of $S(C_k)$ with $K\in\mathrm{GL}(\k)$. Lemma \ref{lem:cijkinvariant} implies that $\alpha_{ik}^j=\alpha_{jk^*}^i$, and thus
\begin{eqnarray*}
\mid {}^{\C_i}\mathcal{P}^{\C_j}\mid=\sum_{k\in J}\mid{}^{K S(\C_k)^TK^{-1}}\mathcal{P}^{ 1}\mid \alpha_{jk^*}^i
=\sum_{t\in T}\mid{}^{\C_t}\mathcal{P}^{ 1}\mid \alpha_{jt}^i.
\end{eqnarray*}
A similar argument to the one in the proof of Proposition \ref{prop:Pcomplete} (1) shows that
\begin{eqnarray*}
(H_1/H_0)^{co H_0} = \bigoplus_{\Y\in\mathcal{P}^1}\operatorname{span}(\overline{\Y}).
\end{eqnarray*}
This means that $\mid{}^{\C_t}\mathcal{P}^{1}\mid$ equals the multiplicity of the simple comodule $V_t$ in the comodule decomposition of $(H_1/H_0)^{\text{co} H_0}$. Moreover, by Example \ref{example:H1/H0} and Proposition \ref{prop:Majidmod=YD}, $(H_1/H_0)^{\text{co} H_0}$ admits a left-left Yetter-Drinfeld module structure over $H_0$, and we can then complete the proof.
\end{proof}

\begin{example} \label{ex:32dim}\emph{(}\cite[Example 6.2]{YLL24} \emph{)}\rm
Let $H$ be the $32$-dimensional Hopf algebra generated by $z, y, t, p_1, p_2$ subject to the following relations:
\begin{eqnarray*}
    &z^{2}=1,\;\; y^{2}=1,\;\; t^{2}=1,\;\;  z y=y z,\;\;  t z=z t,\;\; t y=y t,\\
    &z p_{1}=p_{1} z,\;\; y p_{1}=p_{1} y,\;\; t p_{1}=-p_{1} t,\;\; z p_{2}=p_{2} z,\;\; y p_{2}=p_{2} y,\;\; t p_{2}=-p_{2} t,\\
    &p_{1}^{2}=0,\;\; p_{2}^{2}=0,\;\;  p_{1} p_{2}+p_{2} p_{1}=0.
\end{eqnarray*}
The coalgebra structure and antipode are given by:
\begin{eqnarray*}
    &\Delta(z)=z \otimes z,\;\; \Delta(y)=y \otimes y,\;\; \varepsilon(z)=\varepsilon(y)=1,\\
    &\Delta(t)=\frac{1}{2}\left[(1+y) t \otimes t+(1-y) t \otimes z t\right],\;\;  \varepsilon(t)=1,\\
    &S(z)=z,\;\;  S(y)=y,\;\;  S(t)=\frac{1}{2}\left[(1+y) t+(1-y) z t\right], \\
    &\Delta\left(p_{1}\right)=p_{1} \otimes 1+\frac{1}{2}\left(1+z\right) t \otimes p_{1}+\frac{1}{2}\left(1-z\right) y t \otimes p_{2},\;\;\varepsilon(p_1)=0,\\
    &\Delta\left(p_{2}\right)=p_{2} \otimes 1+\frac{1}{2}\left(1+z\right) y t \otimes p_{2}+\frac{1}{2}\left(1-z\right) t \otimes p_{1},\;\;\varepsilon(p_2)=0,\\
    &S(p_1)=-\frac{1}{4}\left[(1+y)t+(1-y)zt\right]\left[(1+z)p_1+y(1-z)p_2\right],\\
    &S(p_2)=-\frac{1}{4}\left[(1+y)t+(1-y)zt\right]\left[y(1+z)p_2+(1-z)p_1\right].
\end{eqnarray*}
Denote $E=\operatorname{span}\{t, zt, yt, zyt\}$. Then $\mathcal{S}=\{\k 1, \k  z, \k  y, \k  zy,  E\}$. We give the corresponding multiplicative matrix $\E$ of $E$, where
$$
\E=\frac{1}{2}\left(\begin{array}{cc}
t+yt&t-yt\\
zt-zyt&zt+zyt
 \end{array}\right).
$$
In this example, $\mathcal{P}=\{\X_1, \X_2, \X_3, \X_4, \X_5, \X_6, \X_7, \X_8\}$, where
$$
\X_1=\frac{1}{2}\left(\begin{array}{c}
p_1+p_2\\
p_1-p_2
 \end{array}\right),
\X_2=\frac{1}{2}\left(\begin{array}{c}
p_1z-p_2z\\
p_1z+p_2z
 \end{array}\right),
$$
$$
 \X_3=\frac{1}{2}\left(\begin{array}{c}
p_1y+p_2y\\
p_2y-p_1y
 \end{array}\right),
 \X_4=\frac{1}{2}\left(\begin{array}{c}
p_2zy-p_1zy\\
p_1zy+p_2zy
 \end{array}\right),
$$
$$
\X_5=\frac{1}{4}\left(\begin{array}{cc}
(p_1+p_2)(1+y)t+(p_1-p_2)(1-y)zt&(p_1+p_2)(1-y)t+(p_1-p_2)(1+y)zt
 \end{array}\right),
$$
$$
\X_6=\frac{1}{4}\left(\begin{array}{cc}
(p_1+p_2)(1-y)zt+(p_1-p_2)(1+y)t&(p_1+p_2)(1+y)zt+(p_1-p_2)(1-y)t
 \end{array}\right),
$$
$$
\X_7=\frac{1}{8}\left(\begin{array}{cc}
(p_1-p_2)(1+y)t-(p_1+p_2)(1-y)z t&(p_1-p_2)(1-y)t-(p_1+p_2)(1+y)z t
 \end{array}\right),
 $$
 $$
\X_8=\frac{1}{8}\left(\begin{array}{cc}
(p_1-p_2)(1-y)z t+(p_1+p_2)(1+y)t&(p_1+p_2)(1+y)z t-(p_1+p_2)(1-y)t
 \end{array}\right).
 $$
The link quiver of $H$ is shown below:
$$
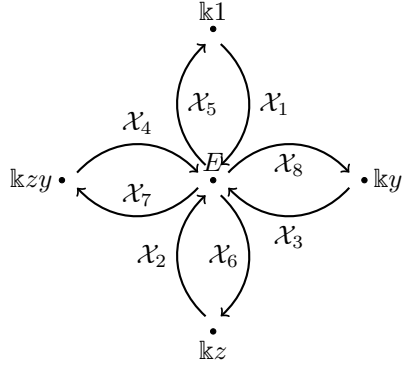
\begin{tikzpicture}
\filldraw [black] (0,0) circle (1pt) node[anchor=south]{$E$};
\filldraw [black] (0,2) circle (1pt)node[anchor=south]{$\k  1$};
\filldraw [black] (0,-2) circle (1pt) node[anchor=north]{$\k  z$};
\filldraw [black] (2,0) circle (1pt)node[anchor=west]{$\k y$};
\filldraw [black] (-2,0) circle (1pt)node[anchor=east]{$\k  zy$};
\draw[thick, ->] (-0.1,0.2) .. controls (-0.6,0.67) and (-0.6,1.34) .. node[anchor=west]{$\X_5$}(-0.1,1.8) ;
\draw[thick, ->] (0.1,1.8) .. controls (0.6,1.34) and (0.6,0.67) .. node[anchor=west]{$\X_1$}(0.1,0.2) ;
\draw[thick, ->] (-0.1,-1.8) .. controls  (-0.6,-1.34)and  (-0.6,-0.67)..  node[anchor=east]{$\X_2$}(-0.1,-0.2);
\draw[thick, ->] (0.1,-0.2) .. controls (0.6,-0.67) and (0.6,-1.34) .. node[anchor=east]{$\X_6$}(0.1,-1.8) ;
\draw[thick, ->] (0.2,0.1) .. controls (0.67,0.6) and (1.34,0.6) .. node[anchor=north]{$\X_8$}(1.8,0.1) ;
\draw[thick, ->] (1.8,-0.1) .. controls (1.34,-0.6) and (0.67,-0.6) ..node[anchor=north]{$\X_3$}(0.2,-0.1) ;
\draw[thick, ->] (-1.8,0.1) .. controls (-1.34,0.6) and (-0.67,0.6) ..  node[anchor=south]{$\X_4$}(-0.2,0.1);
\draw[thick, ->]  (-0.2,-0.1).. controls (-0.67,-0.6) and (-1.34,-0.6)  ..node[anchor=south]{$\X_7$}(-1.8,-0.1)  ;
\end{tikzpicture}
$$
Clearly, the quiver above is not a Hopf quiver, although by
Lemma \ref{lem:Hopflink=Hopfquiver} it can be realized as a generalized Hopf quiver.
\end{example}

Let $H$ be a cosemisimple coquasi-Hopf algebra, and let $\mathrm{Q}$ be a generalized Hopf quiver associated with a cosemisimple datum of $H$. The previous subsection shows that one can construct a modified generalized path coalgebra $\k(\mathrm{Q}, \mathcal{S})$. Next we prove that this modified generalized path coalgebra admits a graded coquasi-Hopf algebra structure with respect to the length grading, and that only generalized Hopf quivers admit such a coquasi-Hopf algebra structure. We emphasize that if $\k(\mathrm{Q}, \mathcal{S})$ admits a graded coquasi-Hopf algebra structure with respect to the length grading, then $\k(\mathrm{Q}, \mathcal{S})$ clearly has the dual Chevalley property.
\begin{theorem}\label{thm:generalizedHopfquiver}
Let $\mathrm{Q}=(\mathrm{Q}_0, \mathrm{Q}_1)$ be a quiver, and let $\mathcal{S}=\{C_i\mid i\in\mathrm{Q}_0\}$ be a family of simple coalgebras indexed by the vertex set of $\mathrm{Q}$. Set $H=\sum_{i\in \mathrm{Q}_0}C_i$. Then the modified generalized path coalgebra $\k(\mathrm{Q}, \mathcal{S})$ admits a graded coquasi-Hopf algebra structure with respect to the length grading if and only if $H$ is a cosemisimple coquasi-Hopf algebra and $\mathrm{Q}$ is a generalized Hopf quiver associated with a cosemisimple datum of $H$. Moreover, the set of such graded coquasi-Hopf algebra structures on $\k(\mathrm{Q}, \mathcal{S})$ is in one-to-one correspondence with the set of left-left Yetter-Drinfeld module structures over $H$ on $\k(\mathrm{Q}_1,\mathcal{S}){}^{co H}$, or equivalently, with the set of $H$-Majid bimodule structures on $\k(\mathrm{Q}_1,\mathcal{S})$.
\end{theorem}

\begin{proof}
``Only if part": If $\k(\mathrm{Q}, \mathcal{S})$ is a graded coquasi-Hopf algebra, then it follows immediately that $\k(\mathrm{Q}, \mathcal{S})_0 = H$ is a coquasi-Hopf subalgebra and is clearly cosemisimple. By Lemmas \ref{lem:k(Q,S)link} and \ref{lem:Hopflink=Hopfquiver}, we know that the link quiver of $\k(\mathrm{Q}, \mathcal{S})$ is $\mathrm{Q}$, and that this quiver is a generalized link quiver associated with a cosemisimple datum of $H$.

``If part": We set $M = \k(\mathrm{Q}_1,\mathcal{S})$ and first show that $M$ forms an $H$-Majid bimodule. From the definition of a generalized Hopf quiver, we know that
$M^{coH}$ has a left-left Yetter-Drinfeld module structure over $H$. Then, by Proposition \ref{prop:Majidmod=YD}, $M^{coH}\otimes H$ has an $H$-Majid bimodule structure, whose bicomodule structures are respectively as follows:
\begin{eqnarray*}
    &\rho _{M^{coH}\otimes H}^{L}\left( m\otimes h\right) :=m_{-1}h_{1}\otimes
(m_{0}\otimes h_{2}), & \\
&\rho _{M^{coH}\otimes H}^{R}\left( m\otimes h\right) :=(m\otimes h_{1})\otimes
h_{2}.&
\end{eqnarray*}
Therefore, the definition of a generalized Hopf quiver shows that $M$ is isomorphic to $M^{co H} \otimes H$ as H-bicomodules, which implies that $M$ admits an $H$-Majid bimodule structure, denoted by $(M, \rho_M^L, \rho_M^R, p_L, p_R)$. According to Lemma \ref{lem:cotensorcoalg}, $\k(\mathrm{Q}, \mathcal{S})\cong \operatorname{CoT}_H(M)$ as coalgebras. Next, we prove that $\operatorname{CoT}_H(M)$ admits a graded coquasi-Hopf algebra structure. We will divide this proof into several steps.\\
$\bullet$ Step 1: $\operatorname{CoT}_H(M)$ admits a graded coquasi-bialgebra structure.\\
Let $m^\prime_0$ be the composition
$$\operatorname{CoT}_H(M)\otimes \operatorname{CoT}_H(M)\xrightarrow{\pi_0\otimes \pi_0} H\otimes H\xrightarrow{m} H,$$
and let $m^\prime_1$ be the composition
$$\operatorname{CoT}_H(M)\otimes \operatorname{CoT}_H(M)\xrightarrow{\pi_0\otimes \pi_1\oplus \pi_1\otimes \pi_0} (H\otimes M)\oplus (M\otimes H)\xrightarrow{p_L\oplus p_R} M,$$
where $\pi_0$ and $\pi_1$ are the projections. For $n\geq 2, $ define $m^\prime_n$ as the composition
$$
\operatorname{CoT}_H(M)\otimes \operatorname{CoT}_H(M)\xrightarrow{\Delta^{n-1}} (\operatorname{CoT}_H(M)\otimes \operatorname{CoT}_H(M))^{\otimes n} \xrightarrow{(m^\prime_1)^{\otimes n}} M^{\otimes n}.
$$
Finally, set $m^\prime=\sum_{n\geq 0} m_n^\prime.$ Moreover, we define
$\Phi^\prime: (\operatorname{CoT}_H(M))^{\otimes 3}\rightarrow \k$ by $$\Phi^\prime (a\otimes b\otimes c):=
\begin{cases}
\Phi(a,b,c), & \text{if } a,b,c \in H; \\
0, & \text{otherwise}.
\end{cases} $$
Next, we prove that the coalgebra $\operatorname{CoT}_H(M)$, together with the unit $1_H$, and the newly defined multiplication $m^\prime$ and reassociator $\Phi^\prime$, forms a coquasi-bialgebra. Note that $m^\prime_n\mid_{M^{\Box s}\otimes M^{\Box t}}=0$ for $s+t\neq n,$ so by Lemma \ref{lemma:universal}, $m^\prime$ is a coalgebra map. Because $\varepsilon \mid_{M^{\Box n}}=0$ for all $ n\geq 1$ and $\Phi$ is convolution-invertible, $\Phi^\prime$ is also convolution-invertible and satisfies (\ref{eq:penta}) and $(\ref{eq:Phi(1)=e})$. By the definition of the multiplication on $\operatorname{CoT}_H(M)$ and the fact that $1_H\cdot x= x \cdot 1_H =x$ for all $x\in M$, one can show that (\ref{eq:1h=h1=h}) holds. In order to prove that $\operatorname{CoT}_H(M)$ satisfies the remaining condition (\ref{eq:associativity}), we set $X=(\operatorname{CoT}_H(M))^{\otimes 3}$ and consider the two maps
$$m^\prime (\id\otimes m^\prime): X\rightarrow \operatorname{CoT}_H(M)$$ and $$\Phi^{\prime} *m^\prime(m^\prime\otimes \id)* (\Phi^{\prime})^{-1}: X\rightarrow \operatorname{CoT}_H(M).$$
Since $m^\prime$ is a coalgebra map, it follows that $m^\prime (\id\otimes m^\prime)$ is also a coalgebra map. We then show that $\Phi^{\prime} *m^\prime(m^\prime\otimes \id)* (\Phi^{\prime})^{-1}$ is a coalgebra map as well. Indeed, for any $a\otimes b\otimes c\in X$, we have
\begin{eqnarray*}
 && ((\Phi^{\prime} *m^\prime(m^\prime\otimes \id)* (\Phi^{\prime})^{-1}) \otimes (\Phi^{\prime} *m^\prime(m^\prime\otimes \id)* (\Phi^{\prime})^{-1})) \Delta(a\otimes b \otimes c)\\
 &=&\Phi^{\prime}(a_1,b_1, c_1)(a_2b_2)c_2(\Phi^{\prime})^{-1}(a_3,b_3,c_3)\otimes \Phi^{\prime}(a_4,b_4,c_4)(a_5b_5)c_5(\Phi^{\prime})^{-1}(a_6, b_6,c_6)\\
 &=&\Phi^{\prime}(a_1,b_1, c_1)(a_2b_2)c_2\otimes (a_3b_3)c_3 (\Phi^{\prime})^{-1}(a_4, b_4,c_4)\\
 &=&\Delta(\Phi^{\prime} *m^\prime(m^\prime\otimes \id)* (\Phi^{\prime})^{-1})(a\otimes b\otimes c).
\end{eqnarray*}
By the constructions of $m^\prime$ and $\Phi^\prime$, we see that $$(m^\prime (\id\otimes m^\prime))_i= (\Phi^{\prime} *m^\prime(m^\prime\otimes \id)* (\Phi^{\prime})^{-1})_i$$ for $i=0,1.$ Lemma \ref{lemma:universal} implies (\ref{eq:associativity}) holds.\\
$\bullet$ Step 2: The construction of $(S^\prime,\alpha^\prime,\beta^\prime)$.\\
Suppose $(S,\alpha,\beta)$ is a coquasi-antipode of $H$. Extend $\alpha$ and $\beta$ to functions on $\operatorname{CoT}_H(M)$ by setting $\alpha(a)=\beta(a)=0$ whenever $a$ is not in $H,$ and denote these extensions by $\alpha^\prime$ and $\beta^\prime.$
Let $Y=(\operatorname{CoT}_H(M))^{cop}$, and define $S^\prime_0$ as the composition
$$
Y\xrightarrow{\pi_0} H\xrightarrow{S}H.
$$
Clearly, $S^\prime_0$ is a coalgebra map. Next we construct a map $T:M\rightarrow M$ such that the composition $$Y\xrightarrow{\pi_1}M\xrightarrow{T}M$$ is an $H$-bicomodule map, and we denote this composition by $S^\prime_1$. According to Proposition \ref{prop:Hopfmod}, we have $$ H\otimes {}^{co H}M\cong M,\;\;h\otimes m\mapsto h\cdot m.$$ By \cite[Lemma 2.8, Corollary 2.11 and Lemma 2.17]{YLL24}, we know that
\begin{eqnarray*}
{}^{co H}M = \bigoplus_{\gamma\in\Gamma}\operatorname{span}(\X_{\gamma}),
\end{eqnarray*}
where each $\X_{\gamma}$ is a non-trivial $(1, \C^{(\gamma)})$-primitive matrix over $\operatorname{CoT}_H(M)$ for some basic multiplicative matrix $\C^{(\gamma)}$ over $H$. We first specify the image of $T$ on elements of a non-trivial $(1,\C)$-primitive matrix $\X=(x_{1i})_{1\times r}$, where $\C=(c_{ij})_{r\times r}$, and then extend it to all of $M$ by considering the expressions $h\cdot x_{1i}$ for any $h\in H$ and $ 1\leq i\leq r$. For any $1\leq i\leq r$, we define
$$T(x_{1i}):=-\sum\limits_{j,k=1}^r\alpha(1)\beta(c_{kj})x_{1k}\cdot S(c_{ji}).$$
In order to extend $T$ to arbitrary elements of $M$, inspired by \cite[Proposition 3.49]{BCPV19}, we introduce the following linear maps. For any $ h, g \in H $, set
\begin{eqnarray*}
\gamma(h,g) := \Phi(S(g_2), S(h_2), h_4) \alpha(h_3) \Phi^{-1}(S(g_1)S(h_1), h_5, g_4) \alpha(g_3),
\end{eqnarray*}
and then define
\begin{eqnarray*}
f(h,g) := \Phi^{-1}(S(g_1)S(h_1), h_3g_3, S(h_5g_5)) \beta(h_4g_4) \gamma(h_2, g_2).
\end{eqnarray*}
Now we define $T: M \rightarrow M$ by specifying its action on all elements of the form $d_{jk}\cdot x_{1i}$, where $d_{jk}$ runs over all entries of a basic multiplicative matrix $\D = (d_{ij})_{s \times s}$ over $H$ and $1\leq i\leq r$, as follows:
\begin{eqnarray*}
T(d_{jk}\cdot x_{1i})&:=&\beta(1)\sum\limits_{l=1}^s\sum\limits_{t=1}^rT(x_{1t})\cdot  S(d_{jl}) f(d_{lk},c_{ti})\\
&=&-\sum\limits_{l=1}^s\sum\limits_{t,m,n=1}^r(\beta(c_{mn})x_{1m}\cdot S(c_{nt}))\cdot S(d_{jl})f(d_{lk},c_{ti}).
\end{eqnarray*}
A routine computation gives
\begin{eqnarray*}
 &&\rho^R_M(T(d_{jk}\cdot x_{1i}))\\
 &=&-\sum\limits_{l,w=1}^s\sum\limits_{t,m,n,u,v=1}^r
 (\beta(c_{mn})x_{1u}\cdot S(c_{vt}))\cdot S(d_{wl})f(d_{lk},c_{ti})\otimes (c_{um}S(c_{nv}))S(d_{jw})\\
 &=&\sum\limits_{w=1}^s (-\sum\limits_{l=1}^s\sum\limits_{t,u,v=1}^n(\beta(c_{uv})x_{1u}\cdot S(c_{vt}))\cdot S(d_{wl})f(d_{lk},c_{ti}))\otimes S(d_{jw})\\
 &=&\sum\limits_{w=1}^s T(d_{wk}\cdot x_{1s})\otimes S(d_{jw}).
\end{eqnarray*}
It follows from \cite[Proposition 3.49]{BCPV19} that
\begin{eqnarray*}
 && \rho^L_M(T(d_{jk}\cdot x_{1i}))\\&=&\sum\limits_{l,v=1}^s\sum\limits_{t,m,n,u=1}^rS(c_{ut})S(d_{vl})  \otimes (-\beta(c_{mn})(x_{1m}\cdot S(c_{nu}))\cdot S(d_{jv})f(d_{lk},c_{ti}))\\
&=&\sum\limits_{l,v,a,b=1}^s\sum\limits_{t,u,m,n,e,f=1}^r
 f^{-1}(d_{av}, c_{eu})S(c_{ut})S(d_{vl}) f(d_{lk},c_{ti}))\\
 &&\otimes (-\beta(c_{mn})x_{1m}\cdot S(c_{nf}))\cdot S(d_{jb})f(d_{ba},c_{fe})\\
 &=&\sum\limits_{a=1}^s\sum\limits_{e=1}^r S(d_{ak} c_{ei})\otimes T(d_{ja}\cdot x_{1e}).
\end{eqnarray*}
Since $S$ is a coalgebra antimorphism, we can use $S$ to equip $M$ with a new comodule structure such that $S^\prime_1$ is an $H$-bicomodule map. For $n\geq 2, $ define $S^\prime_n$ as the composition
$$
(\operatorname{CoT}_H(M))^{cop}\xrightarrow{\Delta^{n-1}} ((\operatorname{CoT}_H(M))^{cop})^{\otimes n} \xrightarrow{(S^\prime_1)^{\otimes n}} M^{\otimes n}.
$$
Set $S^\prime=\sum_{i=0}S^\prime_i$. Since $S^\prime_n\mid_{M^{\Box m}}=0$ for $m\neq n$, it follows from Lemma \ref{lemma:universal} that $$S^\prime:\operatorname{CoT}_H(M)\rightarrow \operatorname{CoT}_H(M)$$ is a coalgebra antimorphism.\\
$\bullet$ Step 3: $(S^\prime,\alpha^\prime,\beta^\prime)$ is a coquasi-antipode for $\operatorname{CoT}_H(M).$\\
By construction, (\ref{eq:Phialphabeta=e}) clearly holds, so it suffices to verify (\ref{eq:alpha,beta}). We first prove that all elements in $M$ satisfy (\ref{eq:alpha,beta}).
Indeed, \cite[Proposition 3.49]{BCPV19} yields
\begin{eqnarray*}
&&(\id*\beta^\prime*S^{\prime})(d_{jk}\cdot x_{1i})\\
&=&-\sum\limits_{l,t,u=1}^s \sum\limits_{v,m,n=1}^rd_{jl}\cdot \beta(d_{lt})((\beta(c_{mn})x_{1m}\cdot S(c_{nv}))\cdot S(d_{tu})f(d_{uk},c_{vi}))\\
&&+\sum\limits_{l,t=1}^s\sum\limits_{m,n=1}^r(d_{jl}\cdot  x_{1m})\cdot\beta(d_{lt}c_{mn})S(d_{tk}c_{ni})\\
&=&-\sum\limits_{d,l,t,u=1}^s \sum\limits_{a,b,v,m,n=1}^rd_{jl}\cdot \beta(d_{lt})\beta(c_{mn})( x_{1a}\cdot (S(c_{bv})S(d_{du}))\Phi(c_{am},S(c_{nb}), S(d_{td})) \\&& f(d_{uk},c_{vi}))+\sum\limits_{l,t=1}^s\sum\limits_{m,n=1}^r(d_{jl}\cdot x_{1m})\cdot\beta(d_{lt}c_{mn})S(d_{tk}c_{ni})\\
&=&-\sum\limits_{d,l,t,u=1}^s \sum\limits_{a,b,v,m,n=1}^rd_{jl}\cdot \beta(d_{lt})\beta(c_{mn})( x_{1a}\cdot (\sum\limits_{c,h=1}^r\sum\limits_{e,g=1}^sf(d_{de},c_{bc})S(d_{eg}c_{ch})
f^{-1}(d_{gu},c_{hv}))\\&&  \Phi(c_{am},S(c_{nb}), S(d_{td}))f(d_{uk},c_{vi}))+\sum\limits_{l,t=1}^s\sum\limits_{m,n=1}^r(d_{jl}\cdot x_{1m})\cdot\beta(d_{lt}c_{mn})S(d_{tk}c_{ni})\\
&=&-\sum\limits_{d,e,g,l,t=1}^s \sum\limits_{a,b,c,h,m,n=1}^r\beta(d_{lt})\beta(c_{mn})  f(d_{de},c_{bc})
\Phi(c_{am},S(c_{nb}), S(d_{td}))\\&&
(\sum\limits_{p,y=1}^r\sum\limits_{w,q=1}^s
(d_{jw}\cdot x_{1p})\cdot S(d_{qg}c_{yh})\Phi^{-1}(d_{wl}, c_{pa}, S(d_{eq}c_{cy})) )\varepsilon(d_{gk})\varepsilon(c_{hi})
\\&&+\sum\limits_{l,t=1}^s\sum\limits_{m,n=1}^r(d_{jl}\cdot x_{1m})\cdot\beta(d_{lt} c_{mn})S(d_{tk}c_{ni})\\
&=&-\sum\limits_{w,q =1}^s\sum\limits_{p,y=1}^r(d_{jw}\cdot x_{1p})\cdot S(d_{qk}c_{yi})(\sum\limits_{d,e,l,t=1}^s \sum\limits_{a,b,c,m,n=1}^r
\sum\limits_{u,v=1}^s\sum\limits_{g,h=1}^r f^{-1}(d_{eu},c_{cg})
\\&&\Phi^{-1}(d_{wl}, c_{pa}, S(c_{gh})S(d_{uv})) f(d_{vq}, c_{hy})\beta(d_{lt})
\Phi(c_{am},S(c_{nb}), S(d_{td}))\beta(c_{mn})  \\&&f(d_{de},c_{bc})
 )
+\sum\limits_{l,t=1}^s\sum\limits_{m,n=1}^r(d_{jl}\cdot x_{1m})\cdot\beta(d_{lt} c_{mn})S(d_{tk}c_{ni})\\
&=&-\sum\limits_{w,q =1}^s\sum\limits_{p,y=1}^r(d_{jw}\cdot x_{1p})\cdot S(d_{qk}c_{yi})(\sum\limits_{d,l,t,v=1}^s \sum\limits_{a,b,m,n,h=1}^r
\Phi^{-1}(d_{wl}, c_{pa}, S(c_{bh})S(d_{dv}))\\&&\beta(d_{lt})
\Phi(c_{am},S(c_{nb}), S(d_{td}))\beta(c_{mn}) f(d_{vq}, c_{hy}))\\&&+\sum\limits_{l,t=1}^s\sum\limits_{m,n=1}^r(d_{jl}\cdot x_{1m})\cdot\beta(d_{lt} c_{mn})S(d_{tk}c_{ni}).
\end{eqnarray*}
Note that by (\ref{eq:penta}), we have
\begin{eqnarray}\label{eq:penta2}
\Phi^{-1}(a,b_1,c_1d_1)\Phi(b_2,c_2,d_2)= \Phi(a_1b_1,c_1,d_1)\Phi^{-1}(a_2,b_2,c_2)\Phi^{-1}(a_3,b_3c_3,d_2).
\end{eqnarray}
By setting $a=d_{wl}, b=c_{pm},c=S(c_{nh}),d=S(d_{tv})$ in (\ref{eq:penta2}) and defining for any $g,h\in H$,
$$\delta(h, g) := \Phi(h_1 g_1, S(g_5), S(h_4)) \beta(h_3) \Phi^{-1}(h_2, g_2, S(g_4)) \beta(g_3),$$
 we obtain from \cite[Proposition 3.49]{BCPV19} that
\begin{eqnarray*}
&&\sum\limits_{d,l,t,v=1}^s \sum\limits_{a,b,m,n,h=1}^r
\Phi^{-1}(d_{wl}, c_{pa}, S(c_{bh})S(d_{dv}))\beta(d_{lt})
\Phi(c_{am},S(c_{nb}), S(d_{td}))\beta(c_{mn})\\&& f(d_{vq}, c_{hy})\\
&=&\sum\limits_{v=1}^s\sum\limits_{h=1}^r\delta(d_{wv}, c_{ph})f(d_{vq}, c_{hy})\\
&=&\beta(d_{wq}c_{py}).
\end{eqnarray*}
It follows that $$(\id*\beta^\prime*S^{\prime})(d_{jk}\cdot x_{1i})=0=\beta^\prime (d_{jk}\cdot x_{1i}).$$
Meanwhile, from \cite[Proposition 3.49]{BCPV19}, we also have
\begin{eqnarray*}
&&(S^\prime*\alpha^\prime*\id)(d_{jk}\cdot x_{1i})\\
&=&\sum\limits_{l,t=1}^s S(d_{jl}) \cdot \alpha(d_{lt})(d_{tk}\cdot x_{1i})
-\sum\limits_{l,t=1}^s\sum\limits_{m,n=1}^r\sum\limits_{u=1}^s\sum\limits_{v,w,y=1}^r((\beta(c_{wy})x_{1w}\cdot S(c_{yv}))\cdot S(d_{ju})\\&&f(d_{ul},c_{vm}))\cdot\alpha(d_{lt}c_{mn})(d_{tk}c_{ni})\\
&=&\sum\limits_{l,t=1}^{s}\sum\limits_{a,d=1}^s\sum\limits_{m=1}^r\alpha(d_{lt})(S(d_{al}) d_{td})\cdot x_{1m}\Phi^{-1}(S(d_{ja}),d_{dk},c_{mi})-\sum\limits_{d,l,t,u=1}^s\sum\limits_{a,b,m,n,v,w,y=1}^r \\&&\beta(c_{wy})f(d_{ul},c_{vm})\alpha(d_{lt}c_{mn})\Phi(c_{aw}, S(c_{yb}), S(d_{jd})) (x_{1a}\cdot (S(c_{bv})S(d_{du})))(d_{tk}c_{ni})\\
&=&\sum\limits_{a,d=1}^s\sum\limits_{m=1}^rx_{1m}\alpha(d_{ad}) \Phi^{-1}(S(d_{ja}),d_{dk},c_{mi})-\sum\limits_{d,l,t,u=1}^s\sum\limits_{a,b,m,n,v,w,y=1}^r \beta(c_{wy})f(d_{ul},c_{vm})\\
&&\alpha(d_{lt}c_{mn})\Phi(c_{aw}, S(c_{yb}), S(d_{jd})) (x_{1a}\cdot (\sum\limits_{g,h=1}^s\sum\limits_{c,e=1}^r f( d_{dg},c_{bc}) S(d_{gh}c_{ce})f^{-1}(d_{hu},c_{ev}) ) )\\&&(d_{tk}c_{ni})\\
&=&\sum\limits_{a,d=1}^s\sum\limits_{m=1}^rx_{1m}\alpha(d_{ad}) \Phi^{-1}(S(d_{ja}),d_{dk},c_{mi})-\sum\limits_{d,g,l,t=1}^s\sum\limits_{a,b,c,m,n,w,y=1}^r \beta(c_{wy})\alpha(d_{lt}c_{mn})\\
&&\Phi(c_{aw}, S(c_{yb}), S(d_{jd}))f( d_{dg},c_{bc})(x_{1a}\cdot S(d_{gl}c_{cm}))(d_{tk}c_{ni})\\
&=&\sum\limits_{a,d=1}^s\sum\limits_{m=1}^rx_{1m}\alpha(d_{ad}) \Phi^{-1}(S(d_{ja}),d_{dk},c_{mi})-\sum\limits_{d,f,g,l,q,t=1}^s\sum\limits_{a,b,c,e,m,n,u,w,y,z=1}^r \beta(c_{wy})\\
&&\alpha(d_{lt}c_{mn})\Phi(c_{aw}, S(c_{yb}), S(d_{jd}))f( d_{dg},c_{bc})(
 (x_{1e}\cdot (S(d_{fl}c_{um})(d_{tq}c_{nz})))\\
&&\Phi(c_{ea}, S(d_{gf}c_{cu}), d_{qk}c_{zi}) )\\
&=&
\sum\limits_{a,d=1}^s\sum\limits_{m=1}^rx_{1m}\alpha(d_{ad}) \Phi^{-1}(S(d_{ja}),d_{dk},c_{mi})-\sum\limits_{d,f,g,q=1}^s\sum\limits_{a,b,c,e,u,w,y,z=1}^r x_{1e} \beta(c_{wy})\\
&&\Phi(c_{aw}, S(c_{yb}), S(d_{jd}))f( d_{dg},c_{bc})
 \alpha(d_{fq}c_{uz})\Phi(c_{ea}, S(d_{gf}c_{cu}), d_{qk}c_{zi}) \\
 &=&
\sum\limits_{a,d=1}^s\sum\limits_{m=1}^rx_{1m}\alpha(d_{ad}) \Phi^{-1}(S(d_{ja}),d_{dk},c_{mi})-\sum\limits_{d,f,g,q=1}^s\sum\limits_{a,b,c,e,u,w,y,z=1}^r x_{1e} \beta(c_{wy})\\
&&\Phi(c_{aw}, S(c_{yb}), S(d_{jd}))f( d_{dg},c_{bc})
 \alpha(d_{fq}c_{uz})(\sum\limits_{p,h=1}^s\sum\limits_{m,n=1}^r
f^{-1}(d_{gp},c_{cm})
\\&&\Phi(c_{ea}, S(c_{mn})S(d_{ph}), d_{qk}c_{zi}) f(d_{hf},c_{nu}))\\
 &=&
\sum\limits_{a,d=1}^s\sum\limits_{m=1}^rx_{1m}\alpha(d_{ad}) \Phi^{-1}(S(d_{ja}),d_{dk},c_{mi})-\sum\limits_{d,h,q=1}^s\sum\limits_{a,b,e,n,w,y,z=1}^r x_{1e}\\
&&\Phi(c_{ea}, S(c_{bn})S(d_{dh}), d_{qk}c_{zi})
\Phi(c_{aw}, S(c_{yb}), S(d_{jd}))\beta(c_{wy})\gamma(d_{hq},c_{nz})
\end{eqnarray*}
Set $h=d_{jk}, g=c_{mi}$, it follows from (\ref{eq:penta}), (\ref{eq:alpha,beta}) and (\ref{eq:Phialphabeta=e}) that
\begin{eqnarray*}
&&\Phi(g_1, S(g_{5})S(h_2),h_4g_7)\Phi(g_2,S(g_4),S(h_1))\beta(g_3)\gamma(h_3, g_6)\\
&=&\Phi^{-1}(S(g_6),S(h_3),h_9g_{11})\Phi(g_1,S(g_5),S(h_2)(h_{10}g_{12}))\Phi(g_2S(g_4),S(h_1),h_{11}g_{13})\beta(g_3)\\
&&\Phi(S(g_8), S(h_5),h_7)\alpha(h_{6})\Phi^{-1}(S(g_7)S(h_4),h_8,g_{10})\alpha(g_9)\\
&=&\Phi^{-1}(S(g_4),S(h_2),h_8g_9)\Phi(g_1,S(g_3),S(h_1)(h_9g_{10}))\beta(g_2)\Phi(S(g_6),S(h_4),h_6)\alpha(h_5)\\
&&\Phi^{-1}(S(g_5)S(h_3),h_7,g_8)\alpha(g_7)\\
&=&\Phi^{-1}(S(g_4),S(h_3)h_5,g_6)\Phi^{-1}(S(h_2),h_6,g_7)\Phi(g_1,S(g_3),S(h_1)(h_7g_8))\beta(g_2)\alpha(h_4)\alpha(g_5)\\
&=&\Phi^{-1}(S(h_2),h_4,g_5)\Phi(g_1,S(g_3),S(h_1)(h_5g_6))\beta(g_2)\alpha(h_3)\alpha(g_4)\\
&=&\Phi^{-1}(S(h_4),h_6,g_5)\Phi(S(h_3),h_7,g_6)\Phi(g_1,S(g_3), (S(h_2)h_8)g_7)\Phi^{-1}(S(h_1),h_9,g_8)\\
&&\beta(g_2)\alpha(g_4)\alpha(h_5)\\
&=&\Phi(g_1,S(g_3),(S(h_2)h_4)g_5)\Phi^{-1}(S(h_1),h_5,g_6)\beta(g_2)\alpha(g_4)\alpha(h_3)\\
&=&\alpha(h_2)\Phi^{-1}(S(h_1),h_3,g).
\end{eqnarray*}
As a result,  $$(S^\prime*\alpha^\prime*\id)(d_{jk}\cdot x_{1i})=0=\alpha^{\prime}(d_{jk}\cdot x_{1i}).$$ Therefore, all elements in $M$ satisfy equation (\ref{eq:alpha,beta}). For any $\sum m^1\otimes m^2\otimes \cdots \otimes m^n\in M^{\Box n}$, we have
\begin{eqnarray*}
   && (\id*\beta^\prime*S^\prime)(\sum m^1\otimes m^2\otimes \cdots \otimes m^n)\\
   &=& \sum m^1_{-2} \beta( m_{-1}^1) S^\prime(m_{0}^1\otimes m^2\otimes \cdots \otimes m^n) +\sum m^1_0 \beta(m^1_1)S^\prime(m^2\otimes \cdots \otimes m^n)+\cdots\\&&+\sum (m^1\otimes m^2\otimes \cdots \otimes m^{n}_0)\beta( m^{n}_1) S( m^n_2)
   \\
   &=& \sum m_{-n-1}^1 S^\prime_1(m^n) \otimes m_{-n}^1 S^\prime_1(m^{n-1})\otimes \cdots \otimes m_{-3}^1 S^\prime_1(m^2) \otimes m_{-2}^1\beta(m_{-1}^1) S^\prime_1(m_0^1)\\
   &&+ \sum m_{-n+1}^1 S^\prime_1(m^n)\otimes m^1_{-n+2}S^\prime_1(m^{n-1}) \otimes \cdots \otimes m_{-1}^1 S^\prime_1(m^2_0) \otimes m_{0}^1\beta(m_1^1) S(m^2_{-1}) \\&&+
   \sum m_0^1 S(m^n_1)\otimes m_1^1 S^\prime_1(m^n_0) \otimes\cdots \otimes m_{n-2}^1 S^\prime_1 m^3\otimes  m^1_{n-1} \beta(m_{n}^1) S^\prime (m^2)+\cdots \\
   &&+\sum m^1 S(m_{n+1}^n) \otimes m^2 S(m_{n}^1)\otimes m^{n-1} S(m^n_{3}) \otimes m^n_0 \beta(m^n_1)S(m^n_{2}).
\end{eqnarray*}
By the definition of $M^{\Box n}$, one can show that $$(\id*\beta^\prime*S^\prime)(\sum m^1\otimes m^2\otimes \cdots \otimes m^n)=0=\beta^\prime(\sum m^1\otimes m^2\otimes \cdots \otimes m^n).$$
A similar argument yields
$$(S^\prime*\alpha^\prime*\id)(\sum m^1\otimes m^2\otimes \cdots \otimes m^n)=0=\alpha^\prime(\sum m^1\otimes m^2\otimes \cdots \otimes m^n).$$ This completes the proof.
\end{proof}

In the Hopf algebra case, if the coradical $H_0$ of a bialgebra $H$ is a Hopf subalgebra, then by \cite[Lemma 5.2.10]{Mon93}, the bialgebra $H$ also admits an antipode. However, in the coquasi-Hopf algebra case, the proof of \cite[Lemma 5.2.10]{Mon93} fails. In fact, even if $\alpha$ and $\beta$ are convolution invertible, we cannot obtain the above conclusion by mimicking the proof of \cite[Lemma 5.2.10]{Mon93}, let alone the fact that we do not know whether $\alpha$ and $\beta$ are convolution invertible.

\begin{remark}\rm\label{rm:multi}
For the reader's convenience, we now describe explicitly the multiplication on $\operatorname{CoT}_{H}(\k(\mathrm{Q}_1, \mathcal{S}))$ from the proof of Theorem~\ref{thm:generalizedHopfquiver}, as it applies to any $\mathcal{M}$-path $c_{i1}^{t(\alpha_n)}\alpha_n\alpha_{n-1}\cdots \alpha_1 c_{1j}^{s(\alpha_{1})}$ of length $n\geq 1$ and any $\mathcal{M}$-path $c_{k1}^{t(\beta_m)}\beta_m\beta_{m-1}\cdots \beta_1 c_{1l}^{s(\beta_{1})}$ of length $m\geq 1$.
Let $D_l^n$ denote the set of all $n$-sequences consisting of $(n-l)$ zeros and $l$ ones. Then $\mid D_l^n\mid = \binom{n}{l}$. Given $d=(d_1,d_2,\cdots,d_{m+n}) \in D_{n}^{m+n}$, let $\bar{d} \in D^{n+m}_{m}$ be the complement sequence obtained from $d$ by replacing each $0$ with $1$ and each $1$ with $0$. Note that $\Delta^{m+n-1}(c_{i1}^{t(\alpha_n)}\alpha_n\alpha_{n-1}\cdots \alpha_1 c_{1j}^{s(\alpha_{1})})$ contains a maximal component as a summand belonging to $\k(\mathrm{Q}_{d_1}, \mathcal{S})\otimes \k(\mathrm{Q}_{d_2}, \mathcal{S})\otimes \cdots \otimes \k(\mathrm{Q}_{d_{m+n}}, \mathcal{S})$, which we denote by $d(c_{i1}^{t(\alpha_n)}\alpha_n\alpha_{n-1}\cdots \alpha_1 c_{1j}^{s(\alpha_{1})})$.
Then we have
\begin{eqnarray*}
&&(c_{i1}^{t(\alpha_n)}\alpha_n\alpha_{n-1}\cdots \alpha_1 c_{1j}^{s(\alpha_{1})})\cdot (c_{k1}^{t(\beta_m)}\beta_m\beta_{m-1}\cdots \beta_1 c_{1l}^{s(\beta_{1})})\\
&=&\sum\limits_{d\in D_{n}^{m+n}} d(c_{i1}^{t(\alpha_n)}\alpha_n\alpha_{n-1}\cdots \alpha_1 c_{1j}^{s(\alpha_{1})}) \cdot^\prime \bar{d}(c_{k1}^{t(\beta_m)}\beta_m\beta_{m-1}\cdots \beta_1 c_{1l}^{s(\beta_{1})}),
\end{eqnarray*}
where $\cdot^\prime$ denotes the multiplication analogous to that on the tensor product of algebras induced by the Majid bimodule action. In particular, when $\k(\mathrm{Q},\mathcal{S})$ is pointed, i.e., $H = \k G$, the multiplicative formula above agrees with the one given in \cite[Subsection 2.4]{HLY11}.
\end{remark}
For a generalized Hopf quiver $\mathrm{Q}$, a coquasi-Hopf subalgebra of $\k(\mathrm{Q}, \mathcal{S})$ is called \textit{large} if it contains $ \k(\mathrm{Q}_0, \mathcal{S})\oplus \k(\mathrm{Q}_1, \mathcal{S})$.
Next, we can give the generalized dual Gabriel's theorem for coquasi-Hopf algebras with the dual Chevalley property.
\begin{theorem}\label{thm:Gabriel}
Let $H$ be a coquasi-Hopf algebra over $\k$ with the dual Chevalley property, and let $\mathrm{Q}(H)=(\mathcal{S},\mathcal{P})$ be its link quiver. Then there exists a coquasi-Hopf algebra structure on $\k(\mathrm{Q}(H),\mathcal{S})$ such that $\operatorname{gr}(H)$ can be embedded into it as a large coquasi-Hopf subalgebra.
\end{theorem}
\begin{proof}
Lemma \ref{lem:Hopflink=Hopfquiver} yields $\mathrm{Q}(H)$ is a generalized Hopf quiver associated with a cosemisimple datum of $H_0$. Then, by Theorem \ref{thm:generalizedHopfquiver}, $\k(\mathrm{Q}(H),\mathcal{S})$ admits a coquasi-Hopf algebra structure whose multiplication is induced by the $H_0$-Majid-bimodule structure on $\k(\mathrm{Q}_1,\mathcal{S})$.
We only need to prove that the $\psi$ constructed in the proof of Proposition \ref{prop:CoalgGabriel} is a coquasi-Hopf algebra morphism. Consider the coalgebra maps $\psi m_{\operatorname{gr}H}, m^\prime(\psi\otimes \psi), \psi S_{\operatorname{gr}H}$ and $S^\prime \psi $. Using  Lemma \ref{lemma:universal}, one can show that $\psi m_{\operatorname{gr}H}= m^\prime(\psi\otimes \psi)$ and $\psi S_{\operatorname{gr}H}=S^\prime \psi$. Moreover, from the construction of
 $\alpha^\prime, \beta^\prime, \Phi^\prime$, it follows that  $\psi$ is a coquasi-Hopf algebra morphism.
\end{proof}

In particular, the above conclusions also yield the following results in the Hopf algebra setting, which can be regarded as a generalization of \cite[Theorem 3.3]{CR02} and \cite[Theorem 4.5]{VZ04}.
\begin{corollary}\label{coro:Hopfversion}
\begin{itemize}
\item[(1)] Let $\mathrm{Q}=(\mathrm{Q}_0, \mathrm{Q}_1)$ be a quiver, and let $\mathcal{S}=\{C_i\mid i\in\mathrm{Q}_0\}$ be a family of simple coalgebras indexed by the vertex set of $\mathrm{Q}$.
Then the modified generalized path coalgebra $\k(\mathrm{Q}, \mathcal{S})$ admits a graded Hopf algebra structure with respect to the length grading if and only if $H = \sum_{i\in \mathrm{Q}_0} C_i$ is a cosemisimple Hopf algebra and $\mathrm{Q}$ is a generalized Hopf quiver associated with a cosemisimple datum of $H$. Moreover, such structures are in one-to-one correspondence with the set of left-left Yetter-Drinfeld module structures over $H$ on $\k(\mathrm{Q}_1,\mathcal{S}){}^{co H}$, or equivalently, with the set of $H$-Hopf bimodule structures on $\k(\mathrm{Q}_1,\mathcal{S})$.
\item[(2)]
Let $H$ be a Hopf algebra over $\k$ with the dual Chevalley property, and let $\mathrm{Q}(H)=(\mathcal{S},\mathcal{P})$ be its link quiver. Then there exists a Hopf algebra structure on $\k(\mathrm{Q}(H),\mathcal{S})$ such that $\operatorname{gr}(H)$ can be embedded into it as a large Hopf subalgebra.
\end{itemize}
\end{corollary}

\begin{remark}\rm
Let $H$ be a cosemisimple coquasi-Hopf algebra and $V$ a finite-dimensional left-left Yetter-Drinfeld module over $H$. Consider the bosonization $B(V)\#H$ of $H$ and the Nichols algebra $B(V)$ of $V$. Then, by Theorem \ref{thm:Gabriel}, there exists a generalized Hopf quiver $\mathrm{Q}$ associated with a cosemisimple datum $(H, \{C_i\}_{i\in I}, (n_1, n_2, \cdots), \{\alpha_{ij}^t\}_{i,j,t\in I})$ such that $B(V)\#H$ is a large coquasi-Hopf subalgebra of $\k(\mathrm{Q},\mathcal{S})$, where $(n_1,n_2,\cdots)$ is the dimension vector of $V$. Moreover, if there exists a left-left Yetter-Drinfeld module $U$ over $H$ with $\underline{\dim}(U)=(0,0,\cdots , n_k,0,0,\cdots)$, then by Corollary \ref{coro:cpdspecial} (4) and Theorem \ref{thm:generalizedHopfquiver}, both $C_k$ and $S(C_k)$ lie in the center of $\mathbb{Z}\mathcal{S}$.
\end{remark}

Combining Proposition \ref{prop:(H1)0} and Theorem \ref{thm:generalizedHopfquiver}, we can characterize when a generalized Hopf quiver is connected.
\begin{corollary}\label{coro:quiverconnect}
Let $H$ be a cosemisimple coquasi-Hopf algebra over $\k$ with a cosemisimple datum $(H, \{C_i\}_{i\in I}, (n_1, n_2, \cdots), \{\alpha_{ij}^t\}_{i,j,t\in I})$, and let $\mathrm{Q}$ be a generalized Hopf quiver with vertex set $\{C_i \mid i \in I\}$, and, for all $i, j \in I$, $\sum_{k\in I}n_k \alpha_{jk}^i$ arrows from $C_j$ to $C_i$, where $(n_1, n_2, \cdots) \in \Lambda_H$. Then $\mathrm{Q}$ is connected if and only if $H$ is generated by $ \bigcup_{i\in I, n_i\neq 0} C_i\oplus S(C_i)$.
\end{corollary}

Let $H$ be a coquasi-Hopf algebra over $\k$. Recall from \cite[Definition 2.3]{HLYY20} that a convolution-invertible linear map
$J:\;H\otimes H\to \k$
is called a \textit{twisting} on $H$ if
$J(h,1)=\varepsilon(h)=J(1,h)$
for all $h\in H$. Given a twisting $J$, we can define
a new coquasi-Hopf algebra $H^{J}$ which coincides with $H$ as a coalgebra, while the multiplication $``\circ^J"$ on $H^{J}$ is given by
\begin{equation}
 a\circ^J b:=J(a_1,b_1)a_2b_2J^{-1}(a_3,b_3)
\end{equation}
for all $a,b\in H$. The reassociator $\Phi^{J}$ and the coquasi-antipode $(S^J,\alpha^J,\beta^{J})$ are given as:
$$\Phi^J(a,b,c)=J(b_1,c_1)J(a_1,b_2c_2)\Phi(a_2,b_3,c_3)J^{-1}(a_3b_4,c_4)J^{-1}(a_4,b_5),$$
$$S^J=S,\;\;\;\;\alpha^J(a)=J^{-1}(S(a_1),a_3)\alpha(a_2),\;\;\;\;\beta^J(a)=J(a_1,S(a_3))\beta(a_2)$$
for all $a,b,c \in H$. Two coquasi-Hopf algebras $H$ and $H^\prime$ are called \textit{twist equivalent} if there is a twisting $J$ on $H$ such that we have a coquasi-Hopf algebra isomorphism
$H^{J}\cong H^\prime.$ A coquasi-Hopf algebra $H$ is called \textit{genuine} if it is not twist equivalent to a Hopf algebra.
\begin{example}\rm\label{ex:Ku}
According to \cite[Theorem 6.1]{NS08}, there are exactly four twisted-inequivalent $8$-dimensional quasi-Hopf algebras over $\mathbb{C}$ with five simple objects $\{a_1, a_2, a_3, a_4, m\}$ and fusion rules: $$\{a_1, a_2, a_3, a_4\} \cong \mathbb{Z}_2 \times \mathbb{Z}_2,\;\; m^2 \cong \sum\limits_{i=1}^4 a_i,\;\; m a_i \cong a_i m \cong m,\;\; i=1,\cdots,4.$$ Their representation categories are Tambara-Yamagami categories. These algebras are $Q_8$, $D_8$, the Kac algebra $K$, and its twist version $K_u$ (see \cite[Section 6]{NS08}). Hence $(K_u)^*$ is a genuine cosemisimple coquasi-Hopf algebra.
\end{example}
\begin{corollary}\label{coro:genuine}
Given a genuine cosemisimple coquasi-Hopf algebra $H$ and a generalized Hopf quiver $\mathrm{Q}$ associated with a cosemisimple datum of $H$, one can obtain a genuine coquasi-Hopf algebra structure on $\k(\mathrm{Q},\mathcal{S})$.
\end{corollary}
\begin{proof}
Suppose, for contradiction, that $\k(\mathrm{Q},\mathcal{S})$ with the coquasi-Hopf algebra structure constructed in the proof of Theorem \ref{thm:generalizedHopfquiver} is not genuine. Then there exist a twisting $J$ and a Hopf algebra $H^\prime$ such that $\k(\mathrm{Q},\mathcal{S})^{J}\cong H^\prime.$ Define $J^\prime:=J\mid_{H\otimes H}$. Then $H^{J^\prime}$ is isomorphic to $H^\prime_0$, which is a contradiction.
\end{proof}

The bijectivity of the coquasi-antipode constructed in the proof of Theorem \ref{thm:generalizedHopfquiver} is proved as follows.
\begin{proposition}\label{prop:k(Q,S)antipodebijective}
Let $H$ be a cosemisimple coquasi-Hopf algebra over $\k$ with a coquasi-antipode $(S, \alpha, \beta)$, and let $\mathrm{Q}$ be a generalized Hopf quiver associated with a cosemisimple datum of $H$. Then the coquasi-antipode $(S^\prime, \alpha^\prime, \beta^\prime)$ of $\k(\mathrm{Q}, \mathcal{S})$ constructed in the proof of Theorem \ref{thm:generalizedHopfquiver} is bijective. Moreover, the category of finite-dimensional right $\k(\mathrm{Q}, \mathcal{S})$-comodules forms a tensor category.
\end{proposition}
\begin{proof}
By Lemma \ref{lem:cosstensor}, $S$ is bijective. For our purpose, we next construct a coalgebra map $S^{\prime\prime}: (\operatorname{CoT}_H(M))^{cop}\rightarrow \operatorname{CoT}_H(M)$ such that $S^{\prime\prime}S^\prime=S^\prime S^{\prime\prime}=\id.$ Let $X=(\operatorname{CoT}_H(M))^{cop}$, and define $S^{\prime\prime}_0$ as the composition
$$
X\xrightarrow{\pi_0} H\xrightarrow{S^{-1}}H.
$$
Clearly, $S^{\prime\prime}_0$ is a coalgebra map. We then construct a map $T^\prime:M\rightarrow M$ such that the composition $$X\xrightarrow{\pi_1}M\xrightarrow{T^\prime}M$$ is an $H$-bicomodule map, and we denote this composition by $S^{\prime\prime}_1$. Analogous to the construction of $S^\prime_1$, we first define the values of $T^\prime$ on a non-trivial $(\D,1)$-primitive matrix $\Y=(y_{i1})_{s\times 1}$, where $\D=(d_{ij})_{s\times s}$, and then extend it to the whole $M.$ For any $1\leq i\leq s$, we define
$$T^\prime(y_{i1}):=-\sum\limits_{j,k=1}^s\beta(1)\alpha(S^{-1}(d_{kj})) y_{j1}\cdot S^{-1}(d_{ik}).$$
It is straightforward to show that $(T^\prime (y_{i1}))_{1\times s}$ is a non-trivial $(1,S^{-1}(\D)^T)$-primitive matrix.
Through calculation, we know that
\begin{eqnarray*}
    TT^\prime(y_{i1})&=&-\sum\limits_{l,t=1}^r\alpha(1)\beta(S^{-1}(d_{tl})) T^\prime(y_{l1}) d_{it}\\
    &=&-\sum\limits_{l,t=1}^r\alpha(1)\beta(S^{-1}(d_{tl})) (-\sum\limits_{j,k=1}^r\beta(1)\alpha(S^{-1}(d_{kj})) y_{j1}\cdot S^{-1}(d_{lk})) d_{it}\\
&=&\sum\limits_{j,k,l,m.n,p,t=1}^r\beta(S^{-1}(d_{tl}))\alpha(S^{-1}(d_{kj})) \Phi^{-1}(d_{jm}, S^{-1}(d_{nk}),d_{ip})y_{m1}(S^{-1}(d_{ln})d_{pt}).
\end{eqnarray*}
Note that $(H^{op}, m^{op}, \mu, \Delta, \varepsilon, (\Phi^{-1})^{321}, S^{-1}, \alpha S^{-1}, \beta S^{-1} )$ remains a coquasi-Hopf algebra. Then (\ref{eq:alpha,beta}) and (\ref{eq:Phialphabeta=e}) yield $$TT^\prime (y_{i1})=y_{i1}.$$
A similar argument shows that $$T^\prime T (x_{1i})=x_{1i},$$
where $x_{1i}$ runs over all entries of a non-trivial $(1,\C)$-primitive matrix $\X=(x_{1i})_{1\times r}$. Proposition \ref{prop:Hopfmod} implies $ M^{co H} \otimes H\cong M$ via the map $m\otimes h\mapsto m\cdot h.$
This means that we only need to define $T^\prime : M \rightarrow M$ on all elements of the form $y_{i1} \cdot e_{jk}$, where $e_{jk}$ runs over all entries of a basic multiplicative matrix $\E = (e_{ij})_{t \times t}$ over $H$ and $1 \leq i \leq s$ as follows:
\begin{eqnarray*}
T^\prime (y_{i1} \cdot e_{jk})&:=& \sum\limits_{u=1}^t \sum\limits_{l=1}^s \alpha(1) f^{-1}(S^{-1}(e_{ju}), S^{-1}(d_{il})) S^{-1}(e_{uk})\cdot T^\prime (y_{l1}).
\end{eqnarray*}
It follows that
\begin{eqnarray*}
TT^\prime (y_{i1} \cdot e_{jk}) &=& \sum\limits_{u=1}^t \sum\limits_{l=1}^s \alpha(1)f^{-1}(S^{-1}(e_{ju}), S^{-1}(d_{il})) T(S^{-1}(e_{uk})\cdot T^\prime (y_{l1})) \\
&=&\sum\limits_{u,m=1}^t \sum\limits_{l,n=1}^sf^{-1}(S^{-1}(e_{ju}), S^{-1}(d_{il})) f(S^{-1} (e_{um}), S^{-1}(d_{ln}) )y_{n1}\cdot e_{mk}\\
&=&y_{i1} \cdot e_{jk}.
\end{eqnarray*}
By a similar argument, we can show that $$T^\prime T(d_{jk}\cdot x_{1i})=d_{jk}\cdot x_{1i},$$
where $d_{jk}$ runs over all entries of a basic multiplicative matrix $\D = (d_{ij})_{s \times s}$ over $H$, and $x_{1i}$ runs over all entries of a non-trivial $(1,\C)$-primitive matrix $\X=(x_{1i})_{1\times r}$.
This means that $TT^\prime=T^\prime T=\id.$
From Lemma \ref{lemma:universal}, one can get a coalgebra map $S^{\prime\prime}:  (\operatorname{CoT}_H(M))^{cop}\rightarrow \operatorname{CoT}_H(M)$ such that $S^{\prime\prime}S^\prime=S^{\prime }S^{\prime\prime}=\id.$
\end{proof}
It is well known that the antipode of a Hopf algebra with the dual Chevalley property is necessarily bijective; see \cite[Corollary 3.6]{Rad77}. However, this argument does not extend to the coquasi-Hopf algebra setting.
But as a corollary of Proposition \ref{prop:k(Q,S)antipodebijective}, the coquasi-antipode of a coquasi-Hopf algebra with the dual Chevalley property is always injective.
\begin{corollary}\label{coro:Sbijective}
Let $H$ be a (not necessarily finite-dimensional) coquasi-Hopf algebra over $\k$ with the dual Chevalley property, and let $(S,\alpha,\beta)$ be a coquasi-antipode. Then $S$ is injective.
\end{corollary}
\begin{proof}
According to Theorem \ref{thm:Gabriel} and Proposition \ref{prop:k(Q,S)antipodebijective}, $\operatorname{gr}H$ admits a coquasi-antipode that is bijective on $H_0\oplus H_1/H_0$. This means that the $S\mid_{H_1}$ is injective. Then by \cite[Theorem 5.3.1]{Mon93}, $S$ is injective.
\end{proof}

\begin{example}\label{ex:DwGquiver}\rm
Let $(D^\omega(G))^*$ be the cosemisimple coquasi-Hopf algebra defined in Example \ref{ex:DwG} over an algebraically closed field of characteristic $0$. \cite[Theorem 9.4]{MN01} tells us that if $G$ is of odd order, then $(D^\omega(G))^*$ is not genuine. However, this statement does not hold for groups of even order: \cite[Example 9.5]{MN01} provides the simplest example of a genuine twisted quantum double, namely $(D^\omega(\mathbb{Z}_2))^*$, where $\omega$ is the nontrivial 3-cocycle on $\mathbb{Z}_2$. In \cite[Theorem 4.1]{LN14}, the authors provided a necessary and sufficient condition for $(D^\omega(G))^*$ to be genuine when $G$ is abelian and $\omega$ is an abelian cocycle, using the total Frobenius-Schur indicator.
Moreover, by \cite[Proposition 10.22]{BCPV19}, $(D^\omega(G))^*$ is a coquasitriangular coquasi-Hopf algebra. Proposition \ref{prop:triangularHopfquiver}  yields that for any finite tuple $(n_1, n_2, \cdots) \in \mathbb{Z}_+^{(I)}$, where $|I|$ equals the number of simple subcoalgebras of $(D^\omega(G))^*$, we can construct a generalized Hopf quiver $Q$. If $(D^\omega(G))^*$ is genuine, then by Theorem \ref{thm:generalizedHopfquiver} and Corollary \ref{coro:genuine}, one can obtain a genuine coquasi-Hopf algebra $H$ whose coradical is $(D^\omega(G))^*$. According to Example \ref{ex:DwG} and Proposition \ref{prop:k(Q,S)antipodebijective}, $\mathcal{M}^H$ is a tensor category that contains the module category of the fixed-point subalgebra $V^G$ of a holomorphic vertex operator algebra $V$ as a semisimple tensor subcategory. Theorem \ref{thm:Gabriel} tells us that the tensor category we constructed has a certain maximality property in some sense.
\end{example}

\subsection{Some remarks}\label{subsection4.3}
We conclude this section with some remarks.
\begin{itemize}
\item[(1)]The so-called quantum shuffle Hopf algebras are precisely the cotensor Hopf algebras of a Hopf bimodule $M$ over a Hopf algebra $H$. Let $U_q^{\geq 0}(\mathfrak{g})$ denote the upper triangular part of the quantized enveloping algebra corresponding to a symmetrizable Cartan matrix. Rosso \cite{Ros98} proved that $U_q^{\geq 0}(\mathfrak{g})$ is isomorphic, as a Hopf algebra, to the subalgebra generated by elements of degree $0$ and degree $1$ of the cotensor Hopf algebra associated with an appropriate Hopf bimodule over the group algebra of $\mathbb{Z}^n$.
    When both $H$ and $M$ are finite-dimensional, the structure of the cotensor Hopf algebra becomes particularly transparent. In this case, $H^*$ naturally inherits a Hopf algebra structure, and $M^*$ becomes a Hopf bimodule over $H^*$. Moreover, the graded dual of the cotensor Hopf algebra $\operatorname{CoT}_H(M)$ is isomorphic to the tensor Hopf algebra $T_{H^*}(M^*)$, where
    $$T_{H^*}(M^*) = H^*\oplus M^* \oplus (M^* \otimes_{H^*} M^*)\oplus \cdots \oplus (M^*)^{\otimes_H^n} \oplus \cdots.$$
    In particular, Hopf algebra structures can be constructed via covering quivers \cite{GS98} (using the path algebra structure) and Hopf quivers (using the path coalgebra structure), by taking $T_H(M)$ and $\operatorname{CoT}_H(M)$, respectively.
    Let $\mathrm{Q}=(\mathrm{Q}_0, \mathrm{Q}_1)$ be a finite generalized Hopf quiver associated with a cosemisimple datum of a finite-dimensional cosemisimple Hopf algebra $H$, and let $M=\k(\mathrm{Q}_1,\mathcal{S})$. Corollary \ref{coro:Hopfversion} allows us to endow $\k(\mathrm{Q}, \mathcal{S})$ with a Hopf algebra structure with the dual Chevalley property, namely that arising from $\operatorname{CoT}_H(M)$. Moreover, every coradically graded Hopf algebra (with the dual Chevalley property) can be realized as a large Hopf subalgebra of such a Hopf algebra. Considering the graded dual, we can obtain examples of (infinite-dimensional) radically graded Hopf algebras (with the Chevalley property) that are of interest in \cite{HQWZ26}.

     Moreover, Bulacu \cite{Bul20} extended the tensor Hopf algebra construction of Nichols \cite{Nic78} to the quasi-Hopf algebra context. Given a quasi-Hopf algebra $H$ and a quasi-Hopf $H$-bimodule $M$, Bulacu first endowed $T_H(M)$ with a quasi-Hopf algebra structure. He then proved that this quasi-Hopf algebra is isomorphic to the biproduct quasi-Hopf algebra of $T(V)$ and $H$, where $V$ is a suitable set of coinvariants of $M$ and $T(V)$ is the tensor Hopf algebra of $V$ built inside the braided monoidal category of left $H$-Yetter-Drinfeld modules. For the definitions of quasi-Hopf $H$-bimodules and left $H$-Yetter-Drinfeld modules over a quasi-Hopf algebra $H$, we refer the reader to \cite{BCPV19}. To the best of our knowledge, the parallel general results concerning $\operatorname{Cot}_H(M)$ have not been extended to the case of coquasi-Hopf algebras unless $H$ is a pointed cosemisimple coquasi-Hopf algebra \cite{Hua09}, and the proof method in \cite{Hua09} does not work for the non-pointed case. In Theorem \ref{thm:generalizedHopfquiver}, we actually describe the coquasi-Hopf coalgebra structure of $\operatorname{Cot}_H(M)$ when $H$ is a cosemisimple coquasi-Hopf algebra, and we believe that the proof method can be analogously extended to arbitrary coquasi-Hopf algebras $H$. Since this is unrelated to the main theme of the present paper, we will focus on the general theory of $\operatorname{Cot}_H(M)$ in a future work.
\item[(2)]The quiver approach provides a useful tool for characterizing indecomposable objects in comodule categories. Chen and Zhang \cite{CZ07} embedded the quantized algebra $U_q(sl_2)$ into a path coalgebra and subsequently used quiver method to describe its comodules, where $q$ is not a root of unity. Similarly, we can embed a coradically graded coquasi-Hopf algebra into a modified generalized path coalgebra in order to describe its comodules. Note that such modified generalized path coalgebra is equivalent to the path coalgebra $\k \mathrm{Q}(H).$ We briefly describe how to characterize comodules of Loewy length $2$ here.

    Let $H$ be a coradically graded coquasi-Hopf algebra, and let $\mathrm{Q}(H)=(\mathcal{S},\mathcal{P})$ be its link quiver. Now let $M = (m_1, m_2, \dots, m_n)$, where $\{m_i \mid 1 \leq i \leq n\}$ forms a basis of a right $H$-comodule. The key observation is that $$\rho^R(M)=M\;\widetilde{\otimes}\;\B:=\left(\sum\limits_{k=1}^n m_{k}\otimes b_{ki}\right)_{1\times n},$$
where $\B=(b_{ij})_{n \times n}$ forms a multiplicative matrix. For example, suppose there exists a subquiver of $\mathrm{Q}(H)$ of the form
$$
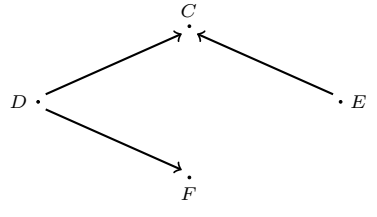
\begin{tikzpicture}
\filldraw [black] (0,1) circle (0.5pt) node[anchor=south]{${}_{C}$};
\filldraw [black] (-2,0) circle (0.5pt)node[anchor=east]{${}_{D}$};
\filldraw [black] (2,0) circle (0.5pt) node[anchor=west]{${}_{E}$};
\filldraw [black] (0,-1) circle (0.5pt)node[anchor=north]{${}_{F}$};
\draw[thick, ->] (-1.9,0.1).. controls (-1.9,0.1) and (-1.9,0.1) .. node[anchor=south]{}(-0.1,0.9);
\draw[thick, ->] (1.9,0.1).. controls (1.9,0.1) and (1.9,0.1) .. node[anchor=south]{}(0.1,0.9);
\draw[thick, ->] (-1.9,-0.1).. controls (-1.9,-0.1) and (-1.9,-0.1) .. node[anchor=south]{}(-0.1,-0.9);
\end{tikzpicture},
$$
where $C, D, E, F \in \mathcal{S}$ are distinct simple subcoalgebras of $H$ with $\dim_{\k}(C) = r^2$, $\dim_{\k}(D) = s^2$, $\dim_{\k}(E) = t^2$, and $\dim_{\k}(F) = u^2$, respectively. We know that there exist a non-trivial $(\C, \D)$-primitive matrix $\X_{r \times s} = (x_{ij})_{r \times s} \in \mathcal{P}$, a non-trivial $(\C, \E)$-primitive matrix $\Y_{r \times t} = (y_{ij})_{r \times t} \in \mathcal{P}$, and a non-trivial $(\F, \D)$-primitive matrix $\Z_{u \times s} = (z_{ij})_{u \times s} \in \mathcal{P}$. Clearly,
\begin{eqnarray*}
&U = \operatorname{span}\{c_{11},\dots,c_{1r}, x_{11},\dots,x_{1s}\},\\
&V= \operatorname{span}\{c_{11},\dots,c_{1r}, x_{11},\dots,x_{1s}, y_{11},\dots,y_{1t}\},\\
&W(k) = \operatorname{span}\{c_{11},\dots,c_{1r},\; k f_{11},\dots,k f_{1u},\; x_{11}+k z_{11},\dots,x_{1s}+k z_{1s}\},\;\;k \in \k^\times
\end{eqnarray*}
are three kinds of indecomposable comodules, with multiplicative matrix
$$\left(\begin{array}{cc}
\C & \X\\
0&  \D
 \end{array}\right),\;\;\begin{pmatrix}
\C & \X & \Y \\
& \D & \\
& & \E
\end{pmatrix},\;\; \begin{pmatrix}
\C & & \X \\
& \F & \Z \\
& & \D
\end{pmatrix},$$
respectively. Besides, we know that $W(k)\cong W(l)$ as right $H$-comodule for all non-zero $k, l\in\k.$ If $D=E$, then $$\{V(k)=\span\{c_{11},\dots,c_{1r}, y_{11}+kx_{11},\dots,y_{1s}+kx_{1s}\}\mid k\in \k\}$$ forms a class of pairwise non-isomorphic indecomposable comodules. More complicated cases are considered similarly.
\item[(3)]In \cite{GLS17}, Geiss, Leclerc, and Schr\"{o}er introduced a family of Iwanaga-Gorenstein algebras defined by quivers with relations determined by symmetrizable Cartan matrices, extending the classical path algebras of quivers associated with symmetric Cartan matrices. They demonstrated that the representation theory of these algebras bears a striking resemblance to that of modulated graphs in the sense of Dlab and Ringel \cite{DR74}, and that these algebras are, in fact, tensor algebras. Ultimately, they obtained new representation-theoretic realizations of all finite root systems, valid over an arbitrary ground field (not necessarily algebraically closed). From the perspective of quivers, in the theory of Dlab and Ringel, each vertex of the quiver is equipped with a division algebra, whereas Geiss, Leclerc, and Schr\"{o}er placed at each vertex a commutative symmetric algebra, and assigned to each arrow the corresponding bimodule. Generalizations of the representation theory of modulated graphs have previously been formulated in \cite{Li12}. It is worth mentioning that Coelho and Liu \cite{CL00} introduced the notion of a generalized path algebra. From the quiver perspective, their idea is to place a simple algebra at each vertex and the corresponding free bimodule on each arrow. Generalized path algebras are very useful for describing the representations and structure of non-elementary algebras. For example, Li and Lin \cite{LL12} studied the Ext quiver of an Artinian $\k$-algebra that is splitting over $\k$, and characterized such Artinian algebras using the generalized Gabriel theorem. The construction we consider in this section can be understood as the dual situation of their setting in some sense.
\end{itemize}
\section{Applications to finite integral tensor categories with the Chevalley property}\label{section5}
In this section, we mainly employ the quiver approach to study the classification of coquasi-Hopf algebras with the dual Chevalley property, thereby obtaining classification results for the corresponding tensor categories.
\subsection{Representation types}
In this subsection, we classify finite-dimensional coquasi-Hopf algebras with the dual Chevalley property according to their corepresentation type by using the quiver method. For general background on representation theory, the reader is referred to \cite{ASS06,ARS95}.

Recall that a finite-dimensional algebra $A$ is of \textit{finite representation type} if there are finitely many non-isomorphic indecomposable $A$-modules. $A$ is \textit{tame} if it is not of finite type, and for every $d>0$, there exist finitely many $A$-$\k[T]$-bimodules $M_i$ free of finite rank as right $\k[T]$-modules, such that all but finite number of indecomposable $A$-modules of dimension $d$ are isomorphic to
$M_i\otimes_{\k[T]}\k[T]/(T-\lambda)$ for some $\lambda\in\k.$ $A$ is a \textit{wild} algebra if there exists a finitely generated $A$-$\k\langle X, Y\rangle$-bimodule $B$ free as a right $\k\langle X, Y\rangle$-module such that the functor $B\otimes_{\k\langle X, Y\rangle}\_$ from the category of finitely generated $\k\langle X, Y\rangle$-modules to the category of finitely generated $A$-modules, preserves indecomposability and reflects isomorphisms. A finite-dimensional coalgebra $C$ is of \textit{finite} (resp., \textit{tame}, \textit{wild}) \textit{corepresentation type} if the dual algebra $C^*$ is of finite (resp. tame, wild) representation type. See \cite{ARS95}  for details.

According to Drozd's fundamental result \cite{Dro79}, every finite-dimensional (co)algebra belongs to precisely one of the following three classes: (co)algebras of finite (co)representation type, tame (co)representation type, or wild (co)representation type.
It is well-known that a finite $\k$-linear abelian category $\mathscr{C}$ is equivalent to the category ${}_A\mathscr{M}$ of finite-dimensional modules over a finite-dimensional $\k$-algebra $A.$ In fact, one can take this algebra to be $\operatorname{End}(P)^{op}$, where $P$ is a projective generator of $\mathscr{C}$. We call a finite tensor category is of \textit{finite representation type, tame representation type, or wild representation type} if
$\operatorname{End}(P)^{op}$ is of finite, tame, or wild representation type, respectively.
By Lemma \ref{lem:integertensor}, the classification of finite integral tensor categories according to their representation type is equivalent to the classification of finite-dimensional coqausi-Hopf algebras according to their corepresentation type.

In order to establish the relationship between link quivers and corepresentation type, we
 first introduce the definition of a separated quiver.
\begin{definition}\emph{(}\cite[\textsection X. 2]{ARS95}\emph{)}
Let $\mathrm{Q}=(\mathrm{Q}_0,\mathrm{Q}_1)$ be a quiver with $\mathrm{Q}_0=\{1,2,\cdots,n\}$. The separated quiver $\mathrm{Q}_{s}$ of $\mathrm{Q}$ has $2n$ vertices $\{1, 2, \cdots, n, 1^\prime, 2^\prime, \cdots, n^\prime\}$, and for each arrow $i \rightarrow j$ in $\mathrm{Q}$, there is an arrow $i \rightarrow j^\prime$ in $\mathrm{Q}_s$.
\end{definition}

From the proofs of \cite[Proposition 5.2]{YLL24} and \cite[Theorem 4.2]{YL24}, we have the following lemma.
\begin{lemma}\label{lem:separatedquiver}
Let $H$ be a finite-dimensional coalgebra over $\k$, and let $\mathrm{Q}(H)$ be its link quiver. If $H$ is of finite corepresentation type, then the separated quiver $\mathrm{Q}(H)_s$ is a finite disjoint union of Dynkin diagrams; if $H$ is of tame corepresentation type, then $\mathrm{Q}(H)_s$ is a finite disjoint union of Dynkin diagrams and Euclidean diagrams, containing at least one Euclidean component.
\end{lemma}

The following proposition provides a description of the link quiver of a finite-dimensional non-cosemisimple coquasi-Hopf algebra with the dual Chevalley property that is of finite or tame corepresentation type, which generalizes \cite[Theorem 4.2]{YL24}.
\begin{proposition}\label{prop:corep=quiver}
Let $H$ be a finite-dimensional non-cosemisimple coquasi-Hopf algebra over $\k$ with the dual Chevalley property, and let $\mathrm{Q}(H)=(\mathcal{S},\mathcal{P})$ be its link quiver.
\begin{itemize}
\item[(1)]$H$ is of finite corepresentation type if and only if $\mid{}^1\mathcal{P}\mid=1$ and ${}^1\mathcal{S}=\{\k g\}$ for some group-like element $g\in G(H)$.
  \item[(2)]If $H$ is of tame corepresentation type, then one of the following two cases occurs:
  \begin{itemize}
  \item[(i)]$\mid{}^1\mathcal{P}\mid=2$ and $\dim_{\k}(C)=1$ for all $C\in {}^1\mathcal{S}$;
  \item[(ii)]$\mid{}^1\mathcal{P}\mid=1$ and ${}^1\mathcal{S}=\{C\}$ for some $C\in\mathcal{S}$ with $\dim_{\k}(C)=4$.
  \end{itemize}
  \item[(3)]If one of the following holds, then $H$ is of wild corepresentation type:
  \begin{itemize}
  \item[(i)]$\mid{}^1\mathcal{P}\mid\geq3$;
  \item[(ii)]$\mid{}^1\mathcal{P}\mid=2$ and there exists some $C\in{}^1\mathcal{S}$ with $\dim_{\k}(C)\geq 4$;
  \item[(iii)]$\mid{}^1\mathcal{P}\mid=1$ and ${}^1\mathcal{S}=\{C\}$ for some $C\in\mathcal{S}$ with $\dim_{\k}(C)\geq 9$.
  \end{itemize}
\end{itemize}
\end{proposition}
\begin{proof}
Suppose $\mid{}^1\mathcal{P}\mid=1$ and ${}^1\mathcal{S}=\{\k g\}$ for some group-like element $g\in G(H)$. By Remark \ref{rm:link=ext}, the link quiver $\mathrm{Q}(H)$ of $H$ coincides with the algebra version of the Ext quiver $\Gamma(H^*)^{\mathrm{a}}$ of $H^*$. Note that $H^*$ is Morita equivalent to a basic algebra $\mathcal{B}(H^*)$. According to Corollary \ref{coro:cpdspecial} (3), every vertex of the Ext quiver of $\mathcal{B}(H^*)$ is the start vertex of one arrow and the end vertex of one arrow, which implies that $\mathcal{B}(H^*)$ is a Nakayama algebra (see \cite[\textsection V. 2. Theorem 2.6]{ASS06}). By \cite[\textsection VI. Theorem 2.1]{ARS95}, the Nakayama algebra $\mathcal{B}(H^*)$ is of finite representation type, hence $H$ is of finite corepresentation type. Conversely, if $H$ is of finite corepresentation type, then by Lemma \ref{lem:separatedquiver} the separated quiver $\mathrm{Q}(H)_s$ is a finite disjoint union of Dynkin diagrams. If $\mid{}^1\mathcal{P}\mid\geq2$, a similar argument to that in the proof of \cite[Proposition 5.2]{YLL24} shows that $H$ is not of finite corepresentation type. If $\mid{}^1\mathcal{P}\mid=1$ and $C_k$ is the unique simple subcoalgebra contained in ${}^1\mathcal{S}$. We claim that $C_k$ is $1$-dimensional. Indeed, Lemma \ref{lem:FPeq} and Remark \ref{rm:FP=dim} yield
$$\sqrt{\dim_{\k}(C_k)}\left(\sum\limits_{t\in I} \sqrt{\dim_{\k}(C_t)}\right)=\sum\limits_{i\in I}\sum\limits_{t\in I} \sqrt{\dim_{\k}(C_i)}\alpha_{i,k}^t.$$
If $\dim_{\k}(C_k)\geq9$, then since $1 \cdot C_k=C_k$, it follows from Proposition \ref{prop:alphaijk} that the separated quiver of $\mathrm{Q}(H)$ contains a vertex which is the end vertex of at least $4$ arrows, and hence is not the union of Dynkin diagrams. If $\dim_{\k}(C_k)= 4$, then using Lemma \ref{lem:cijkinvariant} and a similar argument as in the proof of \cite[Proposition 5.5]{YLL24}, one can show that $\mathrm{Q}(H)$ is not a union of Dynkin diagrams. Consequently, if $H$ is of finite corepresentation type, then $\mid{}^1\mathcal{P}\mid=1$ and the unique simple subcoalgebra contained in ${}^1\mathcal{S}$ is $1$-dimensional. The proof of (1) is complete.

Clearly, (2) is the negation of (3), so it suffices to prove (3). Using similar arguments as in the proof of (1) and \cite[Theorem 4.2]{YL24}, we can establish (3).
\end{proof}

Inspired by \cite[Definition 2.6]{HLYY21}, we define the rank of a coquasi-Hopf algebra with the dual Chevalley property.
\begin{definition}\label{def:rank}
Let $H$ be a coquasi-Hopf algebra over $\k$ with the dual Chevalley property, and let $\mathrm{Q}(H)=(\mathcal{S},\mathcal{P})$ be its link quiver where $\mathcal{S}=\{C_i\mid i\in I\}$, and suppose ${}^1\mathcal{S}=\{C_k\mid k\in J\}$, where $J \subseteq I.$ The rank of $H$ is defined to be the natural number $\sum_{k\in J} \mid{}^1\mathcal{P}^{\C_k}\mid \sqrt{\dim_{\k}(C_k)}$.
\end{definition}
\begin{remark}\rm
Recall that in case $H$ is a Hopf algebra over $\k$ with the dual Chevalley property,
the \textit{rank} of $H$ is defined to be $n$ if $\dim_{\k}(\k\otimes_{H_{0}} H_1)=n+1$ and $H$ is generated by $H_1$ as an algebra \cite{KR06}. According to \cite[Corollary 2.11]{YLL24}, this rank equals the dimension of the space spanned by all entries of non-trivial $(1,\C)$-primitive matrices, where $C$ runs over $ {}^1\mathcal{S}$. In this respect, Definition \ref{def:rank} may be seen as a generalization from the case of Hopf algebras to that of coquasi-Hopf algebras by omitting the generation condition. Moreover, if $H$ is a pointed coquasi-Hopf algebra, then by Remark \ref{rm:Hopf;pointed} (2) our definition coincides with \cite[Definition 2.6]{HLYY21}.
\end{remark}
Clearly, if a finite-dimensional coquasi-Hopf algebra with the dual Chevalley property has rank 0, then it is cosemisimple and hence of finite corepresentation type.
As a consequence of Proposition \ref{prop:corep=quiver}, we have
\begin{corollary}\label{coro:corep=rank}
Let $H$ be a finite-dimensional coquasi-Hopf algebra over $\k$ with the dual Chevalley property, and let $r(H)$ be its rank.
\begin{itemize}
\item[(1)]$H$ is of finite corepresentation type if and only if $r(H)\leq 1$;
  \item[(2)]If $H$ is of tame corepresentation type, then $r(H)=2$;
  \item[(3)]If $r(H)\geq 3$, then $H$ is of wild corepresentation type.
\end{itemize}
\end{corollary}

Next we classify finite-dimensional coquasi-Hopf algebra with the dual Chevalley property of finite corepresentation type. We begin by providing a detailed quiver description for this case.
\begin{lemma}\label{lem:finite=}
Let $H$ be a finite-dimensional non-cosemisimple coquasi-Hopf algebra over $\k$ with the dual Chevalley property, and let $\mathrm{Q}(H)$ be its link quiver. Then the following statements are equivalent:
\begin{itemize}
  \item[(1)]$H$ is of finite corepresentation type;
  \item[(2)]There is exactly one arrow $C\rightarrow \k1$ in $\mathrm{Q}(H)$ whose end vertex is $\k1$ and $\dim_{\k}(C)=1$;
  \item[(3)]There is exactly one arrow $\k1\rightarrow D$ in $\mathrm{Q}(H)$ whose start vertex is $\k1$ and $\dim_{\k}(D)=1$;
  \item[(4)]Every vertex in $\mathrm{Q}(H)$ is the start vertex of of exactly one arrow and the end vertex of exactly one arrow; in other words, $\mathrm{Q}(H)$ is a disjoint union of basic cycles;
  \item[(5)]$H_{(1)}$ is a pointed coquasi-Hopf algebra and the link quiver of $H_{(1)}$ is a basic cycle.
\end{itemize}
\end{lemma}
\begin{proof}
Proposition \ref{prop:corep=quiver} (1) shows the equivalence of (1) and (2). The equivalence of (2) and (3) follows from Lemma \ref{lem:P1prime>0}. By Corollary \ref{coro:cpdspecial} (3), (2) and (4) are equivalent. Corollary \ref{coro:H(1)pointed} establishes the equivalence between (2) and (5).
\end{proof}
Since the link quiver of a non-cosemisimple coquasi-Hopf algebra with the dual Chevalley property $H$ coincides with that of $\operatorname{gr} H$, the following corollary follows immediately.

\begin{corollary}
Let $H$ be a finite-dimensional non-cosemisimple coquasi-Hopf algebra over $\k$ with the dual Chevalley property. Then $H$ is of finite corepresentation type if and only if $\operatorname{gr} H$ is of finite corepresentation type.
\end{corollary}
\begin{remark}\rm
Suppose that $H$ is a coradically graded coquasi-Hopf algebra of finite corepresentation type. Then by Theorem \ref{thm:Gabriel} and Lemma \ref{lem:finite=}, there exists a unique generalized Hopf quiver $\mathrm{Q}$ associated with $(H_0, \{C_i\}_{i\in I}, (0,0,\cdots ,n_k,\cdots,0,0), \{\alpha_{ij}^t\}_{i,j,t\in I})$ with $n_k=1$ and $\dim_{\k}(C_{k})=1$ such that $H$ is a large coquasi-Hopf subalgebra of $\k(\mathrm{Q}, \mathcal{S})$. In particular, if $H_0$ is genuine, a similar argument as in the proof of Corollary \ref{coro:genuine} shows that $H$ is genuine. Therefore, by means of bosonization, we can construct a finite-dimensional genuine coquasi-Hopf algebra with the dual Chevalley property of rank $1$ on a genuine cosemisimple coquasi-Hopf algebra. Corollary \ref{coro:corep=rank} tells us that this coquasi-Hopf algebra is of finite corepresentation type.
\end{remark}

Recall that a finite-dimensional coalgebra $H$ is said to be \textit{coNakayama}, if the dual algebra $H^*$ is a Nakayama algebra. According to \cite[\textsection V. 2. Theorem 2.6]{ASS06} and Remark \ref{rm:link=ext}, $H$ is coNakayama if and only if the link quiver of $H$ is the start vertex of at most one arrow and the end vertex of at most one arrow. We have the following corollary.
\begin{corollary}\label{coro:nakayama}
A finite-dimensional coquasi-Hopf algebra $H$ over $\k$ with the dual Chevalley property is of finite corepresentation type if and only if $H$ is coNakayama. A finite-dimensional quasi-Hopf algebra $H$ over $\k$ with the Chevalley property is of finite representation type if and only if $H$ is a Nakayama algebra.
\end{corollary}
From Lemma \ref{lem:integertensor}, Remark \ref{rm:link=ext}, Lemma \ref{lem:finite=}, Corollary \ref{coro:nakayama} and \cite[\textsection VI. Theorem 2.1]{ARS95} we obtain the following corollary.
\begin{corollary}
Let $\mathscr{C}$ be a finite integral tensor categories with the Chevalley property of finite representation type. Then the Ext quiver of $\mathscr{C}$ is a disjoint union of basic cycles. Moreover, every indecomposable object $X$ is uniserial, that is, it has a unique composition series and hence is a quotient of an indecomposable projective object.
\end{corollary}
Before proceeding further, we give an example of a pointed coquasi-Hopf algebra whose link quiver is a basic cycle.
\begin{example} \label{ex:M(n,s,q)}\emph{(}\cite[Example 2.8]{HLYY21}\emph{)}\rm
Let $\k$ an algebraically closed field with characteristic
$0$. For $n\geq 2$, consider the following Hopf quiver $\mathrm{Q}(\mathbb{Z}_{n},g)$:
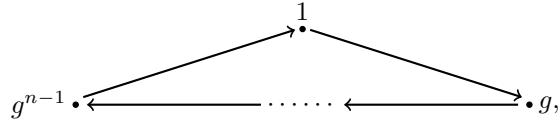
$$
\begin{tikzpicture}
\filldraw [black] (0,0) circle (1pt) node[anchor=south]{1};
\filldraw [black] (-3,-1) circle (1pt)node[anchor=east]{$g^{n-1}$};
\filldraw [black] (3,-1) circle (1pt) node[anchor=west]{$g$,};
\filldraw [black] (0,-1) circle (0pt)node{$\cdots\cdots$};
\draw[thick, ->] (0.1,-0.01) .. controls (0.1,-0.01) and (0.1,-0.01) .. node[anchor=west]{}(2.9,-0.9) ;
\draw[thick, ->] (2.85,-1) .. controls (2.85,-1) and (2.85,-1) .. node[anchor=west]{}(0.55,-1) ;
\draw[thick, ->] (-0.55,-1) .. controls (-0.55,-1) and (-0.55,-1) .. node[anchor=west]{}(-2.85,-1) ;
\draw[thick, ->] (-2.9,-0.9) .. controls (-2.9,-0.9) and (-2.9,-0.9) .. node[anchor=west]{}(-0.1,-0.01) ;
\end{tikzpicture}
$$
Let $1\leq s\leq n-1$ be a natural number, let $q$ be an $n^{2}$-th primitive root of unity, and set $\mathbbm{q}:=q^{n}$. Let $p_{i}^{l}$ denote the path in $\mathrm{Q}(\mathbb{Z}_{n},g)$ starting from $g^{i}$ with length $l$; thus $p_{i}^{0}=g^{i}$.
Let $\Phi_{s}$ be the $3$-cocycle on $\mathbb{Z}_{n}$ defined by
\begin{eqnarray*}
\Phi_{s}(g^{i},g^{j},g^{k})=\mathbbm{q}^{si[\frac{j+k}{n}]},\;\; 0\leq i,j,k\leq  n-1,
\end{eqnarray*}
where $[x]$ denotes the integer part of $x$. For any $h\in \k$, define $l_h = 1+h+\cdots+h^{l-1}$ and $l!_h = 1_h\cdots l_h$.
The Gaussian binomial coefficient is defined by
$\binom{l+m}{l}_h:=\frac{(l+m)!_h}{l!_h\, m!_h}$.
Let $(a,b)$ denote the greatest common divisor of $a$ and $b$.
The rank $1$ pointed coquasi-Hopf algebra $M(n,s,q)$ is, as a coalgebra, $$M(n,s,q)=\bigoplus_{i< \frac{n^{2}}{(n^{2},s)}} k \mathrm{Q}(\mathbb{Z}_{n},g)_{i},$$ which is a subcoalgebra of path coalgebra $\k\mathrm{Q}(\mathbb{Z}_{n},g)$. The reassociator, the multiplication, and the coquasi-antipode are given as follows:
\begin{eqnarray*}
&\Phi(p_i^{l},p_j^{m}, p_k^{t})=\delta_{l+m+t,0}\Phi_{s}(g^{i},g^{j},g^{k}),\\
&\label{eq2.10}p_i^{l}\cdot p_j^{m}=\mathbbm{q}^{-sjl}q^{-sjl}\mathbbm{q}^{s(i+l')[m+j-(m+j)']/n}
\binom{l+m}{l}_{\mathbbm{q}^{-s}q^{-s}}p_{i+j}^{l+m},\\
&\alpha(p_i^{l})=\delta_{l,0}1,\;\;\;\;\beta(p_i^{l})=\delta_{l,0}\frac{1}{\Phi_{s}(g^{i},g^{n-i},g^{i})},\\
&S(g^{i})=g^{n-i},\;\;\;\;S(p_{0}^{1})=\mathbbm{q}^{-s}p_{n-1}^{1},\end{eqnarray*}
for $0\leq l,m,t<\frac{n^{2}}{(n^{2},s)}$ and $0\leq i,j,k\leq n-1$, where $\delta_{a,b}$ is the Kronecker
delta and $l'$ means the remainder of $l$ divided by $n$. According to \cite[Corollary 3.11]{HLY11}, $M(n,s,q)$ is genuine. Moreover, the set $\{ M(n,0,q)\mid q^n=1\}$ are the usual Hopf algebras, called generalized Taft algebras, which is generated by $g$ and $x$ with relations
$$
g^n=1, \;\;\;\;x^d=0,\;\;\;\;xg=qgx,
$$
where $q$ is an $n$-th root of unit of order $d$.
Its comultiplication $\Delta$, counit $\varepsilon$, and the antipode $S$ are given by
$$
\Delta(g)=g\otimes g,\;\; \varepsilon(g)=1,\;\;
\Delta(x)=x\otimes 1+g\otimes x,\;\; \varepsilon(x)=0,\;\;
S(g)=g^{-1},\;\;S(x)=-g^{-1}x.
$$
\end{example}

Now we classify finite-dimensional coquasi-Hopf algebras with the dual Chevalley property of finite corepresentation type. This classification contains the corresponding result for Hopf algebras with the dual Chevalley property given in \cite[Theorem 5.15]{YLL24}.
\begin{proposition}\label{prop:finite}
Let $\k$ be an algebraically closed field with characteristic $0$. Then a finite-dimensional coquasi-Hopf algebra $H$ over $\k$ with the dual Chevalley property is of finite corepresentation type if and only if one of the following conditions holds:
\begin{itemize}
  \item[(1)] $H$ is cosemisimple;
  \item[(2)] $H$ is not cosemisimple and $\operatorname{gr}H_{(1)}\cong M(n,s,q)$.
\end{itemize}
\end{proposition}
\begin{proof}
By Corollary \ref{coro:H(1)pointed} and Lemma \ref{lem:finite=}, a finite-dimensional non-cosemisimple coquasi-Hopf algebra with the dual Chevalley property is of finite corepresentation type if and only if the link-indecomposable component containing $\k 1$ is a pointed coquasi-Hopf algebra of finite corepresentation type. By \cite[Corollary 3.13]{HLY11}, we can complete the proof.
\end{proof}
\begin{remark}\rm
\begin{itemize}
\item[(1)]It should be noted that there is very little general theory available at present for finding all the liftings of non-pointed Hopf algebras with the dual Chevalley property, and even the pointed case has not been fully resolved. As stated in \cite[Subsection 3.7 (1)]{HLY11}, since the reassociators get involved in the lifting method, even if $M(n,s,q)$ is pointed and of rank $1$, finding all liftings of $M(n,s,q)$ becomes extremely complicated. In our subsequent work, we will study lifting methods for coquasi-Hopf algebras with the dual Chevalley property, which can be seen as a quasi-analogue of the lifting methods for Hopf algebras \cite{AAGM14, AG19}. It is worth pointing out that Angiono \cite{Ang10} determined the liftings when $n$ is coprime to $2,3,5,7.$
\item[(2)]Let $H$ be a finite-dimensional non-cosemisimple coquasi-Hopf algebra with the dual Chevalley property of finite corepresentation type, and let $\mathcal{S}$ be the set of all simple subcoalgebras of $H.$ According to Corollary \ref{coro:H=CH(1)}, there exists a subset $\mathcal{S}_0\subseteq\mathcal{S}$ such that $H\cong \bigoplus_{C\in\mathcal{S}_0} CH_{(1)}$ with $\operatorname{gr}H_{(1)}\cong M(n,s,q)$. Thus $H$ is generated in degree $1$ of its coradical filtration.
\item[(3)]A standard dualization process yields parallel classification results in Proposition \ref{prop:finite} for finite-dimensional quasi-Hopf algebras with the Chevalley property of finite representation type.
\end{itemize}
\end{remark}

We subsequently give the classification of finite integral tensor categories with the Chevalley property that are of finite representation type.
\begin{theorem}\label{thm:finite}
Let $\k$ be an algebraically closed field with characteristic $0$, and
let $\mathscr{C}$ be a finite integral tensor categories over $ \k$ with the Chevalley property. Then $\mathscr{C}$ is of finite representation type if and only if either $\mathscr{C}$ is a fusion category (i.e., a finite semisimple tensor category), or is tensor equivalent to $\mathscr{M}^H$ for some finite-dimensional coquasi-Hopf algebra $H$ with the dual Chevalley property such that $\operatorname{gr}H_{(1)}\cong M(n,s,q)$.
\end{theorem}
\begin{proof}
The ``if part" is trivial; we only need to prove the ``only if part".
According to Lemma \ref{lem:integertensor}, $\mathscr{C}$ is tensor equivalent to $\mathscr{M}^H$ for some finite-dimensional coquasi-Hopf algebra $H$ with the dual Chevalley property. Since $\mathscr{C}$ is of finite representation type, it follows that $H$ is of finite corepresentation type. Now the theorem follows immediately from Proposition \ref{prop:finite}.
\end{proof}

For tame corepresentation type, the quiver method alone is insufficient for a complete classification; in other words, the condition in Proposition \ref{prop:corep=quiver} (2) is not sufficient. Even for Hopf algebras with the dual Chevalley property and pointed coquasi-Hopf algebras, there are few classification results concerning tame corepresentation type; currently, only some necessary conditions for the coradically graded case have been obtained. See, for example, \cite[Theorem 5.2]{YL24} and \cite[Theorem 6.5]{HLY15}. The reassociator renders Hopf-type proofs inapplicable. A key step in the Hopf case is to decompose a radically graded Hopf algebra $H$ as the bosonization $R_H\#H/J_H$, where $R_H$ is a local tame subalgebra, so that Ringel's remarkable classification result for local algebras \cite{Rin75} can be applied. This decomposition fails in the quasi-Hopf setting: the analogue of
$R_H$ is not an associative subalgebra. Moreover, the methods for the elementary (pointed) case also break down. The authors in \cite{HLY15} applied Gelaki's method \cite{Gel05} to construct graded elementary quasi-Hopf algebras of tame type, but this method is also inapplicable to non-elementary quasi-Hopf algebras. Consequently, we do not solve the tame corepresentation type classification here; instead we provide only necessary conditions for tame corepresentation type.
\begin{lemma}\label{lem:1tame}
Let $H$ be a finite-dimensional coquasi-Hopf algebra over $\k$ with the dual Chevalley property of tame corepresentation type.
Then $H_{(1)}$ is of tame corepresentation type.
\end{lemma}
\begin{proof}
Since $H$ is of tame corepresentation type, it follows from Proposition \ref{prop:corep=quiver} that either $\mid {}^1\mathcal{P}\mid>1$ or there exists $C\in{}^1\mathcal{S}$ with $\dim_{\k}(C)>1$.
Hence $H_{(1)}$ is not of finite corepresentation type.
Moreover, there is an inclusion from $\mathscr{M}^{H_{(1)}}$ to $\mathscr{M}^{H}$. Consequently, $H_{(1)}$ is not of wild corepresentation type. By a fundamental result of \cite{Dro79}, $H_{(1)}$ is therefore of tame corepresentation type.
\end{proof}
According to Proposition \ref{prop:corep=quiver}, if $H$ is a coquasi-Hopf algebra with the dual Chevalley property, then the following two cases occur:
\begin{itemize}
  \item[(i)]$\mid{}^1\mathcal{P}\mid=2$ and $\dim_{\k}(C)=1$ for all $C\in {}^1\mathcal{S}$;
  \item[(ii)]$\mid{}^1\mathcal{P}\mid=1$ and ${}^1\mathcal{S}=\{C\}$ for some $C\in\mathcal{S}$ with $\dim_{\k}(C)=4$.
  \end{itemize}
Proposition \ref{prop:(H1)0} yields that in case (1),
$H_{(1)}$ is pointed, while in case (2),
$H_{(1)}$ is non-pointed.

Before addressing case (i), we first present the following example.
\begin{example}\emph{(}\cite[Section 5]{HLY15}\emph{)}\rm
Let $\k$ be an algebraically closed field with characteristic $0$, and let $W = \mathbb{Z}_m \times \mathbb{Z}_n=\langle g_1,g_2|g_{1}^{m}=g_{2}^{n}=1, g_{1}g_{2}=g_{2}g_{1}\rangle$ with $m$ even and $m \mid n$. Suppose $g, h$ generate $W$ with orders $o(g), o(h)$, and let $q, p$ be $o(g)$-th and $o(h)$-th primitive roots of unity, respectively. Choose integers $l_1 \mid m$, $l_2 \mid n$, set $q_2 := q^{l_1}$, $p_1 := p^{l_2}$, and assume that $p_1 q_2$ is an $l$-th primitive root of unity. The elementary Hopf algebra $H=H(m,n,l_{1},l_{2},g,h)$ is defined to be an associative algebra
generated by elements $x,y$ and $g,h$, with relations
$$g,h \;\text{generate}\; \mathbb{Z}_{m}\times \mathbb{Z}_{n},
\;\;x^{2}=y^{2}=(xy)^{l}+(-q_{2})^{l}(yx)^{l}=0,$$
$$gxg^{-1}=q^{-1}x,\;gyg^{-1}=y,\;hxh^{-1}=x,\;hyh^{-1}=p^{-1}y.$$
The comultiplication $\Delta$, counit $\varepsilon$, and antipode
$S$ are given by
$$\Delta(g)=g\otimes g,\;\;\Delta(h)=h\otimes h,\;\;\varepsilon(g)=\varepsilon(h)=1,$$
$$\Delta(x)=x\otimes 1+ g^{\frac{o(g)}{2}}h^{l_{2}}\otimes x,
\;\;\Delta(y)=y\otimes 1+ g^{l_{1}}h^{\frac{o(h)}{2}}\otimes y,\;\;\varepsilon(x)=\varepsilon(y)=0,$$
$$S(g)=g^{-1},\;\;S(h)=h^{-1},\;\;S(x)=-g^{\frac{o(g)}{2}}h^{-l_{2}}x,\;\;
S(y)=-g^{-l_{1}}h^{\frac{o(h)}{2}}y.$$
Set $X:=g^{-\frac{o(g)}{2}}h^{-l_{2}}x,\;Y:=g^{-l_{1}}h^{-\frac{o(h)}{2}}y$ and
assume that $m=\mathbbm{m}^2$ and $n=\mathbbm{n}^2$. For any natural number $l$, let $\zeta_{l}$ be an $l$-th primitive root of unity. Let $k\mathbb{Z}_l$ be the group algebra of the cyclic group and for all
$0\leq i\leq l-1$, define
$$1_{i}:=\frac{1}{l}\sum_{j=0}^{l-1}(\zeta_{l}^{l-i})^{j}g^{j}.$$ Set \begin{eqnarray*}
J_{a,b,c}&=&\sum_{x_{1},x_{2}=1}^{m}\sum_{y_{1},y_{2}=1}^{n}\zeta_{m}^{ax_{1}(y_{1}-y_{1}')}\zeta_{\mathbbm{m}(\mathbbm{m},\mathbbm{n})}^{bx_{2}(y_{1}-y_{1}')}
\zeta_{n}^{cx_{2}(y_{2}-y_{2}'')}1_{x_{1}}1_{x_{2}}\otimes 1_{y_{1}}1_{y_{2}}\notag\\
&=& \sum_{x_{1},x_{2},y_{1},y_{2}}\zeta_{m}^{(ax_{1}+bx_{2})(y_{1}-y_{1}')}
\zeta_{n}^{cx_{2}(y_{2}-y_{2}'')}1_{x_{1}}1_{x_{2}}\otimes 1_{y_{1}}1_{y_{2}} \ \end{eqnarray*}
for some $0\leq a,b<\mathbbm{m}$ and $0\leq c< \mathbbm{n}$. Consider the subalgebra $A(H, J_{a,b,c})\subset H$ which is generated by $X,Y$ and $\mathbbm{g}_{i}:=g_{i}^{\mathbbm{m}_{i}},\; \mathbbm{h}_{i}:=h_{i}^{\mathbbm{n}_{i}}$ for $i=1,2$. Assume that $g=h_{1}h_{2}^{\sigma}$ and $h=h_{2}$ for some $0\leq \sigma < n$. Then by \cite[Proposition 5.5]{HLY15}, $A(H, J_{a,b,c})$ is a quasi-Hopf subalgebra of $H^{J}$ if and only if
$$a=0,\;\;\mathbbm{n}|l_{2},\;\;l_{1}+b\equiv 0(\mathbbm{m}),\;\;c+\sigma l_{1}\equiv 0(\mathbbm{n}).$$ Moreover, \cite[Proposition 5.6]{HLY15} yields that if $b\neq 0$ or $c\neq 0$, then the quasi-Hopf algebra $(A(H),\Phi_{0,b,c})$ is genuine.
\end{example}

\begin{proposition}\label{prop:tame1-dim}
Let $H$ be a finite-dimensional coradically graded coquasi-Hopf algebra over an algebraically closed field with characteristic $0$ of tame corepresentation type. If $\mid{}^1\mathcal{P}\mid=2$ and $\dim_{\k}(C)=1$ for all $C\in {}^1\mathcal{S}$, then $(H_{(1)})^*$ is twist equivalent to one of the following quasi-Hopf algebras: \begin{itemize}
\item[(i)]  $H(m,n,l_{1},l_{2},g,h)$
for some $m,n,l_{1},l_{2}$ and two generators $g,h$ of $\mathbbm{Z}_{m}\times \mathbbm{Z}_{n}$;
\item[(ii)]  $A(H,J)$ for some $H=H(m,n,l_{1},l_{2},g,h)$ and twist $J\in H\otimes H$.
\end{itemize}
\end{proposition}
\begin{proof}
According to Corollary \ref{coro:H(1)pointed} and Lemma \ref{lem:1tame}, $H_{(1)}$ is a link-indecomposable pointed coradically graded coquasi-Hopf algebra of tame corepresentation type. \cite[Theorem 6.5]{HLY15} completes the proof.
\end{proof}
Next, we consider the case where $H_{(1)}$ is non-pointed.
\begin{proposition}\label{prop:tameC}
Let $H$ be a finite-dimensional coquasi-Hopf algebra over $\k$ with the dual Chevalley property of tame corepresentation type with $H_{(1)}$ non-pointed, and let $\mathrm{Q}(H)=(\mathcal{S},\mathcal{P})$ be its link quiver.
\begin{itemize}
\item[(1)]Then ${}^1\mathcal{P}=\{\X\}$ for some non-trivial $(1,\C)$-primitive matrix, where $\dim_{\k}(C)=4$. Moreover, both $C$ and $S(C)$ are in the center of $\mathbb{Z}\mathcal{S}$, and $(H_{(1)})_0$ is generated by $\{c\in C\mid C\in {}^1\mathcal{S}\}\cup\{S(c)\in C\mid C\in {}^1\mathcal{S}\}.$
\item[(2)]There exist invertible $4\times 4$ matrices $K, K^\prime$ over $\k$ such that
$
\C\odot^\prime\X=(\X\odot\C)K
$
and
$S(\C)\odot^\prime\X=(\X\odot S(\C))K^\prime.$
\end{itemize}
\end{proposition}
\begin{proof}
By Proposition \ref{prop:corep=quiver} and Corollary \ref{coro:H(1)pointed}, we know that ${}^1\mathcal{P}=\{\X\}$ for some non-trivial $(1,\C)$-primitive matrix, where $\dim_{\k}(C)=4$.
(1) follows from Corollary \ref{coro:cpdspecial} (4) together with Proposition \ref{prop:(H1)0}. A similar argument as in the proof of \cite[Lemma 7.3]{YL24} shows that (2) holds.
\end{proof}

As a corollary of Theorem \ref{thm:Gabriel}, Propositions \ref{prop:corep=quiver} and \ref{prop:tameC}, we have
\begin{corollary}\label{coro:tameC}
Let $H$ be a finite-dimensional coradically graded coquasi-Hopf algebra over $\k$ of tame corepresentation type, with $H_{(1)}$ non-pointed. Then there exists a unique generalized Hopf quiver $\mathrm{Q}$ associated with $(H_0, \{C_i\}_{i\in I}, (0,0,\cdots ,n_k,\cdots,0,0), \{\alpha_{ij}^t\}_{i,j,t\in I})$ with $n_k=1$ and $\dim_{\k}(C_{k})=4$ such that $H$ is a large coquasi-Hopf subalgebra of $\k(\mathrm{Q}, \mathcal{S})$.
\end{corollary}

\begin{example}\rm
According to \cite[Theorem 5.2]{YL24}, the Hopf algebra $H$ defined in Example \ref{ex:32dim} is a link-indecomposable non-pointed Hopf algebra with the dual Chevalley property of tame corepresentation type. Indeed, we have $H\cong (\k\langle x,y\rangle/(x^2, y^2, (xy)^2+(yx)^2))^*\# H_0.$
\end{example}

\begin{remark}\rm
Let $H$ be a finite-dimensional coquasi-Hopf algebra with the dual Chevalley property. By the general degeneration theory \cite{Gei95}, the tameness of $\operatorname{gr} H$ implies that $H$ is also of tame corepresentation type.
\end{remark}

At the end of this subsection, we start from some known Hopf algebras to construct examples of coquasi-Hopf algebras with the dual Chevalley property of finite or tame corepresentation type.
\begin{example}\rm
Let $H$ be a $16$-dimensional Hopf algebra (see \cite[Example 6.1]{YLL24}) over an algebraically closed field $\k$ generated by $x,y,z,u$ with relations
\begin{eqnarray*}
 &x^2=y^2=1, \;z^2=\frac{1}{2}(1+x+y-xy),\;yx=xy,\;zx=yz,\;zy=xz,\label{H8multi}\\
   &u^2=0,\;\; ux=-xu,\;\; uy=-yu,\;\; uz=\sqrt{-1}xzu,
\end{eqnarray*}
and with coalgebra structure and antipode given by
\begin{eqnarray*}
&\Delta(x)=x\otimes x,\; \Delta(y)=y\otimes y,\;\varepsilon(x)=\varepsilon(y)=1,\;\;S(x)=x,\;\;S(y)=y,\label{H8comulti1}\\
&\Delta(z)=\frac{1}{2}(1\otimes 1+1\otimes x+ y\otimes 1-y\otimes x)(z\otimes z),\; \varepsilon(z)=1,\;\;S(z)=z,\label{H8comulti2}\\
&\Delta(u)=x\otimes u+u\otimes 1,\;\; \varepsilon(u)=0,\;\;S(u)=-xu.
\end{eqnarray*}
A basis of $H$ is $\{1,x,y,xy,z,xz,yz,xyz,u,xu,yu,xyu,zu,xzu,yzu,xyzu\}.$ We define linear functions $\Phi,\alpha,\beta$ pointwise on the basis elements as follows:
\begin{eqnarray*}
&\Phi(a,b,c):=\left\{
\begin{aligned}
-\varepsilon(a)\varepsilon(b)\varepsilon(c),~~~ &\text{if} ~~~  a,b,c \in \{z,xz,yz,xyz\}; \\
\varepsilon(a)\varepsilon(b)\varepsilon(c),\;\;~~~&\text{otherwise},
\end{aligned}
\right.\\
&\alpha(a):=\varepsilon(a),\\
&\beta(a):=\left\{\begin{aligned}
-\varepsilon(a),~~~ &\text{if} ~~~  a \in \{z,xz,yz,xyz\}; \\
\varepsilon(a),\;\;~~~&\text{otherwise}.
\end{aligned}
\right.
\end{eqnarray*}
Then the tuple $H(x,y,z,u)=(H, m, \mu, \Delta, \varepsilon, \Phi, S, \alpha, \beta)$ forms a genuine coradically graded coquasi-Hopf algebra, whose coradical is $(K_u)^*$ in Example \ref{ex:Ku}, where $m,\mu,\Delta,\varepsilon, S$ coincide with those of $H.$ Denote $E = \operatorname{span}\{z,xz, yz, xyz\}$. Then
$
\mathcal{S} = \{\k 1, \k x, \k y, \k xy, E\}.
$
We give the corresponding multiplicative matrix $\E$ of $E$, where
$$
\mathcal{E} = \frac{1}{2}\begin{pmatrix}
z + yz & z - yz \\
xz - xyz & xz + xyz
\end{pmatrix}.
$$
In this example,
$
\mathcal{P} = \{(u), (xu), (yu), (xyu), \X\},
$
where
$$
\mathcal{X} = \frac{1}{2}\begin{pmatrix}
zu + yzu & zu - yzu \\
xyzu - xzu & -xzu - xyzu
\end{pmatrix}
$$
is a non-trivial $(\E, \E)$-primitive matrix. The link quiver of $H$ is shown below:
$$
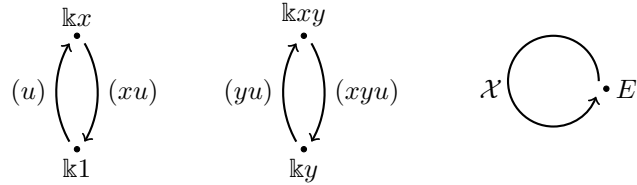
\begin{tikzpicture}
\filldraw [black] (6,0) circle (1pt) node[anchor=south]{$\k  x$};
\filldraw [black] (6,-1.5) circle (1pt)node[anchor=north]{$\k  1$};
\filldraw [black] (9,0) circle (1pt) node[anchor=south]{$\k xy$};
\filldraw [black] (9,-1.5) circle (1pt)node[anchor=north]{$\k  y$};
\filldraw [black] (13,-0.7) circle (1pt)node[anchor=west]{$E$};
\filldraw [black] (11.2,-0.7) circle (0pt)node[anchor=west]{$\X$};
\draw[thick, ->] (5.9,-1.4) .. controls (5.67,-1) and (5.67,-0.5) .. node[anchor=east]{$(u)$}(5.9,-0.1) ;
\draw[thick, ->] (6.1,-0.1) .. controls (6.33,-0.5) and(6.33,-1) .. node[anchor=west]{$(xu)$}(6.1,-1.4);
\draw[thick, ->] (8.9,-1.4) .. controls (8.67,-1) and (8.67,-0.5).. node[anchor=east]{$(yu)$}(8.9,-0.1);
\draw[thick, ->] (9.1,-0.1) .. controls (9.33,-0.5) and (9.33,-1) ..  node[anchor=west]{$(xyu)$}(9.1,-1.4);
\draw[thick, ->] (12.9,-0.6) arc (0:340:0.6);
\end{tikzpicture}.
$$
According to Lemma \ref{lem:finite=}, $H(x,y,z,u)$ is of finite corepresentation type.
\end{example}

\begin{example}\rm
Let $\k$ be an algebraically closed field of characteristic $0$. We consider the following $64$-dimensional Hopf algebra $H$, which is a Hopf subalgebra of the Hopf algebra $\mathfrak{U}_{17}(I_{17})$ in \cite[Definition 6.17]{ZGH21}.
As an algebra, $H$ is generated by $x$, $y$, $t$, $p_1$, $p_2$ satisfying the following relations:
\begin{eqnarray*}
&x^4 = 1,\;\; y^2 = 1,\;\; t^2 = 1,\;\; xy = yx,\;\; tx = x^{-1}t,\;\; ty = yt,\\
&xp_1 = p_1x,\;\; yp_1 = p_1y,\;\; tp_1 = -p_1t,\;\;xp_2 = -p_2x,\;\; yp_2 = p_2y,\;\; tp_2 = -p_2t,\\
&p_1^2 = 0,\;\; p_2^2 = 0,\;\; p_1p_2 + p_2p_1 = 0,
\end{eqnarray*}
and its coalgebra structure and antipode is given by
\begin{eqnarray*}
 &\Delta(x) = x \otimes x,\;\; \Delta(y) = y \otimes y, \;\; \varepsilon(x) = \varepsilon(y) = 1, \;\;S(x)=x^3,\;\;S(y)=y,\\
&\Delta(t) = \frac{1}{2}\left[(1+y)t \otimes t + (1-y)t \otimes x^2t\right], \;\;\varepsilon(t) = 1, \;\;S(t)=\frac{1}{2}\left[(1+y)t + (1-y)x^2t\right],\\
&\Delta\left(p_{1}\right)=p_{1} \otimes 1+\frac{1}{2}\left(1+x^2\right) t \otimes p_{1}+\frac{1}{2}\left(1-x^2\right) y t \otimes p_{2},\;\;\varepsilon(p_1)=0,\\
    &\Delta\left(p_{2}\right)=p_{2} \otimes 1+\frac{1}{2}\left(1+x^2\right) y t \otimes p_{2}+\frac{1}{2}\left(1-x^2\right) t \otimes p_{1},\;\;\varepsilon(p_2)=0,\\
    &S(p_1)=-\frac{1}{4}\left[(1+y)t+(1-y)x^2t\right]\left[(1+x^2)p_1+y(1-x^2)p_2\right],\\
    &S(p_2)=-\frac{1}{4}\left[(1+y)t+(1-y)x^2t\right]\left[y(1+x^2)p_2+(1-x^2)p_1\right].
\end{eqnarray*}
Note that the coradical of $H$ has a basis $\{x^iy^jt^k\mid 1\leq i\leq 4, 1\leq j\leq 2, 1\leq k\leq 2\}.$ We define linear functions $\Phi,\alpha,\beta$ pointwise on the basis elements of $H_0$ as follows:
\begin{eqnarray*}
&\Phi(a,b,c):=\left\{
\begin{aligned}
-\varepsilon(a)\varepsilon(b)\varepsilon(c),~~~ &\text{if} ~~~  a,b,c \in \{x^iy^jt^k\mid i=1 \text{ or } 3, 1\leq j\leq 2, 1\leq k\leq 2\}; \\
\varepsilon(a)\varepsilon(b)\varepsilon(c),\;\;~~~&\text{otherwise},
\end{aligned}
\right.\\
&\alpha(a):=\varepsilon(a),\\
&\beta(a):=\left\{\begin{aligned}
-\varepsilon(a),~~~ &\text{if} ~~~  a \in \{x^iy^jt^k\mid i=1 \text{ or } 3, 1\leq j\leq 2, 1\leq k\leq 2\}; \\
\varepsilon(a),\;\;~~~&\text{otherwise}.
\end{aligned}
\right.
\end{eqnarray*}
Then $H(x,y,t,p_1,p_2)=(H, m, \mu,\Delta, \varepsilon, \operatorname{gr}\Phi, S, \operatorname{gr}\alpha, \operatorname{gr}\beta)$ forms a coradically graded coquasi-Hopf algebra, where $m,\mu,\Delta,\varepsilon, S$ coincide with those of $H.$ Note that $H(x,y,t,p_1,p_2)_{(1)}$ is exactly the Hopf algebra defined in Example \ref{ex:32dim}. Since $H(x,y,t,p_1,p_2)$ is isomorphic to $H$ as a bialgebra, then by \cite[Theorem 5.2]{YL24} and the fact that $$H\cong (\k\langle x,y\rangle/(x^2, y^2, (xy)^2+(yx)^2))^*\# H_0,$$ $H(x,y,t,p_1,p_2)$ is of tame corepresentation type.
\end{example}

\subsection{Finite braided integral tensor categories with the Chevalley property}
In this subsection, we focus on finite braided integral tensor categories with the Chevalley property. See \cite[Definition 8.1.1]{EGNO15} for the definition of braided categories. By Lemma \ref{lem:integertensor} and \cite[Proposition 10.1]{BCPV19}, it suffices to study finite-dimensional coquasitriangular coquasi-Hopf algebras with the dual Chevalley property.
Note that, as a special case of braided tensor categories, the classification of finite symmetric integral tensor categories with the Chevalley property, i.e., the classification of finite-dimensional (co)triangular (co)quasi-Hopf algebras with the (dual) Chevalley property, has been studied intensively by Andruskiewitsch, Etingof, and Gelaki \cite{AEG01, EG03, EG21a, EG21b}. For the general background on braided tensor categories and coquasitriangular coquasi-Hopf algebras, see \cite{BCPV19,EGNO15}.

Let $(H, R)$ be a coquasitriangular coquasi-Hopf algebra with the dual Chevalley property; see Example \ref{ex:coquaitriangular}. There is a bijection between braided structures on $\mathscr{M}^H$ and coquasitriangular structures on $H$. We first introduce a special class of coquasitriangular structures.
A coquasitriangular structure $R$ over a coquasi-Hopf algebra $H$ with the dual Chevalley property is said to be \textit{concentrated} if $R(a,b)=0$ for all homogeneous elements $a,b$ unless $a,b\in H_0$.

In fact, starting from a coquasitriangular structure $R$ on $H$, we can construct a concentrated coquasitriangular structure on $\operatorname{gr}H$. Let $\operatorname{gr}(R):\operatorname{gr}H \otimes \operatorname{gr}H \rightarrow \k$ be the linear map defined for homogeneous elements $a,b\in \operatorname{gr}H$ by
$$
\operatorname{gr}(R)(a,b):=\left\{
\begin{aligned}
R(a,b), &\text{ if} ~~~ a,b\in H_0; \\
0,\;\;\;\;&\text{ otherwise}.
\end{aligned}
\right.
$$
We have the following lemma.
\begin{lemma}\label{lem:grHcoquasi}
Let $(H, R)$ be a coquasitriangular coquasi-Hopf algebra over $\k$ with the dual Chevalley property. Then $(\operatorname{gr}H, \operatorname{gr}(R))$ remains coquasitriangular.
\end{lemma}
\begin{proof}
We only show (\ref{eq:CQT1}) for $\operatorname{gr}(R)$; since (\ref{eq:CQT2}), (\ref{eq:CQT3}) and the fact that $\operatorname{gr}(R)$ is convolution invertible are obvious. Let $\bar{a} \in H_n/H_{n-1}$ and $ \bar{b}\in H_{m}/H_{m-1}$. Then by the assumption that $(H, R)$ is coquasitriangular, we have $$R(a_1, b_1)a_2b_2=b_1a_1 R(a_2, b_2),$$
which implies that
$$
R(a_1, b_1)a_2b_2/H_{m+n-1}=b_1a_1 R(a_2, b_2)/H_{m+n-1}.
$$
This means that $$\operatorname{gr}(R)(\bar{a}_1, \bar{b}_1)\bar{a}_2\bar{b}_2=\bar{b}_1\bar{a}_1 \operatorname{gr}(R)(\bar{a}_2, \bar{b}_2),$$
and thus (\ref{eq:CQT1}) follows.
\end{proof}
As a corollary of Lemma \ref{lem:grHcoquasi}, we have
\begin{corollary}
If the coquasi-Hopf algebra $\k(\mathrm{Q},\mathcal{S})$ constructed in Theorem \ref{thm:generalizedHopfquiver} is coquasitriangular with coquasitriangular structure $R$, then $\operatorname{gr}(R)$ is a concentrated coquasitriangular structure on $\k(\mathrm{Q},\mathcal{S})$.
\end{corollary}

Next, we show that the generalized dual Gabriel's theorem stated in Theorem \ref{thm:Gabriel} can be restricted to the coquasitriangular setting.
\begin{theorem}\label{thm:gabrielcoquasi}
Let $(H, R)$ be a coquasitriangular coquasi-Hopf algebra with the dual Chevalley property. Then there exists a unique generalized Hopf quiver $\mathrm{Q}$ associated with a cosemisimple datum of $H_0$ such that $\k(\mathrm{Q}, \mathcal{S})$ admits a concentrated coquasitriangular coquasi-Hopf algebra structure with the dual Chevalley property and $(\operatorname{gr}H,\operatorname{gr}(R))$ is isomorphic to a large coquasitriangular coquasi-Hopf subalgebra of $\k(\mathrm{Q}, \mathcal{S})$.
\end{theorem}
\begin{proof}
Lemma \ref{lem:grHcoquasi} implies that $(\operatorname{gr}H, \operatorname{gr}(R))$ remains coquasitriangular. By Theorem \ref{thm:generalizedHopfquiver}, $\operatorname{gr}H$ is a large coquasi-Hopf subalgebra of $\k(\mathrm{Q}(H),\mathcal{S})$, where $\mathrm{Q}(H)=(\mathcal{S}, \mathcal{P})$ is the link quiver of $H$. Since $(H, R)$ is a coquasitriangular coquasi-Hopf algebra with the dual Chevalley property, $R_0:=R\mid_{H_0\otimes H_0}$ is a coquasitriangular structure on $H_0$, making $(H_0, R_0)$ a coquasitriangular cosemisimple coquasi-Hopf algebra. We extend $R_0$ trivially to be a function $R^\prime$ on $\k(\mathrm{Q}(H),\mathcal{S}) \otimes \k(\mathrm{Q}(H),\mathcal{S})$ defined by $$R^\prime(a,b):=\left\{
\begin{aligned}
R(a,b), &\text{ if} ~~~ a,b\in H_0; \\
0,\;\;\;\;&\text{ otherwise}.
\end{aligned}
\right.$$
Now we claim that $(\k(\mathrm{Q}(H),\mathcal{S}), R^\prime)$ is coquasitriangular. Indeed, the fact that $R^\prime$ is convolution invertible, together with \eqref{eq:CQT2} and \eqref{eq:CQT3}, follows immediately from the definition of $R^\prime$. It remains to verify (\ref{eq:CQT1}). For any $\mathcal{M}$-paths $a,b$ of length $0$, (\ref{eq:CQT1}) is obvious. Since $\operatorname{gr}H \supseteq \k(\mathrm{Q}(H)_0, \mathcal{S})\oplus \k(\mathrm{Q}(H)_1, \mathcal{S})$, it follows that $R^\prime\mid_{\k(\mathrm{Q}(H)_0, \mathcal{S})\otimes \k(\mathrm{Q}(H)_1, \mathcal{S})\oplus \k(\mathrm{Q}(H)_1, \mathcal{S})\otimes \k(\mathrm{Q}(H)_0, \mathcal{S})}$ satisfies (\ref{eq:CQT1}). This means that for any $\mathcal{M}$-path $c_{i1}^{t(\alpha)}\alpha c_{1j}^{s(\alpha)}$ of length $1$ and any $\mathcal{M}$-path $c_{kl}^r$ of length $0,$ we have
$$
\sum\limits_{p=1}^{t(\alpha)}\sum\limits_{q=1}^{r}R(c_{ip}^{t(\alpha)}, c_{kq}^r)c_{p1}^{t(\alpha)}\alpha c_{1j}^{s(\alpha)} \cdot c_{ql}^r=\sum\limits_{p=1}^{s(\alpha)}\sum\limits_{q=1}^{r} c_{kq}^r\cdot c_{i1}^{t(\alpha)}\alpha c_{1p}^{s(\alpha)} R( c_{pj}^{s(\alpha)},  c_{ql}^r)
$$
and
$$
\sum\limits_{p=1}^{t(\alpha)}\sum\limits_{q=1}^{r} R(c_{kq}^r, c_{ip}^{t(\alpha)})c_{ql}^r\cdot c_{p1}^{t(\alpha)}\alpha c_{1j}^{s(\alpha)}=\sum\limits_{p=1}^{s(\alpha)}\sum\limits_{q=1}^{r} c_{i1}^{t(\alpha)}\alpha c_{1p}^{s(\alpha)}\cdot c_{kq}^r R(c_{ql}^r, c_{pj}^{s(\alpha)}).
$$
By the definition of $R^\prime$, Remark \ref{rm:multi}, and the fact that $R$ is convolution invertible,
we can show that
\begin{eqnarray*}
&&(c_{i1}^{t(\alpha_n)}\alpha_n\alpha_{n-1}\cdots \alpha_1 c_{1j}^{s(\alpha_{1})})\cdot(c_{k1}^{t(\beta_m)}\beta_m\beta_{m-1}\cdots \beta_1 c_{1l}^{s(\beta_{1})})\\
&=&\sum\limits_{p=1}^{t(\beta_m)}\sum\limits_{q=1}^{s(\beta_1)}\sum\limits_{u=1}^{t(\alpha_n)}\sum\limits_{v=1}^{s(\alpha_1)}R(c_{kp}^{t(\beta_m)}, c_{iu}^{t(\alpha_n)}) (c_{p1}^{t(\beta_m)}\beta_m\beta_{m-1}\cdots \beta_1 c_{1q}^{s(\beta_{1})})\\
&&(c_{u1}^{t(\alpha_n)}\alpha_n\alpha_{n-1}\cdots \alpha_1 c_{1v}^{s(\alpha_{1})})R^{-1}(c_{ql}^{s(\beta_1)}, c_{vj}^{s(\alpha_1)}),
\end{eqnarray*}
for all $\mathcal{M}$-paths of the form $c_{i1}^{t(\alpha_n)}\alpha_n\alpha_{n-1}\cdots \alpha_1 c_{1j}^{s(\alpha_{1})}$ with $n\geq 1$ and $c_{k1}^{t(\beta_m)}\beta_m\beta_{m-1}\cdots \beta_1 c_{1l}^{s(\beta_{1})}$ with $m\geq 1$. This completes the proof.
\end{proof}
We now proceed to construct a concentrated coquasitriangular structure on $\k(\mathrm{Q},\mathcal{S})$ under certain assumptions.
\begin{lemma}\label{lem:k(Q,S)coquasitriangular}
Let $(H, R)$ be a cosemisimple coquasitriangular coquasi-Hopf algebra over $\k$, and let $\mathrm{Q}$ be a generalized Hopf quiver associated with a cosemisimple datum of $H$. If $R(a_1,b_1)R(b_2, a_2)=\varepsilon(a)\varepsilon(b)$ for all $a,b\in H$, then $\k(\mathrm{Q},\mathcal{S})$ admits a concentrated coquasitriangular coquasi-Hopf algebra structure.
\end{lemma}
\begin{proof}
Let $M = \k(\mathrm{Q}_1,\mathcal{S})$. It follows from the proof of Theorem \ref{thm:generalizedHopfquiver} that $M$ forms an $H$-Majid bimodule. Then by Proposition \ref{prop:Hopfmod}, we have $M\cong M^{co H}\otimes H$. Since $H$ is coquasitriangular, Example \ref{ex:coquaitriangular} yields that $M^{co H}$ admits a left-left Yetter-Drinfeld module structure given by
$$ h\vartriangleright m:=R(m_{-1}, h)m_0.$$
Next, we use Proposition \ref{prop:Majidmod=YD} to define a new $H$-Majid bimodule structure on $M$, thereby obtaining a new coquasi-Hopf algebra structure on $\k(\mathrm{Q},\mathcal{S})$. For any $n\in M^{co H}$ and $h,l\in H$, define
\begin{eqnarray*}
&l\cdot
(n\otimes h):=\Phi (l_{1}, n_{-3}, h_{1})\Phi
^{-1}(n_{-1}, l_{3},
h_{2})R(n_{-2},l_2) n_0\otimes l_{4}h_{3}, & \\
    &(n\otimes
h)\cdot l:=\Phi ^{-1}(n_{-1}, h_{1}, l_{1})n_{0}\otimes
h_{2}l_{2}, &  \\
&\rho _{M^{co H}\otimes H}^{L}\left( n\otimes h\right) :=n_{-1}h_{1}\otimes
(n_{0}\otimes h_{2}), & \\
&\rho _{M^{co H}\otimes H}^{R}\left( n\otimes h\right) :=(n\otimes h_{1})\otimes
h_{2}.&
\end{eqnarray*}
With the $H$-Majid bimodule structure defined above, the proof of Theorem \ref{thm:generalizedHopfquiver} yields that we can form a coquasi-Hopf algebra. To avoid ambiguity, we denote this coquasi-Hopf algebra by $CoT_H(M)^\prime$. We extend $R$ trivially to be a function $R^\prime$ on $CoT_H(M)^\prime \otimes CoT_H(M)^\prime$ whenever one of the homogeneous elements $a, b$ lies outside $H$. It remains to verify that $R^\prime $ satisfies (\ref{eq:CQT1}). Similar to the proof of the Theorem \ref{thm:gabrielcoquasi}, we only need to show that $R^\prime\mid_{\k(\mathrm{Q}(H)_0, \mathcal{S})\otimes \k(\mathrm{Q}(H)_1, \mathcal{S})\oplus \k(\mathrm{Q}(H)_1, \mathcal{S})\otimes \k(\mathrm{Q}(H)_0, \mathcal{S})}$ satisfies (\ref{eq:CQT1}). Indeed, by (\ref{eq:CQT1}), (\ref{eq:CQT3}) and the condition that $R(a_1,b_1)R(b_2, a_2)=\varepsilon(a)\varepsilon(b)$, we have
\begin{eqnarray*}
&&R(l_1,n_{-1}h_1) l_2\cdot (n_0\otimes h_2)\\
&=&\Phi^{-1}(n_{-7},h_1,l_1)R(l_2,h_2)\Phi(n_{-6},l_3,h_3)R(l_4,n_{-5})\Phi^{-1}(l_5,n_{-4},h_4)\Phi(l_6,n_{-3},h_5)\\
&&\Phi^{-1}(n_{-1},l_8,h_6)R(n_{-2},l_7) n_0\otimes R^{-1}(l_9,h_7)h_8l_{10} R(l_{11},h_9)\\
&=&\Phi^{-1}(n_{-5},h_1,l_1)R(l_2,h_2)\Phi(n_{-4},l_3,h_3)R(l_4,n_{-3})
\Phi^{-1}(n_{-1},l_6,h_4)R(n_{-2},l_5) \\&&n_0\otimes R^{-1}(l_7,h_5)h_6l_{8} R(l_{9},h_7)\\
&=&\Phi^{-1}(n_{-3},h_1,l_1)R(l_2,h_2)\Phi(n_{-2},l_3,h_3)
\Phi^{-1}(n_{-1},l_4,h_4)n_0\otimes R^{-1}(l_5,h_5)h_6l_{6} R(l_{7},h_7)\\
&=&\Phi^{-1}(n_{-1},h_1,l_1)R(l_2,h_2)n_0\otimes R^{-1}(l_3,h_3)h_4l_{4} R(l_{5},h_5)\\
&=&\Phi^{-1}(n_{-1},h_1,l_1)n_0\otimes h_2l_{2} R(l_{3},h_3)\\
&=&(n\otimes h_1) \cdot l_1 R(l_2, h_2).
\end{eqnarray*}
\end{proof}

With the help of the preceding lemma, we can classify the concentrated coquasitriangular structures on $\k(\mathrm{Q},\mathcal{S})$ in the case where $\mathrm{Q}$ is connected.
\begin{proposition}\label{prop:1-1triangular}
Let $H$ be a finite-dimensional cosemisimple coquasi-Hopf algebra over $\k$, and let $\mathrm{Q}$ be a connected generalized Hopf quiver associated with a cosemisimple datum of $H$. Then the set of all concentrated coquasitriangular structures on $\k(\mathrm{Q},\mathcal{S})$ is in bijection with the set of coquasitriangular structures on $H$ that satisfy $R(a_1,b_1)R(b_2, a_2)=\varepsilon(a)\varepsilon(b)$ for all $a,b\in H$.
\end{proposition}
\begin{proof}
According to Lemma \ref{lem:k(Q,S)coquasitriangular}, if $R(a_1,b_1)R(b_2, a_2)=\varepsilon(a)\varepsilon(b)$ for all $a,b\in H$, then we can construct a concentrated coquasitriangular structure on $\k(\mathrm{Q},\mathcal{S})$. Conversely, suppose that $(\k(\mathrm{Q},\mathcal{S}), R)$ is a coquasitriangular coquasi-Hopf algebra with $R$ concentrated. It follows that $(H, R\mid_{H\otimes H})$ is a coquasitriangular coquasi-Hopf algebra. Suppose that $\mathrm{Q}$ is associated with a cosemisimple datum $(H, \{C_i\}_{i\in I}, (n_1, n_2, \cdots), \{\alpha_{ij}^t\}_{i,j,t\in I})$. Let $\mathcal{S} = \{C_{i} \mid i \in \mathrm{Q}_0\}$ be the set of simple subcoalgebras of $H$, and let
$\mathcal{M}=\{\C_{i}\mid i \in \mathrm{Q}_0\},$ where each
$\C_{i}=(c_{jk}^{i})_{r_{i}\times r_i}$ is a basic multiplicative matrix of $C_{i}\in\mathcal{S}.$ For $n_i,n_j\neq0$, there exist two arrows $\alpha:1\rightarrow C_i$ and $\beta: 1\rightarrow C_j$ in $\mathrm{Q}$. Consider two $\mathcal{M}$-paths of length $1$, denoted $c_{k1}^{i}\alpha1$ and $c_{l1}^{j}\beta1$. Here we write $c_{k1}^{i}\alpha1$ as $c_{k1}^{i}\alpha$ for short (and similarly $c_{l1}^{j}\beta1$ as $c_{l1}^{j}\beta$). Then, by Remark \ref{rm:multi} and (\ref{eq:CQT1}), in $\operatorname{CoT}_{H}(\k(\mathrm{Q}_1, \mathcal{S}))$, we have
\begin{eqnarray*}
c_{k1}^i\alpha \cdot c_{l1}^j\beta&=&\sum\limits_{u=1}^{r_i}c_{ku}^i\cdot c_{l1}^j\beta \otimes c_{u1}^i\alpha+\sum\limits_{v=1}^{r_j} c_{k1}^i\alpha \cdot c_{lv}^j \otimes c_{v1}^j\beta\\
&=&\sum\limits_{u=1}^{r_i}c_{ku}^i\cdot c_{l1}^j\beta \otimes c_{u1}^i\alpha+\sum\limits_{m,v=1}^{r_j}\sum\limits_{n=1}^{r_i}R(c_{lm}^j,c_{kn}^i) c_{mv}^j\cdot c_{n1}^i\alpha \otimes c_{v1}^j\beta\\
&=&\sum\limits_{u=1}^{r_i}\sum\limits_{v=1}^{r_j}R(c_{lv}^j,c_{ku}^i) c_{v1}^j\beta \cdot c_{u1}^i\alpha\\
&=&\sum\limits_{u=1}^{r_i}\sum\limits_{v=1}^{r_j}R(c_{lv}^j,c_{ku}^i)(\sum\limits_{m=1}^{r_j}c_{vm}^j\cdot c_{u1}^i\alpha \otimes c_{m1}^j\beta+\sum\limits_{n=1}^{r_i}c_{v1}^j\beta \cdot c_{un}^i\otimes c_{n1}^i\alpha )\\
&=&\sum\limits_{u=1}^{r_i}\sum\limits_{m,v=1}^{r_j}R(c_{lv}^j,c_{ku}^i)c_{vm}^j\cdot c_{u1}^i\alpha \otimes c_{m1}^j\beta+\sum\limits_{n,p,u=1}^{r_i}\sum\limits_{q,v=1}^{r_j}R(c_{lv}^j,c_{ku}^i) R(c_{up}^i, c_{vq}^j)\\&& c_{pn}^i\cdot c_{q1}^j\beta\otimes c_{n1}^i\alpha.
\end{eqnarray*}
It follows that $$c_{kn}^i\cdot c_{l1}^j\beta=\sum\limits_{p,u=1}^{r_i}\sum\limits_{q,v=1}^{r_j}R(c_{lv}^j,c_{ku}^i) R(c_{up}^i, c_{vq}^j)c_{pn}^i\cdot c_{q1}^j\beta$$
for all $1\leq n\leq r_i.$ According to Lemma \ref{Lemma:cap=0} and \cite[Corollary 2.6]{YLL24}, $c_{pn}^i\cdot c_{q1}^j\beta \neq 0.$ It follows that $$\sum\limits_{p,u=1}^{r_i}\sum\limits_{q,v=1}^{r_j}R(c_{lv}^j,c_{ku}^i) R(c_{up}^i, c_{vq}^j)=\delta_{k,p}\delta_{q,l}.$$ Thus we have $R(a_1,b_1)R(b_2, a_2)=\varepsilon(a)\varepsilon(b)$ for all $a,b$ ranging over the elements of simple subcoalgebras in $\mathcal{S}{}^1$. By Proposition \ref{prop:(H1)0} and the fact that $H$ is finite-dimensional, $H$ is generated by $\{c\in C\mid C\in \mathcal{S}{}^1\},$ then using (\ref{eq:CQT2}) and (\ref{eq:CQT3}), we can complete the proof.
\end{proof}

In summary, from the preceding results we obtain the following proposition.
\begin{proposition}\label{prop:braidedtensorcategory}
Let $\mathscr{C}$ be a finite braided integral tensor category over $\k$ with the Chevalley property. Then
$\mathscr{C}$ is tensor equivalent to $\mathscr{M}^H$ for some finite-dimensional coquasitriangular coquasi-Hopf algebra $(H,R)$ with the dual Chevalley property. Moreover, we have an embedding $\mathscr{M}^{\operatorname{gr}(H)}\rightarrow \mathscr{M}^{\k(\mathrm{Q}(H),\mathcal{S})}$ of braided tensor categories, where $\mathrm{Q}(H)$ is the link quiver of $H$, and $\mathscr{M}^{\operatorname{gr}(H)}$ contains all objects of
$\mathscr{M}^{\k(\mathrm{Q}(H),\mathcal{S})}$ of Loewy length $\leq 2$. In particular, if the Ext quiver of $\mathscr{C}$ is connected, then $R(a_1,b_1)R(b_2, a_2)=\varepsilon(a)\varepsilon(b)$ for all $a,b\in H_0$.
\end{proposition}
\begin{proof}
By Lemma \ref{lem:integertensor} and \cite[Proposition 10.1]{BCPV19}, $\mathscr{C}$ is tensor equivalent to $\mathscr{M}^H$ for some finite-dimensional coquasitriangular coquasi-Hopf algebra $H$ with the dual Chevalley property. According to Theorem \ref{thm:gabrielcoquasi}, we have an embedding $\mathscr{M}^{\operatorname{gr}(H)}\rightarrow \mathscr{M}^{\k(\mathrm{Q}(H),\mathcal{S})}$ of braided tensor categories, and $\mathscr{M}^{\operatorname{gr}(H)}$ contains all objects of
$\mathscr{M}^{\k(\mathrm{Q}(H),\mathcal{S})}$ of length $\leq 2$. If the Ext quiver of $\mathscr{C}$ is connected, then by Remark \ref{rm:link=ext}, $H$ is link-indecomposable. If $H = \k1$, the proof is trivial; if $H \neq \k1$, then by Theorem \ref{thm:gabrielcoquasi}, $(\operatorname{gr}H, \operatorname{gr}(R))$ is a large coquasitriangular coquasi-Hopf subalgebra of $\k(\mathrm{Q}(H) ,\mathcal{S})$, where $\mathrm{Q}(H)$ is the link quiver of $H$. A similar argument as in the proof of Proposition \ref{prop:1-1triangular}, we can show that $$R(a_1,b_1)R(b_2, a_2)=\varepsilon(a)\varepsilon(b)$$ for all $a,b\in H_0$.
\end{proof}

\begin{corollary}
Let $\k$ be an algebraically closed field with characteristic $0$, and let $\mathscr{C}$ be a finite braided non-fusion integral tensor categories over $\k$ with the Chevalley property of finite representation type. Then $\mathscr{C}$ is tensor equivalent to $\mathscr{M}^H$ for some finite-dimensional coquasitriangular coquasi-Hopf algebra $H$ with the dual Chevalley property such that $\operatorname{gr}H_{(1)}$ is a generalized Taft algebra presented by generators $g$ and $x$ with relations
$$
g^n = 1,\;\; x^2 = 0,\;\; gx = -xg,
$$
where $n$ is an even integer.
\end{corollary}
\begin{proof}
This follows directly from Theorem \ref{thm:finite} and \cite[Corollary 4.4]{HL12}.
\end{proof}

\section*{Acknowledgement}
The author would like to thank Professors Xiao-Wu Chen, Gongxiang Liu, and Yu Ye for their valuable comments and suggestions. The author is also grateful to Fengchang Li for useful discussions on quantum groups and to Dongdong Hu for helpful discussions on modulated graphs.


\begin{thebibliography}{99}
\setlength{\itemsep}{0em}
\bibitem{AAGM14}N. Andruskiewitsch, I. Angiono, A. Garc{\'{i}}a Iglesias, A. Masuoka, C. Vay, Lifting via cocycle deformation, J. Pure Appl. Algebra 218 (4) (2014) 684-703. https://doi.org/10.1016/j.jpaa.2013.08.008.
\bibitem{AEG01}N. Andruskiewitsch, P. Etingof, S. Gelaki, Triangular Hopf algebras with the Chevalley property, Machigan Math. J. 49 (2) (2001) 277-298. https://doi.org/10.1307/mmj/1008719774.
\bibitem{AS10}N. Andruskiewitsch, H.-J. Schneider, On the classification of finite-dimensional pointed Hopf algebras, Ann. Math. (2) 171 (1) (2010) 375-417. https://doi.org/10.4007/annals.2010.171.375.
\bibitem{Ang10}I.E. Angiono, Basic quasi-Hopf algebras over cyclic groups, Adv. Math. 225 (6) (2010) 3545-3575. https://doi.org/10.1016/j.aim.2010.06.013.
\bibitem{AG19}I. Angiono, A. Garc{\'{i}}a Iglesias, Liftings of Nichols algebras of diagonal type II: all liftings are cocycle deformations,
Selecta Math. (N.S.) 25 (2019) (1) Paper No. 5, 95 pp. https://doi.org/10.1007/s00029-019-0452-4.
\bibitem{AP12}A. Ardizzoni, A. Pavarin, Preantipodes for dual quasi-bialgebras, Israel J. Math. 192 (1) (2012) 281-295. https://doi.org/10.1007/s11856-012-0024-1.
\bibitem{AP13}A. Ardizzoni, A. Pavarin, Bosonization for dual quasi-bialgebras and preantipode, J. Algebra 390 (2013), 126-159. https://doi.org/10.1016/j.jalgebra.2013.05.014.
\bibitem{ASS06}I. Assem, D. Simson, A. Skowro{\'{n}}ski, Elements of the Representation Theory of Associative Algebra Vol.1, Techniques of representation theory, Lond. Math. Soc. Students Texts 65. Cambridge university Press, 2006.
\bibitem{ARS95}M. Auslander, I. Reiten, S. Smal{\o}, Representation Theory of Artin Algebras, Cambridge Studies in Adv. Math. 36. Cambridge University Press, 1995.
\bibitem{Bal09}A. Balan, Yetter-Drinfeld modules and Galois extensions over coquasi-Hopf algebras, Politehn. Univ. Bucharest Sci. Bull. Ser. A Appl. Math. Phys. 71 (3) (2009) 43-60.
\bibitem{BIR09}M. Beattie, M. C. Iovanov, S. Raianu, The antipode of a dual quasi-Hopf algebra with nonzero integrals is bijective,
Algebr. Represent. Theory 12 (2-5) (2009) 251-255. https://doi.org/10.1007/s10468-009-9148-3.
\bibitem{Bul20}D. Bulacu, Quasi-quantum groups obtained from tensor braided Hopf algebras, J. Algebraic Combin. 52 (2020) (3) 405-453. https://doi.org/10.1007/s10801-019-00907-5.
\bibitem{BC03}D. Bulacu, S. Caenepeel, Integrals for (dual) quasi-Hopf algebras. Applications, J. Algebra 266 (2003) 552-583. https://doi.org/10.1016/S0021-8693(03)00175-3.
\bibitem{BCPV19}D. Bulacu, S. Caenepeel, F. Panaite, F. Van Oystaeyen, Quasi-Hopf Algebras: A Categorical Approach, Cambridge University Press, Cambridge, 2019.
\bibitem{Che54}C. Chevalley, Th\'{e}orie des groupes de Lie, Tome III, Hermann, Paris, 1954.
\bibitem{CHZ06}X.-W. Chen, H.-L. Huang, P. Zhang, Dual Gabriel theorem with applications, Sci. China Ser. A 49 (1) (2006) 9-26. https://doi.org/10.1007/s11425-004-5235-4.
\bibitem{CZ07}X.-W. Chen, P. Zhang, Comodules of $U_q(sl_2)$ and modules of $SL_q(2)$ via quiver methods, J. Pure Appl. Algebra 211 (2007) (3) 862-876. https://doi.org/10.1016/j.jpaa.2007.03.010.
\bibitem{CR02}C. Cibils, M. Rosso, Hopf quivers, J. Algebra 254 (2) (2002) 241-251. https://doi.org/10.1016/S0021-8693(02)00080-7.
\bibitem{CL00}F.U. Coelho, S. X. Liu, Generalized path algebras, in: Interactions Between Ring Theory and Representations of Algebras, Murcia, in: Lect. Notes Pure Appl. Math., vol. 210, Marcel Dekker, New York, 2000, pp. 53-66.
\bibitem{DPR90}R. Dijkgraaf, V. Pasquier, P. Roche, Quasi Hopf algebras, group cohomology and orbifold models, Nucl. Phys. B, Proc. Suppl. 18B (1990) 60-72. https://doi.org/10.1016/0920-5632(91)90123-V.
\bibitem{DR74}V. Dlab, C. M. Ringel, Representations of Graphs and Algebras, Carleton Mathematical Lecture Notes, vol. 8, Department of Mathematics, Carleton University, Ottawa, Ont., 1974.
\bibitem{Dri87}V.G. Drinfeld, Quantum groups, in: Proceedings of the International Congress of Mathematicians, vol. 1, 2, Berkeley, Calif., 1986, Amer. Math. Soc., Providence, RI, 1987.
\bibitem{Dri90}V.G. Drinfeld, Quasi-Hopf algebras, Algebra Anal. 1 (6) (1989) 114-148; translation in Leningr. Math. J. (ISSN 0234-0852) 1 (6) (1990) 1419-1457.
\bibitem{Dro79}J. A. Drozd, Tame and wild matrix problems, representation theory, II, in: Proc. Second Internat. Conf., Carleton Univ., Ottawa, Ont., 1979, Lecture Notes in Math., vol. 832, Springer, Berlin, 1980, pp. 242-258.
\bibitem{DNR25}C. Dong, S.-H. Ng, L. Ren, Orbifolds and minimal modular extensions, Adv. Math. 462 (2025), Paper No. 110103, 43 pp. https://doi.org/10.1016/j.aim.2025.110103.
\bibitem{EG03}P. Etingof, S. Gelaki, The classification of finite-dimensional triangular Hopf algebras over an algebraically closed field of characteristic $0$, Mosc. Math. J. 3 (2003) (1) 37-43, 258.
\bibitem{EG04}P. Etingof, S. Gelaki, Finite-dimensional quasi-Hopf algebras with radical of codimension $2$, Math. Res. Lett. 11 (2004) 685-696. https://doi.org/10.4310/MRL.2004.v11.n5.a11.
\bibitem{EG05}P. Etingof, S. Gelaki, On radically graded finite-dimensional quasi-Hopf algebras, Mosc. Math. J. 5 (2) (2005) 371-378. https://doi.org/10.17323/1609-4514-2005-5-2-371-378.
\bibitem{EG06}P. Etingof, S. Gelaki, Liftings of graded quasi-Hopf algebras with radical of prime codimension, J. Pure Appl. Algebra 205 (2) (2006) 310-322. https://doi.org/10.1016/j.jpaa.2005.06.016.
\bibitem{EG21a}P. Etingof, S. Gelaki, Finite symmetric integral tensor categories with the Chevalley property, with an appendix by Kevin Coulembier and Pavel Etingof, Int. Math. Res. Not. (12) (2021) 9083-9121. https://doi.org/10.1093/imrn/rnz093.
\bibitem{EG21b}P. Etingof, S. Gelaki, Finite symmetric tensor categories with the Chevalley property in characteristic $2$, J. Algebra Appl. 20 (1) (2021), Paper No. 2140010. https://doi.org/10.1142/S0219498821400107.
\bibitem{EGNO15}P. Etingof, S. Gelaki, D. Nikshych, V. Ostrik, Tensor Categories, Math. Surveys Monogr., vol. 205, American Mathematical Society, Providence, 2015.
\bibitem{EO04}P. Etingof, V. Ostrik, Finite tensor categories, Mosc. Math. J. 4 (3) (2004) 627-654. https://doi.org/10.17323/1609-4514-2004-4-3-627-654.
\bibitem{FGR22}V. Farsad, A.M. Gainutdinov, I. Runkel, The symplectic fermion ribbon quasi-Hopf algebra and the $SL(2, \mathbb{Z})$-action on its centre, Adv. Math. 400 (2022) Paper No. 108247, 87 pp. https://doi.org/10.1016/j.aim.2022.108247.
\bibitem{Gab72}P. Gabriel, Unzerlegbare Darstellungen I, Manuscr. Math. 6 (1972) 71-103. https://doi.org/10.1007/BF01298413.
\bibitem{GR25}T. Gannon, A. Riesen, Orbifolds of pointed vertex operator algebras I, Adv. Math. 482 (2025), part A, Paper No. 110546, 47 pp. https://doi.org/10.1016/j.aim.2025.110546.
\bibitem{Gei95}C. Geiss, On degenerations of tame and wild algebras, Arch. Math. 64 (1995) 11-16. https://doi.org/10.1007/BF01193544.
\bibitem{GLS17}C. Geiss, B. Leclerc, J. Schr$\ddot{o}$er, Quivers with relations for symmetrizable Cartan matrices I: foundations, Invent. Math. 209 (1) (2017) 61-158. https://doi.org/10.1007/s00222-016-0705-1.
\bibitem{Gel05}S. Gelaki, Basic quasi-Hopf algebras of dimension $n^3$, J. Pure Appl. Algebra 198 (2005) 165-174. https://doi.org/10.1016/j.jpaa.2004.10.003.
\bibitem{GS98}E.L. Green, {\O}. Solberg, Basic Hopf algebras and quantum groups, Math. Z. 229 (1998) 45-76. https://doi.org/10.1007/PL00004650.
\bibitem{Hua09}H.-L. Huang, Quiver approaches to quasi-Hopf algebras, J. Math. Phys. 50 (4) (2009) 043501. https://doi.org/10.1063/1.3103569.
\bibitem{Hua12}H.-L. Huang, From projective representations to quasi-quantum groups, Sci. China Math. 55 (10) (2012) 2067-2080. https://doi.org/10.1007/s11425-012-4437-4.
\bibitem{HL12}H.-L. Huang, G. Liu, On coquasitriangular pointed Majid algebras, Comm. Algebra 40 (2012) (10) 3609-3621. https://doi.org/10.1080/00927872.2011.582059.
\bibitem{HLY11}H.-L. Huang, G. Liu, Y. Ye, Quivers, quasi-quantum groups and finite tensor categories, Comm. Math. Phys. 303 (3) (2011) 595-612. https://doi.org/10.1007/s00220-011-1229-6.
\bibitem{HLY15}H.-L. Huang, G. Liu, Y. Ye, Graded elementary quasi-Hopf algebras of tame representation type, Israel J. Math. 209 (1) (2015) 157-186. https://doi.org/10.1007/s11856-015-1214-4.
\bibitem{HLYY20}H.-L. Huang, G. Liu, Y. Yang, Y. Ye, Finite quasi-quantum groups of diagonal type, J. Reine Angew. Math. 759 (2020) 201-243. https://doi.org/10.1515/crelle-2017-0058.
\bibitem{HLYY21}H.-L. Huang, G. Liu, Y. Yang, Y. Ye, Finite quasi-quantum groups of rank two, Trans. Amer. Math. Soc. Ser. B 8 (2021) 635-678. https://doi.org/10.1090/btran/79.
\bibitem{HLYY26}H.-L. Huang, G. Liu, Y. Yang, Y. Ye, On the classification of finite quasi-quantum groups over abelian groups, Adv. Math. 486 (2026), Paper No. 110740. https://doi.org/10.1016/j.aim.2025.110740.
\bibitem{HYZ18}H.-L. Huang, Y. Yang, Y. Zhang, On nondiagonal finite quasi-quantum groups over finite abelian groups, Selecta Math. (N.S.) 24 (2018) (5) 4197-4221. https://doi.org/10.1007/s00029-018-0420-4.
\bibitem{HQW26}Y.-M. Huang, Z.-K. Mi, T.-C. Qi, Q.-S. Wu, Chevalley property and discriminant ideals of Cayley-Hamilton Hopf algebras, Int. Math. Res. Not. IMRN 2026 (11) Paper No. rnag111. https://doi.org/10.1093/imrn/rnag111.
\bibitem{HQWZ26}Y.-M. Huang, T.-C. Qi, Q.-S. Wu, R.-P. Zhu, Chevalley property of module-finite Hopf algebras and discriminant ideals, arXiv:2604.15986.
\bibitem{Jos94}A. Joseph, On the prime and primitive spectra of the algebra of functions on a quantum group, J. Algebra 169 (1994) 441-511. https://doi.org/10.1006/jabr.1994.1294.
\bibitem{KR06}L. Krop, D.E. Radford, Finite-dimensional Hopf algebras of rank one in characteristic zero, J. Algebra 302 (2006) 214-230. https://doi.org/10.1016/j.jalgebra.2006.03.031.
\bibitem{Lar71}R.G. Larson, Characters of Hopf algebras, J. Algebra 17 (1971) 352-368. https://doi.org/10.1016/0021-8693(71)90018-4.
\bibitem{Li12}F. Li, Modulation and natural valued quiver of an algebra, Pac. J. Math. 256 (1) (2012) 105-128. https://doi.org/10.2140/pjm.2012.256.105.
\bibitem{LL12}F. Li, Z. Lin, Approach to Artinian algebras via natural quivers, Trans. Amer. Math. Soc. 364 (2012) (3) 1395-1411. https://doi.org/10.1090/S0002-9947-2011-05410-3.
\bibitem{LL08}F. Li, G. Liu, Generalized path coalgebras and a generalized dual Gabriel theorem, Acta Math. Sinica (Chinese Ser.) 51 (5) (2008) 853-862.
\bibitem{Li22}K. Li, The link-indecomposable components of Hopf algebras and their products, J. Algebra 593 (2022) 235-273. https://doi.org/10.1016/j.jalgebra.2021.11.016.
\bibitem{LZ19}K. Li, S. Zhu, On the exponent of finite-dimensional non-cosemisimple Hopf algebras, Comm. Algebra 47 (11) (2019) 4476-4495. https://doi.org/10.1080/00927872.2018.1539176.
\bibitem{LN14}G. Liu, S.-H. Ng, On Total Frobenius-Schur Indicators, in: Recent Advances in Representation Theory, Quantum Groups, Algebraic Geometry, and Related Topics, Contemp. Math., 623, American Mathematical Society, Providence, RI, 2014 pp. 193-213.
\bibitem{LS07}C. Lomp, A. Sant'Ana, Chain and distributive coalgebras, J. Pure Appl. Algebra 211 (3) (2007) 581-595. https://doi.org/10.1016/j.jpaa.2007.02.009.
\bibitem{Maj92}S. Majid, Tannaka-Krein theorems for quasi-Hopf algebras and other results, Contemp. Math. 134 (1992) 219-232.
in: Deformation theory and quantum groups with applications to mathematical physics,
Contemp. Math., 134,
American Mathematical Society, Providence, RI, 1992, 219-232.
\bibitem{Maj95}S. Majid, Foundations of Quantum Group Theory, Cambridge University Press, 1995.
\bibitem{MT19}S. Majid, W-Q. Tao, Generalised noncommutative geometry on finite groups and Hopf quivers, J. Noncommut. Geom. 13 (2019) 1055-1116. https://doi.org/10.4171/jncg/345.
\bibitem{MN01}G. Mason, S.-H. Ng, Group cohomology and gauge equivalence of some twisted quantum doubles, Trans. Am. Math. Soc. 353 (9) (2001) 3465-3509. https://doi.org/10.1090/S0002-9947-01-02771-4.
\bibitem{Mon93}S. Montgomery, Hopf Algebras and Their Actions on Rings, CBMS Regional Conference Series in Mathematics, 82. Published for the Conference Board of the Mathematical Sciences, Washington, DC; by the American Mathematical Society, Providence, RI, 1993.
\bibitem{Mon95}S. Montgomery, Indecomposable coalgebras, simple comodules and pointed Hopf algebras, Proc. Amer. Math. Soc. 123 (8) (1995) 2343-2351. https://doi.org/10.2307/2161257.
\bibitem{NS08}S.H. Ng, P. Schauenburg, Central invariants and higher indicators for semisimple quasi-Hopf algebras, Trans. Am. Math. Soc. 360 (4) (2008) 1839-1860. https://doi.org/10.1090/S0002-9947-07-04276-6.
\bibitem{Nic78}W. Nichols, Bialgebras of type I, Comm. Alg 6 (1978) 1521-1552. https://doi.org/10.1080/00927877808822306.
\bibitem{Rad77}D.E. Radford, Operators on Hopf algebras, Amer. J. Math. 99 (1) (1977) 139-158. https://doi.org/10.2307/2374012.
\bibitem{Rad12}D.E. Radford, Hopf Algebras, Series on Knots and Everything, vol. 49, World Scientific Publishing Co. Pte. Ltd., Hackensack, NJ, 2012.
\bibitem{RT91}N. Reshetikhin, V. Turaev, Invariants of 3-manifolds via link polynomials and quantum groups, Invent. Math. 103 (1991) 547-597. https://doi.org/10.1007/BF01239527.
\bibitem{Rin75}C.M. Ringel, The representation type of local algebras, in: Representations of Algebras, Proc. Internat. Conference, Ottawa 1974, in: Springer LNM, vol. 488, 1975, pp. 282-305.
\bibitem{SS93}S. Shnider, S. Sternberg, Quantum Groups: From Coalgebras to Drinfel'd Algebras, Grad. Texts Math. Phys., II, International Press, Cambridge, MA, 1993.
\bibitem{SS07}D. Simson and A. Skowro{\'{n}}ski, Elements of the representation theory of associative algebras, Vol. 3, London Math. Soc. Stud. Texts, 72, Cambridge University Press, Cambridge, 2007.
\bibitem{Ros98}M. Rosso, Quantum groups and quantum shuffles, Invent. Math. 133 (1998) (2) 399-416. https://doi.org/10.1007/s002220050249.
\bibitem{VZ04}F. Van Oystaeyen, P. Zhang, Quiver Hopf algebras, J. Algebra 280 (2004) 577-589. https://doi.org/10.1016/j.jalgebra.2004.06.008.
\bibitem{YLL24}J. Yu, K. Li, G. Liu, Hopf algebras with the dual Chevalley property of finite corepresentation type, Algebr. Represent. Theory 27 (5) (2024) 1821-1867. https://doi.org/10.1007/s10468-024-10284-8.
\bibitem{YL24}J. Yu, G. Liu, Coradically graded Hopf algebras of tame corepresentation type, arXiv:2407.21389.
\bibitem{YL26}J. Yu, G. Liu, Hopf algebras with the dual Chevalley property of discrete corepresentation type, J. Algebra 688 (2026) 803-843. https://doi.org/10.1016/j.jalgebra.2025.10.011.
\bibitem{ZGH21}Y. Zheng, Y. Gao, N. Hu, Finite-dimensional Hopf algebras over the Hopf algebra $H_{b:1}$ of Kashina, J. Algebra 567 (2021) 613-659. https://doi.org/10.1016/j.jalgebra.2020.09.035.
\end{thebibliography}
\end{document}